\documentclass{amsart}

\title[SEMIARTINIAN
VON NEUMANN REGULAR ALGEBRAS]
{Representation of artinian partially ordered sets over semiartinian Von Neumann regular algebras}

\date{March 10, 2009}

\usepackage{amsmath,amsthm,amscd,amsxtra,amssymb}

\newtheorem{theo}{Theorem}[section]
\newtheorem{pro}[theo]{Proposition}
\newtheorem{prodef}[theo]{Proposition and Definition}
\newtheorem{lem}[theo]{Lemma}
\newtheorem{cor}[theo]{Corollary}

\newtheorem{ex}[theo]{Example}

\theoremstyle{remark}
\newtheorem{re}[theo]{Remark}

\newtheorem{notation}[theo]{Notation}
\newtheorem{remnotation}[theo]{Remark and Notation}
\newtheorem{notations}[theo]{Notations}
\newtheorem{definition}[theo]{Definition}

\numberwithin{equation}{section}

\newcommand{\putl}[5]{\put(#1,#2){\line(#3,#4){#5}}}

\newcommand{\putbb}[3]{\put(#1,#2){\makebox(0,0)[b]{#3}}}

\newcommand{\CA}{\mathcal{A}}
\newcommand{\CB}{\mathcal{B}}
\newcommand{\CC}{\mathcal{C}}

\newcommand{\CF}{\mathcal{F}}
\newcommand{\CG}{\mathcal{G}}
\newcommand{\CH}{\mathcal{H}}

\newcommand{\CK}{\mathcal{K}}
\newcommand{\CL}{\mathcal{L}}
\newcommand{\CM}{\mathcal{M}}
\newcommand{\CN}{\mathcal{N}}

\newcommand{\CP}{\mathcal{P}}
\newcommand{\CQ}{\mathcal{Q}}

\newcommand{\CV}{\mathcal{V}}

\newcommand{\CY}{\mathcal{Y}}

\newcommand{\BC}{\mathbb C}
\newcommand{\BD}{\mathbb D}
\newcommand{\BE}{\mathbb E}
\newcommand{\BF}{\mathbb F}

\newcommand{\BL}{\mathbb L}
\newcommand{\BM}{\mathbb M}
\newcommand{\BN}{\mathbb N}

\newcommand{\BP}{\mathbb P}

\newcommand{\BR}{\mathbb R}

\newcommand{\ZA}{\mathbf{A}}
\newcommand{\ZB}{\mathbf{B}}

\newcommand{\ZK}{\mathbf{K}}

\newcommand{\ZM}{\mathbf{M}}

\newcommand{\ZO}{\mathbf{O}}
\newcommand{\ZP}{\mathbf{P}}

\newcommand{\ZS}{\mathbf{S}}

\newcommand{\ZU}{\mathbf{U}}

\newcommand{\za}{\mathbf{a}}
\newcommand{\zb}{\mathbf{b}}
\newcommand{\zc}{\mathbf{c}}
\newcommand{\zd}{\mathbf{d}}
\newcommand{\ze}{\mathbf{e}}
\newcommand{\zf}{\mathbf{f}}
\newcommand{\zg}{\mathbf{g}}
\newcommand{\zh}{\mathbf{h}}
\newcommand{\zi}{\mathbf{i}}

\newcommand{\zm}{\mathbf{m}}

\newcommand{\zo}{\mathbf{o}}
\newcommand{\zp}{\mathbf{p}}

\newcommand{\zr}{\mathbf{r}}
\newcommand{\zs}{\mathbf{s}}

\newcommand{\zu}{\mathbf{u}}

\newcommand{\bet}{{\beth\,}}

\renewcommand{\a}{\alpha}
\renewcommand{\b}{\beta}
\newcommand{\g}{\gamma}
\renewcommand{\d}{\delta}

\renewcommand{\l}{\lambda}
\renewcommand{\r}{\rho}
\newcommand{\s}{\sigma}
\renewcommand{\t}{\tau}

\newcommand{\f}{\varphi}

\renewcommand{\L}{\varLambda}
\renewcommand{\Xi}{\varXi}
\renewcommand{\Pi}{\varPi}
\renewcommand{\Sigma}{\varSigma}

\newcommand{\F}{\varPhi}
\renewcommand{\P}{\varPsi}

\newcommand{\al}{\boldsymbol{\aleph}}

\newcommand{\nin}{\not\in} 
\newcommand{\sbs}{\subset} 
\newcommand{\psbs}{\subsetneqq} 
\newcommand{\nsbs}{\not\subset} 
\newcommand{\sps}{\supset} 
\newcommand{\psps}{\supsetneqq} 
\newcommand{\vu}{\emptyset} 
\newcommand{\is}{\simeq} 
\newcommand{\subis}{\lesssim} 
\renewcommand{\le}{\leqslant} 
\renewcommand{\ge}{\geqslant} 

\newcommand{\defug}{\,\,\colon\!\!=}

\DeclareMathOperator{\Ker}{Ker} 
\DeclareMathOperator{\Imm}{Im} 
\DeclareMathOperator{\Hom}{Hom} \DeclareMathOperator{\End}{End}
 \DeclareMathOperator{\BIE}{BiEnd}
\DeclareMathOperator{\Supp}{Supp} 
 \DeclareMathOperator{\Soc}{Soc}
\DeclareMathOperator{\Tr}{Tr} \DeclareMathOperator{\LL}{L}

\newcommand{\CFM}{\BC\BF\BM}

\newcommand{\FR}{\BF\BR}
\newcommand{\FM}{\BF\BM}

\renewcommand{\H}[3]{\Hom_{#3}\left(#1, #2\right)}

\newcommand{\Ord}{{\ZO\zr\zd}}

\newcommand{\Simp}{{\ZS\zi\zm\zp}}
\newcommand{\Prosimp}{{\ZP\zr\zo\zs\zi\zm\zp}}
\newcommand{\Prim}{{\ZP\zr\zi\zm}}

\newcommand{\map}[3]{#1\colon #2\to#3}

\newcommand{\lmap}[3]{#1\colon#2\longrightarrow#3}

\renewcommand{\max}[1]{#1^{\bigstar}}

\begin{document}

\author{Giuseppe Baccella}

\address{Dipartimento di Matematica Pura ed Applicata\\
Universit\`a di L'Aquila.\\ L'Aquila, 67100 Italy}

\email{baccella@univaq.it}



\dedicatory{Dedicated to the memory of Adalberto Orsatti \emph{(Il Maestro)} and of Dimitri Tyukavkin}

\begin{abstract}
If $R$ is a semiartinian Von Neumann regular ring, then the set
    $\Prim_{R}$ of
    primitive ideals of $R$, ordered by inclusion, is an artinian poset in
    which all maximal chains have a greatest element. Moreover, if $\Prim_{R}$
    has no infinite antichains, then the lattice $\BL_{2}(R)$ of all ideals of $R$ is anti-isomorphic to the lattice of all upper subsets of $\Prim_{R}$. Since the assignment $U\mapsto r_R(U)$ defines a bijection from any set $\Simp_R$ of representatives of simple right $R$-modules to $\Prim_{R}$, a natural partial order is induced in $\Simp_R$, under which the maximal elements are precisely those simple right $R$-modules which are finite dimensional over the respective endomorphism division rings; these are always $R$-injective. Given any artinian poset $I$ with at least two elements and having a finite cofinal subset, a lower subset $I'\sbs I$ and a field $D$, we present a construction which produces a semiartinian and unit-regular $D$-algebra $D_I$ having the following features: (a) $\Simp_{D_I}$ is order isomorphic to $I$; (b) the assignment $H\mapsto\Simp_{D_I/H}$ realizes an anti-isomorphism from the lattice $\BL_{2}(D_I)$ to the lattice of all upper subsets of $\Simp_{D_I}$; (c) a non-maximal element of $\Simp_{D_I}$ is injective if and only if it corresponds to an element of $I'$, thus $D_I$ is a right $V$-ring if and only if $I' = I$; (d) $D_I$ is a right \emph{and} left $V$-ring if and only if $I$ is an antichain; (e) if $I$ has finite dual Krull length, then $D_I$ is (right and left) hereditary; (f) if $I$ is at most countable and $I' = \vu$, then $D_I$ is a countably dimensional $D$-algebra.

\end{abstract}

\maketitle

\setcounter{section}{-1}
\section{Introduction}

For a given right semiartinian ring $R$ we introduced in \cite{Bac:15} what we called the {\em natural preorder\/} ``$\,\preccurlyeq\,$'' in the class of all simple right
modules over $R$. The idea was to define,
for every simple module $U_R$, a particular $U$-peak ideal $I(U)$
(in the sense that the right socle of $R/I(U)$ is essential,
projective and $U$-homogeneous) and, given another simple module
$V_R$, to declare that $U\preccurlyeq V$ in case $I(U)\sbs I(V)$.
It turn out that the natural preorder is a Morita invariant;
moreover $U\is V$ if and only if both $U\preccurlyeq V$ and
$V\preccurlyeq U$, so that ``$\,\preccurlyeq\,$'' induces the
\emph{natural partial order} in any set $\Simp_R$ of representatives
of simple right $R$-modules. With respect to the natural partial
order, $\Simp_R$ is an artinian poset in which every maximal chain
has a maximum.

It is worth to observe that, since the class of right semiartinian
rings is closed by factor rings, for every $U\in\Simp_{R}$ the
primitive ring $R/r_{R}(U)$ has nonzero socle. This implies that $U$
is the unique (up to an isomorphism) simple and faithful right
$R/r_{R}(U)$-module and the assignment $U\mapsto r_{R}(U)$ defines a
bijection from $\Simp_{R}$ to the set $\Prim_{R}$ of (right)
primitive ideals of $R$. In view of this fact it would appear quite
natural to declare $U\preccurlyeq V$ in case $r_{R}(U)\sbs
r_{R}(V)$; moreover we must recall that Camillo and Fuller already
remarked in \cite{CamilloFuller:2} that the set $\Prim_{R}$, ordered
by inclusion, is always artinian when $R$ is right semiartinian. The
point is that in many interesting cases $\Prim_{R}$ is just the set
of all maximal (two-sided) ideals and the above partial order
becomes the trivial one, giving thus no insight into the structure
of $R$; for example, this is the case when if $R$ is left perfect,
in particular when $R$ is right artinian. The situation changes
dramatically when $R$ is a regular ring; in this case it turns out
that $I(U) = r_{R}(U)$ for all $U\in\Simp_{R}$, therefore
$U\preccurlyeq V$ if and only if $r_{R}(U)\sbs r_{R}(V)$; moreover
$U$ is a maximal element of $\Simp_{R}$ if and only if $U$ is
finite dimensional as a vector space over the division ring $\End(U_{R})$ and, if it is the case, then $U_{R}$ is injective. By the regularity, every ideal of $R$ is the intersection of all right primitive ideals containing it, therefore the order structure of $\Simp_{R}$, or equivalently of $\Prim_{R}$, strictly affects the order structure of the lattice  $\BL_{2}(R)$ of all ideals of $R$; for instance, if $\Simp_{R}$ has no infinite antichains, then $\BL_{2}(R)$ is anti-isomorphic to the lattice of all upper subsets of $\Simp_{R}$, therefore $\BL_{2}(R)$ is artinian (see \cite[Corollary 4.8 and Theorem 4.5]{Bac:15}).

The main subject of the present work is to investigate which artinian partially ordered sets can be realized as $\Simp_{R}$, or equivalently as $\Prim_{R}$, for some semiartinian and regular ring $R$. This problem appears as a special instance of the more general problem of determining those complete lattices which are isomorphic to $\BL_{2}(R)$ for some regular ring $R$. A rather general answer to this problem was given by Bergman in \cite{Bergman:001}, by showing that if $L$ is a complete and distributive lattice, which has a compact greatest element and each element of which is the supremum of compact join-irreducible elements, then there exists a unital, regular and locally matricial algebra $R$ over any given field $F$ such that $\BL_{2}(R)$ is isomorphic to $L$. Our main result is that if $I$ is an artinian poset and $D$ is a division ring, then there exists a unit-regular and semiartinian ring $D_I$, having $D$ as subring, such that $\Simp_{D_I}$ is isomorphic to $I$ provided $I$ has a finite cofinal subset, otherwise $\Simp_{D_I}$ is isomorphic to the poset obtained from $I$ by adding a suitable maximal element.

As it was proved in \cite{BacCiamp:14}, if $R$ is a semiartinian and unit-regular ring, then the abelian group $\ZK_{0}(R)$ is free of rank $|\Simp_{R}|$; however, in the same paper the order structure of $\ZK_{0}(R)$ was investigated only in the case in which $R$ satisfies the so called restricted comparability axiom (see in Section 4 below for the definition). In a forthcoming paper we will resume that investigation, precisely we will characterize those partially ordered abelian groups which can be realized as $\ZK_{0}(R)$ for some semiartinian and unit-regular ring $R$. In particular we will see that if $I$ is an artinian poset having a finite cofinal subset, then $\ZK_{0}(D_I)$ is isomorphic to the free abelian group generated by $I$, together with the submonoid generated by the elements $i-j$ for $j<i$ in $I$ as the positive cone.

Now $\ZK_{0}(R)$ is the Grothendieck group of the abelian monoid $\CV(R)$ of isomorphism classes of finitely generated projective right $R$-modules. When $R$ is a regular ring, then $\CV(R)$ enjoys some fundamental and well known properties. The inverse problem of deciding wether, given an abelian monoid $M$ having the same properties, there exists a regular ring $R$ such that $\CV(R)$ is isomorphic to $M$, is known as the \emph{Realization Problem for Von Neumann Regular Rings\/}; Ara recently wrote a nice survey on it (see \cite{Ara:011}). Only after the present work was complete we became aware of the recent important works by Ara and Brustenga \cite{AraBrustenga:010} and by Ara \cite{Ara:012} on this problem. Precisely, given a field $K$, in the first one a regular $K$-algebra $Q(E)$ is associated to a column-finite quiver $E$, via the Leavitt path algebra $L(E)$ of $E$ (see \cite{AbramsArandaPino:010}), in such a way that $\CV(Q(E))$ is isomorphic to $\CV(L(E))$; in the second one a regular $K$-algebra $Q(\BP)$ is functorially associated to each finite poset $\BP$, in such a way that $\CV(Q(\BP))$ is the abelian monoid generated by $\BP$ with the only relations given by $p = p+q$ if and only if $q<p$ in $\BP$. To some extent our present research parallels the above works. Our construction of the ring $D_I$ is far from being functorial on $I$, exactly as the map which assigns to a set $X$ the ring $\CFM_{X}(D)$ of all column-finite $X\times X$-matrices with entries in a given ring $D$ is not a functor on $X$. Nonetheless, if $I$ and $J$ are isomorphic artinian posets, then the rings $D_I$ and $D_J$ turn out to be isomorphic and we can list several nice ring and module theoretical features of $D_I$.  It would be interesting to find relationships, if any, between the algebra $Q(\BP)$ of Ara and our algebra $D_{\BP}$ when $\BP$ is a finite poset.

Our work is divided into nine sections. In section 1 we examine some basic features of artinian posets needed when dealing with semiartinian and regular rings. In particular, given an artinian poset $I$, for every ordinal $\a$ we consider the {\em
$(\a+1)$-th layer\/} $I^{\bullet}_{\a+1}$ of $I$, namely: $I^{\bullet}_{1}$ is the set of all minimal elements of $I$ and, for every ordinal $\a>1$ one defines recursively $I^{\bullet}_{\a+1}$ as the set of all minimal elements of the set $I\setminus \left(\bigcup_{\b<\a}I^{\bullet}_{\b+1}\right)$. The set of all layers is a partition of $I$ and we define the \emph{canonical length function} $\map{\l_I}I{\Ord}$ as the function which assigns to every $i\in I$ the (unique) successor ordinal $\l_I(i)$ such that $i$ belongs to the $\l_I(i)$-th layer of $I$ (recall that a \emph{length function} on an artinian poset $I$ is any strictly increasing map from $I$ to the well ordered class $\Ord$ of all ordinals).

The second, third and fourth sections are devoted to the study of the natural partial order of $\Simp_{R}$, when $R$ is a semiartinian and regular ring. We recall that if $R$ is any right semiartinian ring and $M$ is any right $R$-module, then we define the ordinal $h(M) = \min\{\,\a\mid M\cdot\Soc_\a(R_R) = M\}$; if $M$ is finitely generated, then $h(M)$ is a successor ordinal if. If $U_R$ is simple and $h(U) = \a+1$, then $U_{R/\Soc_\a(R_R)}$ is projective and $\a$ is the largest ordinal such that $\H{U}{R/\Soc_\a(R_R)}{R}\ne 0$ (see \cite[Theorem
1.3]{BacDic:1}) while, if $R$ is regular, $\a$ is the \emph{unique} ordinal with this property. Now $h$ defines a length function on the artinian poset $\Simp_R$ and if $\l$ denotes the canonical length function on $\Simp_R$, then it turns out that $\l(U)\le h(U)$ for every $U\in\Simp_R$. We concentrate our attention on two special classes of semiartinian and regular rings. A ring $R$ belongs to the first one if and only if the two length functions $\l$ and $h$ coincide, while it belongs to the second one if and only if the assignment $H\mapsto\Simp_{R/H}$ realizes an anti-isomorphism from the lattice $\BL_{2}(R)$ to the lattice of all upper subsets of $\Simp_{R}$. We say that $R$ is \emph{well behaved} in the first case and \emph{very well behaved} in the second. Of course, if $R$ is very well behaved then $R$ is well behaved and, in addition, $\Simp_R$ has only finitely many maximal elements. We illustrate with examples that these latter two conditions are actually independent and, together, do not imply that $R$ is very well behaved. Next, for any semiartinian and regular ring $R$, we pass to establish which properties of the poset $\Simp_R$ are connected with the various comparability axioms on $R$.

We start with section 5 our construction of semiartinian unit-regular rings. The scenario of the whole drama is the ring $Q = \CFM_{X}(D)$ of all column-finite matrices with entries in a given ring $D$, where $X$ is a suitable \underline{transfinite ordinal}, together with the ideal $\BF\BR_X(D)$ of all matrices with only finitely many nonzero rows. It is well known that if $R$ is any ring and $\map{\f,\psi}QR$ are two ring isomorphisms, then $\f(\BF\BR_X(D)) = \psi(\BF\BR_X(D))$; let's say that the elements of this latter ideal are the \emph{finite-ranked} elements of $R$. Thus, the first main step is to associate to every ordinal $\xi\le X$ a family $(Q_\a)_{\a\le\xi}$ of unital subrings of $Q$ having the following features: (a) if $\a<\xi$, then $Q_\a$ is isomorphic to $Q$, (b) by denoting with $F_\a$ the ideal of $Q_\a$ of all finite-ranked elements when $\a<\xi$, then $Q_{\b}\cap F_{\a} = 0$ whenever
$\a<\b\le\xi$. Actually, we already gave in \cite{Bac:12} a construction which aimed to the same objective. Unfortunately the proof of Proposition 4.2 in that paper contains a gap. Filling that gap - if ever possible, would have required a considerable work and the result would have not been suitable for our present purposes either. Thus we decided to completely reorganize the construction by using a totally different approach, in which we rely mainly on ordinal arithmetic. With the new construction we have at disposal a total control of the parametrization of the entries of the matrices we deal with, as it is needed in order to accomplish the subsequent main construction.

With section 6 artinian posets enter the scene. First, we define a \emph{polarized (artinian) poset} as an ordered pair $(I,I')$, where $I$ is an artinian poset and $I'$ is a lower subset of $I$. Starting from a polarized artinian poset $(I,I')$, a ring $D$ and an appropriately sized transfinite ordinal $X$, to each element $i\in I$ we associate a (not necessarily unital) subring $H_i$ of $Q = \CFM_X(D)$, in such a way that $\CH = \{H_i\mid i\in I\}$ is an independent set of $(D,D)$-submodules of $Q$ with the following features: (a) if $i$ is a maximal element of $I$, then $H_i$ is isomorphic to $D$; (b) if $i$ is not maximal and belongs to $I'$ (resp. to $I\setminus I'$), then $H_i$ is isomorphic to $\FR_X(D)$ (resp. to the left ideal $\FM_X(D)$ of $Q$ whose elements are all matrices with only finitely many nonzero entries); moreover $H_i H_j = 0$ if and only if $i,j$ are not comparable, while both $H_i H_j$ and $H_j H_i$ are nonzero and
are contained in $H_i$ if $i\le j$. This enables us to consider the (not necessary unital) subring $H_I = \bigoplus_{i\in I}H_i$ and the unital subring $D_I = H_I + \mathbf{1}_{Q}D$ of $Q$ and we show that $H_I = D_I$ if and only if $I$ has a finite cofinal subset. The study of this subring, together with the strict relationship between upper subsets of $I$ and ideals of $D_I$, is the subject of sections 7 and 8.

In section 9, finally, we take $D$ as a division ring and show that, given a polarized artinian poset $(I,I')$, the ring $D_I$ has the following features: (a) $D_I$ is a unit-regular and semiartinian ring, which is also (right and left) hereditary in case $I$ has finite dual Krull length; (b) there is a map $i\mapsto U_I$ from $I$ to $\Simp_{D_I}$ which is an order isomorphism in case $I$ has a finite cofinal subset, otherwise $\Simp_{D_I}$ contains $D_I/H_I$ as an additional maximal element; (c) a non-maximal element $U_i$ of $\Simp_{D_I}$ is injective if and only if $i\in I'$, thus $D_I$ is a right $V$-ring if and only if $I' = I$; (d) $D_I$ is a right \emph{and} left $V$-ring if and only if $I$ is an antichain; (e) if $I$ has a finite cofinal subset, then the assignment $H\mapsto\Simp_{D_I/H}$ realizes an anti-isomorphism from the lattice $\BL_{2}(D_I)$ to the lattice of all upper subsets of $\Simp_{D_I}$; (f) if $I$ is at most countable and $I' = \vu$, then $D_I$ is countably dimensional over $D$.

We conclude this introduction with a few remarks about terminology and notations. In several instances we deal with rings without multiplicative
identity and subrings which are not unital subrings but, often, they
have their own multiplicative identities. However, in order to avoid
any ambiguity, if not otherwise stated the word ``ring'' means
``ring with multiplicative identity'', while ``subring'' means
``unital subring'' (that is, if we state that a ring $R$ is a
subring of some ring $T$ we mean that $R$ shares the same
multiplicative identity of $T$) and all ring homomorphisms preserve
multiplicative identity.

Given a ring $R$, we shall denote with $\Simp_R$ a chosen
irredundant set of re\-presentatives of all simple right
$R$-modules, while $\Prosimp_R$ will be the subset of $\Simp_R$ of
representatives of all simple and projective right $R$-modules. If
any given set $\ZU$ of simple right $R$-modules turns out to be an
irredundant set of representatives of all simple right $R$ modules,
we shall summarize this fact by writing $\ZU = \Simp_R$.

Recall that the {\it Loewy chain\/} (or {\it lower Loewy chain\/},
according to some authors) of a right $R$-module $M$ is the
non-decreasing chain of submodules $(\Soc_{\a}(M))_{\a\ge 0}$, parametrized over the ordinals,
defined by the following rules: set $\Soc_0(M) = 0$ and,
recursively, define $\Soc_{\a+1}(M)$ in such a way that
$\Soc_{\a+1}(M)/\Soc_\a(M) = \Soc(M/\Soc_\a(M))$ (we denote by $\Soc(M)$ the
socle of $M$) for each ordinal $\a$ and $\Soc_\a(M) =
\bigcup_{\b<\a} \Soc_\b(M)$ if $\a$ is a limit ordinal. The module
$M/\Soc_{\a}(M)$ is called the {\it $\a$-th Loewy factor\/} of $M$,
the first ordinal $\xi$ such that $\Soc_{\xi}(M) = \Soc_{\xi+1}(M)$
is called the {\it Loewy length\/} of $M$ (denoted by $\LL(M)$) and
one says that $M$ is {\it semiartinian\/} or a {\it Loewy module\/}
if $\Soc_{\xi}(M) = M$. The ring $R$ is {\it right semiartinian\/}
if the module $R_R$ is semiartinian or, equivalently, if every
non-zero right $R$-module contains a simple submodule; if it is the
case, then each $\Soc_\a(R_R)$ is an ideal.

If $R$ is a right semiartinian ring and $M$ is some right
$R$-module, we define the ordinal $h(M) = \min\{\,\a\mid
M\cdot\Soc_\a(R_R) = M\}$; clearly, when $M$ is finitely generated $h(M)$ is not a limit ordinal if. If $U_R$ is simple and $h(U) = \a+1$,
then $U_{R/\Soc_\a(R_R)}$ is projective and $\a$ is the largest
ordinal such that $\H{U}{R/\Soc_\a(R_R)}{R}\ne 0$ (see \cite[Theorem
1.3]{BacDic:1}) while, if $R$ is regular, $\a$ is the \emph{unique}
ordinal with this property.

\section{Some preliminary notions on artinian partially ordered sets.}

Let $I$ be a given partially ordered set. For every subset $J\sbs I$
define
\begin{gather*}
\{\le J\} \defug \{k\in I\mid k\le j\text{ for al }j\in J\}, \\
 \{J\le\} \defug \{k\in I\mid j\le k\text{ for al }j\in J\};
\end{gather*}
thus the notations $\{\le i\}$ and $\{i\le\}$ have an obvious
meaning for every element $i\in I$.  A {\em lower subset\/} (resp.
{\em upper subset\/}) of a poset $I$ is a subset $J\sbs I$ such that
if $j\in J$, then $\{\le j\}\sbs J$ (resp. $\{j\le\}\sbs J$). In
particular $\{\le K\}$ and $\{K\le\}$ are respectively the smallest lower
subset and the smallest upper subset of $I$ which contain a given subset $K\sbs I$. We denote by
$\Uparrow\!\!I$ (resp. $\Downarrow\!\!I$) the set of all upper
subsets (resp. lower subsets) of $I$; both $\Uparrow\!\!I$ and
$\Downarrow\!\!I$ are complete and distributive lattices and the map
$J\mapsto I\setminus J$ is an anti-isomorphism from
$\Uparrow\!\!I$ to $\Downarrow\!\!I$ .

For every subset $J$ of $I$ let us denote by $J_1$ the set of all
minimal elements of $J$. We recall that the {\em dual classical
Krull filtration\/} of a poset $I$ is the ascending chain
$(I_{\a})_{0\le\a}$ of subsets of $I$ defined as follows (see
\cite{Albu:10}):
\begin{gather*}
I_{0} \defug \emptyset, \\
I_{\a+1} \defug I_{\a}\cup\left(I\setminus I_{\a}\right)_1
\quad \text{ for all }\a, \\
I_{\a} \defug \bigcup_{\b<\a}I_{\b}\quad \text{ if $\a$ is a limit
ordinal }.
\end{gather*}
Clearly there exists a smallest ordinal $\xi$ such that $I_{\xi+1} =
I_{\xi}$; moreover $I$ is artinian (i.~e. it satisfies the DCC or,
equivalently, every chain of $I$ is well ordered) if and only if $I
= I_{\xi}$ and, in this case, the ordinal $\xi$ is called the {\em
dual classical Krull dimension\/} of $I$. In the sequel we shall
make use of the following further notations: for every ordinal $\a$
\[
I^{\bullet\bullet}_{\a} \defug I\setminus I_{\a}, \quad
I^{\bullet}_{\a+1} \defug \left(I\setminus I_{\a}\right)_1.
\]
Observe that $I_\a$ is a lower subset, while $I^{\bullet\bullet}_\a$
is an upper subset. If $I$ is artinian, then it is clear that
$\left\{I^{\bullet}_{\a+1}\mid \a<\xi\right\}$ is a partition of $I$
and
\[
I_{\a} = \bigcup_{\b<\a}I^{\bullet}_{\b+1}
\]
for all $\a<\xi$; we will often call $I^{\bullet}_{\a+1}$ the {\em
$(\a+1)$-th layer\/} of $I$. A similar notion is introduced in E. Harzheim book \cite{Harzheim:1} where, given a finite poset $I$, for every positive integer $n$ the \emph{$n$-level} $L_n$ of $I$ is defined exactly as our $n$-th layer. Of course, every subset of an artinian
poset is artinian with respect to the induced partial order.

\begin{pro}\label{lowersbs}
If $J$ is a lower subset of an artinian poset $I$, then
\[
J_{\a} = J\cap I_{\a}\quad\text{for every ordinal $\a$}.
\]
\end{pro}
\begin{proof}
It is obvious that $J_{0} = J\cap I_{0}$. Take any ordinal $\a>0$
and assume inductively that $J_{\b} = J\cap I_{\b}$ for every
$\b<\a$. If $\a$ is a limit ordinal, then one immediately infers
that $J_{\a} = J\cap I_{\a}$. Suppose that $\a = \b+1$ for some
$\b$. From the inductive hypothesis it follows easily that $J\setminus
J_{\b} = J\cap(I\setminus I_{\b})$ and then $(J\setminus
J_{\b})_1\sbs J\cap(I\setminus I_{\b})_1$, because $J$ is a lower subset of $I$. As a result we obtain:
\[
J_{\b+1} = J_{\b}\cup(J\setminus J_{\b})_1 = (J\cap
I_{\b})\cup[J\cap(I\setminus I_{\b})_1] =
J\cap[I_{\b}\cup(I\setminus I_{\b})_1] = J\cap I_{\b+1},
\]
as wanted.
\end{proof}

If $I$ is any partially ordered {\em class\/}, Gary Brookfield
defines in \cite{Brook:1} the {\em minimum length function\/}
$\map{\l_I}I{\Ord}$ as follows: for every $i\in I$
\[
\l_I(i) \defug \min\{\l(j)\mid \l\text{ is a length function on
}I\},
\]
where a {\em length function\/} on $I$ is any strictly increasing
function from $I$ to $\Ord$. If it exists, $\l_I$ itself is a length
function. It turns out that if $I$ is an artinian po{\em set\/},
then $I$ admits a length function and $\l_I$ can be defined
recursively as follows: for every $i\in I$
\begin{equation}\label{length}
\l_I(i) = \begin{cases} 0\text{ if $i$ is a minimal element of $I$}, \\
                                               \sup\{\l_I(j)+1\mid
j<i\}\text{ otherwise }
            \end{cases}
\end{equation}
(see \cite[Proposition 3.9]{Brook:1}).

\begin{pro}\label{lengthlayer}
Let $I$ be an artinian poset, whose dual classical Krull dimension
is $\xi$, and let $i\in I$. Then for every ordinal $\a$ we have
\begin{equation}\label{lengthlayereq}
\l_I(i) = \a\text{ if and only if } i\in I^{\bullet}_{\a+1}.
\end{equation}
Consequently $\l_I(I)$ is an ordinal and
\begin{equation}\label{eq dclKdim}
    \l_I(I) = \xi.
\end{equation}
\end{pro}
\begin{proof}
Denoting by $P(\a)$ the statement \eqref{lengthlayereq}, we see that
$P(0)$ is obviously true. Given an ordinal $\a>0$, assume that $P(\b)$
is true whenever $\b<\a$. Suppose that $\l_I(i) = \a$ and let $\g$
be the unique ordinal such that $i\in I^{\bullet}_{\g+1}$.
Necessarily $\a\le\g$ by the inductive hypothesis, therefore $i\in
I^{\bullet\bullet}_{\a}$. Assume that $i\not\in I^{\bullet}_{\a+1}$,
that is, $i$ is not a minimal element of $I^{\bullet\bullet}_{\a}$.
Then there would be some $j\in I^{\bullet\bullet}_{\a}$ such that
$j<i$ and hence $\l_I(j)<\l_I(i) = \a$. Using the inductive
hypothesis we would get $j\in I^{\bullet}_{\l_I(j)+1}\cap
I^{\bullet\bullet}_{\a} = \emptyset$: a contradiction. Hence $i\in
I^{\bullet}_{\a+1}$. Conversely, suppose that the latter condition
holds. If $j<i$, then $j\in I_\a$ and so there is some $\b<\a$ such
that $j\in I^{\bullet}_{\b+1}$. As a consequence it follows from the
inductive hypothesis that $\l_I(j) = \b<\a$ and hence $\l_I(j)+1\le
\a$, showing that $\l_I(i)\le\a$. It is not the case that
$\l_I(i)<\a$ otherwise, again from the inductive hypothesis we would
get $i\in I^{\bullet}_{\a+1}\cap I^{\bullet}_{\l_I(i)+1} =
\emptyset$. We conclude that $\l_I(i) = \a$, namely that $P(\a)$ holds and this shows the first part of the proposition.

Now, by the assumption we have that
\[
I = \bigcup_{\a<\xi}I^{\bullet}_{\a+1}.
\]
If $\a<\xi$, namely $\a\in\xi$, then $I^{\bullet}_{\a+1}$ is not
empty and, by the above, $\l_I(i) = \a$ for every $i\in
I^{\bullet}_{\a+1}$. Thus $\xi\sbs\l_I(I)$. Conversely, if
$\a\in\l_I(I)$, that is $\a = \l_I(i)$ for some $i\in I$, again by
the above we must have that $i\in I^{\bullet}_{\a+1}$, therefore
$\a<\xi$. As a result $\l_I(I)\sbs\xi$, which proves the equality
\eqref{eq dclKdim}.
\end{proof}

\begin{notation}\label{notation lambda}
If $I$ is an artinian poset and $i\in I$, we shall denote by $\l(i)$
the ordinal $\l_I(i)+1$; in other words $\l(i)$ will be the unique
successor ordinal such that $i\in I^{\bullet}_{\l(i)}$. Of course,
the map $i\mapsto \l(i)$ defines a particular length function
$\map{\l}I{\Ord}$; we call it the \emph{canonical length function},
since it suits our future purposes better than the minimal length
function.
\end{notation}

According to \cite[Corollary 3.5]{Brook:1}, if $I$ is an artinian
poset and $i\in I$, then $\l_I(j) =  \l_{\{\le i\}}(j)$ for every
$j\in \{\le i\}$; thus, combining Proposition \ref{lengthlayer} with
\cite[Proposition 3.6]{Brook:1} we obtain the following result.

\begin{cor}\label{lengthlayerr}
Let $I$ be an artinian poset and let $i\in I$. Then for every
ordinal $\a<\l(i)$ there exists an element $j\in I^{\bullet}_{\a+1}$
such that $j<i$.
\end{cor}

\begin{re}\label{rem wellord}
It is quite natural that sometimes authors working in different areas of Mathematics concentrate the interest on the same object. As often happened, and continues to happen, according tho the specific area in which it is considered that object gets different names. This is the case for posets which satisfy DCC: ring theorists call them artinian posets, as we do, while set theorists, in particular those who investigate partially ordered sets, call them \emph{well-founded posets} and call \emph{well quasi-ordered}, or \emph{partially well-ordered} the well-founded posets without infinite antichains (see \cite{Harzheim:1}, for instance).
\end{re}

\section{The natural partial order of $\Simp_{R}$ when $R$ is a
semiartinian regular ring.}

We recall that if $R$ is any regular ring, then $\Soc(R_{R}) =
\Soc({_{R}R})$; in fact, every minimal right (or left) ideal of $R$
is generated by an idempotent and, for every idempotent $e\in R$, we
have that $eR_{R}$ is simple if and only if $_{R}Re$ is simple. By a
straightforward induction it follows also that $\Soc_{\a}(R_{R}) =
\Soc_{\a}({_{R}R})$ for every ordinal $\a$. Thus, when dealing with
a regular ring $R$, there will be no ambiguity in using the
notations $\Soc(R)$ and $\Soc_{\a}(R)$.

Throughout this section, if not otherwise specified, $R$ will be a
given semiartinian and regular ring with Loewy length $\xi$ and we
set
\[
L_\a \defug \Soc_{\a}(R)
\]
for every ordinal $\a$. As a first consequence it is easy to infer
that if $x\in R$, then
\[
h(xR) = \min\{\,\a\leqslant\xi\mid x\in L_{\a}\}
\]
 and we write $h(x)$ for $h(xR)$ (see the introduction for the definition of the length function $h$). As we anticipated in the introduction, by the regularity of $R$ the
correspondence $U\mapsto r_R(U)$ defines an order isomorphism from
the set $\Simp_{R}$, equipped with the natural partial order
introduced in \cite{Bac:15}, and the set $\Prim_{R}$ of all
primitive ideals ordered by inclusion; this latter is then an
artinian poset in which every maximal chain has a maximum. The
hypothesis of regularity of $R$ allows to give the following
characterizations of the natural partial order of $\Simp_{R}$, in
addition to those we gave in \cite[Theorem 2.2]{Bac:15}.

\begin{theo}\label{orderreg}
Let $R$ be a semiartinian and regular ring and let $U,V$ be simple
right $R$-modules such that $\a+1 = h(U)<\b+1 = h(V)$. Then the
following conditions are equivalent:
\begin{enumerate}
\item $U\prec V$.
\item If $y\in L_{\b+1}\setminus L_{\b}$ is such that $(yR+L_{\b+1})/L_{\b}\is V$, then
\[
U^n\lesssim (yR+L_{\a+1})/L_{\a} \quad\text{for every positive
integer }n.
\]
\item If $y\in L_{\b+1}\setminus L_{\b}$ is such that
$(yR+L_{\b+1})/L_{\b}\is V$, then for every positive integer $n$
there is $x\in L_{\a+1}\setminus L_{\a}$ such that
$(xR+L_{\a+1})/L_{\a}\is U$ and
\[
(xR)^n\lesssim yR
\]
(here $(xR)^n$ stands for the direct sum of $n$ copies of $xR$).
\item If $y\in L_{\b+1}\setminus L_{\b}$ is such that
$(yR+L_{\b+1})/L_{\b}\is V$, then there is $x\in L_{\a+1}\setminus
L_{\a}$ such that $(xR+L_{\a+1})/L_{\a}\is U$ and
\[
RxR\sbs RyR.
\]
\end{enumerate}
The above elements $x,y$ can be chosen to be idempotent.
\end{theo}
\begin{proof}
(1)$\Rightarrow$(2) Assume (1), take $y\in L_{\b+1}\setminus L_{\b}$
with $(yR+L_{\b+1})/L_{\b}\is V$, set $A = (yR+L_{\a+1})/L_{\a}$ and note that $U\lesssim A$ by \cite[Theorem
2.2]{Bac:15}. Let $B = A\cap\Tr_{R/L_\a}(U)$ and suppose that $B$ is
finitely generated. Then $A = B\oplus C$ for some $C\le A$ and there
is an idempotent $z\in R$ such that $C = (zR+L_{\a+1})/L_{\a}$.
Observing that $B = BL_\b$, we infer that $V\is A/AL_\b \is
(B/BL_\b)\oplus(C/CL_\b) = (C/CL_\b) \is (zR+L_{\b+1})/L_{\b}$; on
the other hand $\H{U}{(zR+L_{\a+1})/L_{\a}}{R}= 0$ by the above and
this leads to a contradiction with (1), taking \cite[Theorem
2.2]{Bac:15} into account. Thus (2) holds.

(2)$\Rightarrow$(3) Suppose (2), let $y$ be as in (3) and choose
$u\in L_{\a+1}\setminus L_{\a}$ with
$uR/uL_{\a}\is(uR+L_{\a+1})/L_{\a}\is U$. As $U\lesssim
(yR+L_{\a+1})/L_{\a}\is yR/yL_{\a}$, it follows from
\cite[Proposition 2.20]{Good:3} that $uR = xR\oplus x'R$, where
$xR\lesssim yR$ and $x'R\sbs L_\a$. Thus $xR/xL_{\a}\is U$ and (3)
is true with $n = 1$. Let $n\ge 1$ and assume that $(uR)^n\lesssim
yR$ for some $u\in R$ such that $uR/uL_{\a}\is U$. Then $yR =
y'R\oplus y''R$, where $y'R\is (uR)^n\sbs L_{\b}$ and therefore
$y''R/y''L_{\b}\is V$. By the inductive hypothesis $uR/uL_{\a}\is
U\lesssim y''R/y''L_{\a}$ and, using again \cite[Proposition
2.20]{Good:3}we infer that $uR = xR\oplus x'R$, where $xR\lesssim
y''R$ and $x'R\sbs L_\a$. As a result we get
\[
(xR)^{n+1} = (xR)^n\oplus xR \lesssim (uR)^n\oplus xR \lesssim
y'R\oplus y''R = yR
\]
and $xR/xL_{\a}\is U$.

(3)$\Rightarrow$(4) Let $x,y$ be as in (3), with $n = 1$. Then the
regularity of $R$ implies that there is an $R$-module epimorphism
from $yR$ to $xR$, hence $RxR = \Tr_R(xR) \sbs \Tr_R(yR) = RyR$.

(4)$\Rightarrow$(1) Take $x,y$ as in (4) and observe that,
consequently,
\begin{gather*}
\Tr_{R/L_a}(U) = (R/L_a)(x+L_\a)(R/L_a) \\
\sbs (R/L_a)(y+L_\a)(R/L_a) = \Tr_{R/L_a}(yR+L_{\a+1})/L_{\a}).
\end{gather*}
Inasmuch as $U$ is $R/L_a$-projective, we infer that
$\H{U}{(yR+L_{\a+1})/L_{\a}}{R}$$\ne 0$ and hence $U\prec V$ by
\cite[Proposition 2.1]{Bac:15}.
\end{proof}

It is a trivial observation that the map $U\mapsto h(U)$ defines a
length function on $\Simp_R$ and
\begin{equation}\label{eq length}
    \l(U) \le h(U)\quad \text{for all }U\in\Simp_R,
\end{equation}
where $U\mapsto\l(U)$ is the \emph{canonical} length function on
$\Simp_R$ (Notation \ref{notation lambda}). The inequality in \eqref{eq length} may be strict. For
example, given any successor ordinal $\xi$, there exists a regular
and semiartinian ring $R$ with Loewy length $\xi$ and having all
primitive factors artinian (see \cite{DungSmith:1} and
\cite{Bac:12}); in this case every element of $\Simp_R$ is
maximal (see \cite[Corollary 4.8]{Bac:15}), that is $\Simp_R$ is an
antichain and, if $\xi>1$, for every ordinal $\a$ such that
$1\le\a<\xi$ there are infinitely many $U\in\Simp_R$ with $h(U) =
\a$, while $\l(U) = 1$ \emph{for every} $U\in\Simp_R$. Thus, while simple projective modules are always minimal elements of $\Simp_R$, there may exist non-projective minimal simple modules (see also Example \ref{ex indec2compconn}, Section \ref{sect conncomp}).

We now investigate when the inequality \eqref{eq length} is actually an equality. First a general result.

\begin{prodef}\label{pro lambda=h}
If $R$ is a regular and semiartinian ring $R$ then, with the above notations, the following conditions are equivalent:
\begin{enumerate}
   \item $\l(U) = h(U)$ for every $U\in\Simp_R$.
  \item For every ordinal $\a$ the following equality holds:
  \begin{equation}\label{eq lambda=h}
  \left(\Simp_{R}\right)_{\a} = \{U\in\Simp_{R}\mid UL_{\a} = U\}.
  \end{equation}
  \item For every ordinal $\a$ the following equality holds:
  \begin{equation}\label{eq lambda=hh}
  \Prosimp_{R/L_{\a}} = \left(\Simp_{R/L_{\a}}\right)_{1}.
  \end{equation}
\end{enumerate}
If any, and hence all of the above conditions holds, then we say that $R$ is \emph{well behaved}.
\end{prodef}
\begin{proof}
First, observe that for every ordinal $\a$ we have the equalities
\begin{gather*}
\left(\Simp_{R}\right)_{a} = \{U\in\Simp_{R}\mid \l(U)\le\a\}, \\
\{U\in\Simp_{R}\mid UL_{\a} = U\} = \{U\in\Simp_{R}\mid h(U)\le\a\},
\end{gather*}
the first of which follows from Proposition \ref{lengthlayer}. Thus, since $\l(U)\le h(U)$ for every $U\in\Simp_{R}$, the equivalence between (1) and (2) easily follows.

(2)$\Rightarrow$(3) Given any ordinal $\a$, it follows from (2) that
\begin{align*}
\Prosimp_{R/L_{\a}} &= \{U\in\Simp_{R}\mid h(U) = \a+1\} \\
&= \{U\in\Simp_{R}\mid \l(U) = \a+1\} \\
&= \left(\Simp_{R}\right)^{\bullet}_{\a+1}\\
&= \left(\Simp_{R}\setminus\left(\Simp_{R}\right)_{\a}\right)_{1} \\
&= \left(\Simp_{R/L_{\a}}\right)_{1},
\end{align*}
hence the equality \eqref{eq lambda=hh} holds.

(3)$\Rightarrow$(2) Assume (1), let $P(\a)$ denote the property
\[
\left(\Simp_{R}\right)_{\a} = \{U\in\Simp_{R}\mid UL_{\a} = U\}
\]
and let us prove that $P(\a)$ is true for every ordinal $\a$. If $\a = 0$, then $P(\a)$ is merely the equality $\emptyset = \emptyset$. Given an ordinal $\a>0$, assume that $P(\b)$ holds for every $\b<\a$. If $\a$ is a limit ordinal, then $P(\a)$ follows from the fact that $L_{\a} = \bigcup_{\b<\a}L_{\b}$. Assume that $\a = \b+1$ for some $\b$. Then we have
\begin{align*}
\left(\Simp_{R}\right)_{\b+1} &= \left(\Simp_{R}\right)_{\b}\cup\left(\Simp_{R}
                                 \setminus\left(\Simp_{R}\right)_{\b}\right)_{1} \\
&= \{U\in\Simp_{R}\mid UL_{\b} = U\}\cup\left(\Simp_{R/L_{\b}}\right)_{1} \\
&= \{U\in\Simp_{R}\mid UL_{\b} = U\}\cup\Prosimp_{R/L_{\b}} \\
&= \{U\in\Simp_{R}\mid UL_{\b+1} = U\},
\end{align*}
proving the equality \eqref{eq lambda=h}.
 \end{proof}

There are at least three interesting situations in which a regular and semiartinian ring $R$ turns out to be well behaved. The first two are certain finiteness conditions on the poset $\Simp_R$ and are the subject of the remaining part of the present section; the third one is connected with a comparability condition and will be discussed in Section \ref{sect comparability}.

     \begin{lem}\label{lem orderreg}
      Let $R$ be a regular and semiartinian ring and let $U,V\in\Simp_R$ be such that $h(U)< h(V)$. If $U,V$ are not comparable and $x$ is an idempotent such that $(xR+L_{h(V)-1})/L_{h(V)-1} \is V$, then there is a nonnegative integer $n$ and two orthogonal idempotents $y,z$ such that $xR = yR\oplus zR$ and satisfying the following conditions:
\begin{gather}
(yR+L_{h(V)-1})/L_{h(V)-1} \is V, \label{eq cororderreg1} \\
 (zR+L_{h(U)-1})/L_{h(U)-1} \is U^n, \label{eq cororderreg2}\\
 U\not\subis(yR+L_{h(U)-1})/L_{h(U)-1}. \label{eq cororderreg3}
 \end{gather}
 \end{lem}
\begin{proof}
According to Theorem \ref{orderreg} we may consider the largest
nonnegative integer $n$ such that $U^n$ imbeds, necessarily as a
direct summand, into $(xR+L_{h(U)-1})/L_{h(U)-1}$. By the regularity
of $R$, there are orthogonal idempotents $y,z$ such that $xR =
yR\oplus zR$ and \eqref{eq cororderreg2} holds. Now \eqref{eq
cororderreg1} follows since $z\in L_{h(U)-1}$  and the choice of $n$
guarantees that \eqref{eq cororderreg3} holds as well.
\end{proof}

    \begin{pro}\label{pro orderreg}
      Let $R$ be a regular and semiartinian ring. If the layer $\left(\Simp_R\right)_{\a}^{\bullet}$ is finite for every $\a$, then $R$ is well behaved.
 \end{pro}
\begin{proof}
Given an ordinal $\a$, let $P(\a)$ denote the following property:
\[
    \text{if $U\in\Simp_R$ and $h(U) = \a+1$, then $\l(U) = h(U)$.}
\]
Our task is to show that $P(\a)$ is true for every $\a$. Without the
regularity hypothesis on $R$, we already know that $P(0)$ holds.
Thus, given an ordinal $\a>0$, suppose inductively that $P(\b)$
holds whenever $\b<\a$, take $U\in\Simp_R$ such that $h(U) = \a+1$
and assume that $\l(U) = \b+1<\a+1$. It follows from the inductive
assumption that $\Prosimp_{R/L_\b}$ is contained in the $\b+1$-th
layer $(\Simp_R)^{\bullet}_{\b+1}$ to which $U$ belongs,
consequently $V\not\preccurlyeq U$ for all $V\in\Prosimp_{R/L_\b}$.
On the other hand, by the hypothesis $(\Simp_R)^{\bullet}_{\b+1}$ is
finite, therefore, by applying finite induction and Lemma \ref{lem
orderreg}, we obtain that there exists an idempotent $y\in R$ such
that $(yR+L_{\a})/L_{\a} \is U$ and $V\not\subis(yR+L_{\b})/L_{\b}$
for every $V\in\Prosimp_{R/L_\b}$. Inasmuch as the trace of
$\Prosimp_{R/L_\b}$ in $R/L_\b$ equals the socle and, whence, is
essential, we infer that $(yR+L_{\b})/L_{\b} = 0$ and so $y\in
L_{\b}$. This contradicts the assumption that $h(U) = \a+1>\b$.
We conclude that $\l(U) = \a+1$ and this shows that $P(\a)$ is true.
\end{proof}

There is a natural way to link the ideal structure of a regular and semiartinian ring $R$
and the order structure of $\Simp_R$. Indeed, observe that if $H$ is
an ideal of $R$, then $\Simp_{R/H}$ is an upper subset of
$\Simp_{R}$, so that we may consider the decreasing map
\[
\lmap{\F}{\BL_2(R)}{\Uparrow\!\!\Simp_{R}}
\]
defined by $\F(H) = \Simp_{R/H}$. This map is injective and has as a left inverse the map
\[
\lmap{\P}{\Uparrow\!\!\Simp_{R}}{\BL_2(R)}
\]
defined by $\P(\ZS) = \bigcap\{r_R(U)\mid U\in\ZS\}$. In fact, it is clear that $\F(H)\sps\F(K)$ whenever $H\sbs K$. Inasmuch as $R$ is
regular, then every ideal of $R$ is the intersection of all
primitive ideals containing it. Thus, given $H\in\BL_2(R)$, we have
\begin{gather*}
\P(\F(H)) = \P\left(\Simp_{R/H}\right) = \bigcap\left\{r_R(U)\mid U\in\Simp_{R/H}\right\} \\
= \bigcap\left\{r_R(U)\mid \text{$U\in\Simp_{R}$ and $UH = 0$}\right\} = H.
\end{gather*}

\begin{definition}\label{def vwbehaved}
We say that $R$ is \emph{very well behaved} in case $\F$ and $\P$ are anti-isomorphisms each inverse of the other.
\end{definition}

If $\Simp_{R}$ has no infinite antichains, then $R$ is very well behaved; this is a particular case of \cite[Theorem 4.5]{Bac:15}, because all ideals of a regular ring are left pure. In general, as we are going to see the property of being $R$ very well behaved entails a finiteness condition on the poset $\Simp_{R}$. We can see it at first in case $R$ has all primitive factor rings artinian.

\begin{pro}\label{cor uppersimp}
If $R$ is a semiartinian and regular ring with all right primitive factor rings artinian, then $R$ is very well behaved if and only if $R$ is semisimple.
\end{pro}
\begin{proof}
Assume that $R$ is not semisimple. Then $\Simp_{R}$ is an infinite antichain and $\Prosimp_{R}$ is a \emph{proper} upper subset of $\Simp_{R}$. Since we have
\[
\Psi(\Prosimp_{R}) = 0 = \Psi(\Simp_{R}),
\]
it follows that $\F$ is not an anti-isomorphism.
\end{proof}

\begin{pro}\label{uppersimpp}
Let $R$ be a regular and semiartinian ring. If $R$ is very well behaved, hen the following properties hold:
\begin{enumerate}
  \item Every factor ring of $R$ is very well behaved.
  \item $R$ is well behaved and $\Simp_{R}$ has finitely many maximal elements.
\end{enumerate}
\end{pro}
\begin{proof}
(1) Let $H$ be an ideal of $R$, let $\ZS$ be an upper subset of $\Simp_{R/H}$ and let $U\in\Simp_{R}$, $V\in\ZS$ be such that $V\preccurlyeq U$. Then $UH = 0$, therefore $U\in\Simp_{R/H}$ and hence $U\in\ZS$. We infer that $\Uparrow\!\!\Simp_{R/H} \sbs \Uparrow\!\!\Simp_{R}$. As a consequence, the restrictions of $\F$ and $\Psi$ to $\{H\sbs\}$ and $\Uparrow\!\!\Simp_{R/H}$, respectively, define an anti-isomorphism from $\{H\sbs\}$ to $\Uparrow\!\!\Simp_{R/H}$. As a result, the assignment $K/H\mapsto \Simp_{R/K}$ is an anti-isomorphism from $\BL_2(R/H)$ to $\Uparrow\!\!\Simp_{R/H}$, meaning that $R/H$ is very well behaved.

(2) We claim that if $R$ is very well behaved, then $\Prosimp_{R} = \left(\Simp_{R}\right)_{1}$. Indeed, by setting $\ZS = \{\Prosimp_{R}\preccurlyeq\}$, we have that $\Psi(\ZS) = 0$ and consequently
\[
\ZS = \F(\Psi(\ZS)) = \F(0) = \Simp_{R}.
\]
As a result, for every $U\in\Simp_{R}$ we have that $\l(U) = 1$ implies $h(U) = 1$, proving our claim. Given any ordinal $\a$, according to (1) the ring $R/L_{\a}$ is very well behaved and we infer from the above that $\Prosimp_{R/L_{\a}} = \left(\Simp_{R/L_{\a}}\right)_{1}$. Thus $R$ is well behaved.

Finally, if $\ZM$ is the set of all maximal elements of $\Simp_{R}$ and $H = \Psi(\ZM)$, then $R/H$ is very well behaved and has all primitive factor rings artinian. Thus $R/H$ is semisimple by Proposition \ref{cor uppersimp} and so $\ZM$ is finite.
\end{proof}

The two conditions in property (2) of the previous proposition are actually independent and, even together, do not imply that $R$ is very well behaved; moreover a factor ring of a well behaved ring need not be well behaved. We illustrate all this with the next example, which also shows that the reverse of Proposition \ref{pro orderreg} does not hold; however we have to wait till the last section (see Theorem \ref{theo semiartunitreg}, properties (7) and (8)) in order to se that there exists a regular and semiartinian ring $R$ such that each layer $\left(\Simp_R\right)_{\a}$ is finite for every $\a$, but $\Simp_{R}$ has infinitely many maximal elements, so that $R$ is well behaved but is not \emph{very} well behaved.

\begin{ex}\label{ex indec2compconn}
There exists an indecomposable, semiartinian and regular ring $R$, together with a semiartinian and regular subring $S$, satisfying the following conditions:
\begin{enumerate}
  \item Both $\Simp_R$ and $\Simp_S$ have finitely many maximal elements.
  \item $R$ is well behaved but not very well behaved.
  \item $S$ is not well behaved and is isomorphic to a factor ring of $R$.
\end{enumerate}
\end{ex}
\begin{proof}
Given a field $F$, let us consider the ring $Q = \CFM_{\BN^{*}}(F)$ and remember that $\Soc(Q) = \FR_{\BN^{*}}(F)$ consists of all matrices with finitely many nonzero rows. By setting $X =\{2,4,6,\ldots\}$ and $Y=\{1,3,5,\ldots\}$, for the purposes of the example we want to build it is convenient to view the elements of $Q$ as blocked matrices of the form $\bigl(
\begin{smallmatrix}
A&|&B \\ \hline \\ C&|&D
\end{smallmatrix}
\bigr)$, where $A\in\CFM_{X}(F)$, $B\in\CFM_{X,Y}(F)$, $C\in\CFM_{Y,X}(F)$ and $D\in\CFM_{Y}(F)$. Set $T = \prod_{n>0}T_n$, where $T_n = Q$ for all $n>0$, and let us consider the idempotents $v,w\in T$ defined by
\[
v_n = \left(%
\begin{array}{cccc|cccc}
  1 & 0 & 0 & \ldots & 0 & 0 & 0 & \ldots   \\
  0 & 1 & 0 & \ldots & 0 & 0 & 0 & \ldots   \\
  0 & 0 & 1 & \ldots & 0 & 0 & 0 & \ldots   \\
\vdots & \vdots & \vdots & \ddots & \vdots & \vdots & \vdots &\ddots \\
 \hline
  0 & 0 & 0 & \ldots & 1 & 0 & 0 & \ldots   \\
  0 & 0 & 0 & \ldots & 0 & 0 & 0 & \ldots \\
  0 & 0 & 0 & \ldots & 0 & 0 & 0 & \ldots \\
\vdots & \vdots & \vdots & \ddots & \vdots & \vdots & \vdots &\ddots \\
\end{array}%
\right)\!,\,\,
w_n = \left(%
\begin{array}{cccc|cccc}
  0 & 0 & 0 & \ldots & 0 & 0 & 0 & \ldots   \\
  0 & 0 & 0 & \ldots & 0 & 0 & 0 & \ldots   \\
  0 & 0 & 0 & \ldots & 0 & 0 & 0 & \ldots   \\
\vdots & \vdots & \vdots & \ddots & \vdots & \vdots & \vdots &\ddots \\
 \hline
  0 & 0 & 0 & \ldots & 0 & 0 & 0 & \ldots   \\
  0 & 0 & 0 & \ldots & 0 & 1 & 0 & \ldots \\
  0 & 0 & 0 & \ldots & 0 & 0 & 1 & \ldots \\
\vdots & \vdots & \vdots & \ddots & \vdots & \vdots & \vdots &\ddots \\
\end{array}%
\right)
\]
for all $n>0$; note that $v,w$ are orthogonal and $v+w=1_T$. Given $A\in\CFM_{X}(F)$, let us define the element $x_{A}\in T$ by setting $(x_{A})_{n} = \bigl(
\begin{smallmatrix}
A&|&0 \\ \hline \\ 0&|&0
\end{smallmatrix}
\bigr)$ for all $n>0$ and set
\[
K \defug \{x_{A}\mid A\in\CFM_{X}(F)\}.
\]
Next, for every $n>0$ let $L_{n}$ be the subset of $T$ of those elements $x$ such that $x_{m} = 0$ if $m\ne n$ and $x_{n} = \bigl(
\begin{smallmatrix}
0&|&0 \\ \hline \\ 0&|&D
\end{smallmatrix}
\bigr)$ for some $D\in\CFM_{Y}(F)$. Now it is immediate to check that $vF$, $wF$, $K$ and $L \defug \bigoplus_{n>o}L_{n}$ are independent $F$-subspaces of $T$ and
\[
R \defug vF\oplus wF\oplus K\oplus L
\]
is a regular subring of $T$. It can be seen easily that $K, L_1, L_2,\ldots$ are minimal ideals of $R$ which are the traces of pairwise non isomorphic simple projective right $R$-modules $U_0, U_1, U_2,\ldots$ respectively; moreover
\[
\Soc(R) = K\oplus L \quad\text{and}\quad R/\Soc(R)\is F\times F,
\]
therefore $R$ is semiartinian with Loewy length 2. Easy computations show that
\[
vR + \Soc(R) = vF\oplus\Soc(R), \quad wR + \Soc(R) = wF\oplus\Soc(R)
\]
are ideals of $R$ and
\begin{gather*}
V \defug (vF\oplus\Soc(R))/\Soc(R) \is R/(wF\oplus\Soc(R)), \\
W \defug (wF\oplus\Soc(R))/\Soc(R) \is R/(vF\oplus\Soc(R))
\end{gather*}
are non isomorphic simple right $R$ modules, which are the maximal elements of $\Simp_R$. Now observe that, given $n>0$, $a,b\in F$, $k\in K$ and $l\in L$, the element $x = va+wb+k+l$ annihilates $U_n$ if and only if $b=0$ and $l_n = -(va)_n$. We infer that $r_R(U_n)\sbs
r_R(W)$ but $r_R(U_n)\nsbs r_R(V)$. On the other side $x$ annihilates $U_0$ if and only if $a=0$ and $k = 0$, so that $r_R(U_0)\sbs r_R(V)$ but $r_R(U_0)\nsbs r_R(W)$. This shows that the Hasse diagram of $\Simp_R$ is
\begin{equation*}
\begin{split}
\begin{picture}(150,70)(0,0)
\putbb{0}{40}{$V$}
\putl{0}{35}{0}{-1}{22}
\putbb{0}{0}{$U_0$} \putbb{30}{0}{$U_1$} \putbb{60}{0}{$U_2$}
\putbb{85}{0}{$\ldots$} \putbb{120}{0}{$U_n$}
\putbb{150}{0}{$\ldots$} \putbb{85}{20}{$\ldots$}
\putbb{130}{20}{$\ldots$} \putbb{85}{40}{$W$}
\putl{80}{35}{-2}{-1}{45} \putl{85}{35}{-1}{-1}{22}
\putl{90}{35}{1}{-1}{22}
\end{picture}
\end{split}
\end{equation*}
and $R$ is well behaved. If we take $\ZS = \{W,U_1,U_2,\ldots\}$, then $\ZS$ is an upper subset of $\Simp_R$ and $r_R(\ZS) = K$. However $\Simp_{R/K} = \{V,W,U_1,U_2,\ldots\}\psps\ZS$, therefore $R$ is not very well behaved.

Next, let us consider the subring
\[
S \defug vF\oplus wF\oplus L
\]
of $R$, which is clearly isomorphic to the factor ring $R/K$. We see easily that in the poset $\Simp_S$ we have $\l(V) = 1$, but $h(V) = 2$. Thus $S$ is not well behaved, yet $\Simp_S$ has finitely many maximal elements. Finally, both $R$ and $S$ are indecomposable rings, because $0$ and $1$ are the only central idempotents of $R$.
\end{proof}

\section{Connected components of $\Simp_R$.}\label{sect conncomp}

Let $I$ be a poset. Given $i,j\in I$, let us write $i\bowtie j$ to mean that either $i\le j$, or $i\ge j$, and write $i\sim j$ to mean that there are $k_0,k_1,\ldots,k_n\in I$ such that
\[
i = k_0\bowtie k_1\bowtie\cdots\bowtie k_n = j.
\]
Then $\sim$ is the smallest equivalence relation in $I$ containing
the partial order of $I$. The elements of $I/\!\!\!\sim$ are called
the \emph{connected components} of $I$; let us call the
\emph{canonical partition} of $\Simp_R$ the factor set
$\Simp_R/\!\!\!\sim$\,\,. There is a natural link between the
connected components of $\Simp_R$ and central idempotents of $R$.
First note that, without any assumption on the ring $R$, for every
complete set $\{e_1,\ldots,e_n\}$ of pairwise orthogonal and
\emph{central} idempotents of $R$ the set
\begin{equation}\label{eq centralpartition}
    \left\{\Simp_{e_1R},\ldots,\Simp_{e_nR}\right\}
\end{equation}
is a  partition of $\Simp_R$; in our present context, in which $R$
is semiartinian and regular, this partition is always coarser or
equal to the canonical partition. To see this, it is sufficient to
note that if $U\in\Simp_{e_iR}$ and $V\in\Simp_{e_jR}$ with $i\ne
j$, then $r_R(U)\nsbs r_R(V)$, meaning that $U\bowtie V$ is false and therefore $U\sim V$ is false too. In particular, if $\Simp_R$ consists of a single connected
component, then $R$ is indecomposable as ring, while the converse
may fail; in fact Example \ref{ex indec2compconn} displays two indecomposable semiartinian and regular rings $R$ and $S$ for which both $\Simp_R$ and $\Simp_S$ consist of two connected
components.

As we are going to see, if $\Prosimp_R$ is finite, then there is a
complete set $\{e_1,\ldots,e_n\}$ of pairwise orthogonal and central
idempotents of $R$ such that \eqref{eq centralpartition} coincides
with the canonical partition.

\begin{pro}\label{pro finmin}
Let $R$ be a semiartinian and regular ring. Then $\Simp_R$ has
finitely many minimal elements if and only if $\Prosimp_R$ is
finite. If it is the case, then $\Prosimp_R$ coincides with the set
of all minimal elements of $\Simp_R$ and there is a complete set
$\{e_1,\ldots,e_n\}$ of pairwise orthogonal and \emph{central}
idempotents such that \eqref{eq centralpartition} coincides with the
canonical partition; in particular each $e_iR$ is an indecomposable
ring.
\end{pro}
\begin{proof}
We already know that $\Prosimp_R$ is always contained in the set of
minimal elements of $\Simp_R$, thus the ``only if'' part is obvious.
Suppose that $\Prosimp_R$ is finite, let $U$ be a minimal element of
$\Simp_R$ and suppose that $U$ is not projective. Then $h(U) = \a+1$
for some $\a>0$ and, by applying finite induction and Lemma \ref{lem
orderreg}, we infer that there is some $y\in L_{\a+1}$ such that
$yR/yL_\a \is U$ and $\H P{yR}R = 0$ for every $P\in\Prosimp_R$. But
this means that $yR\cap\Soc(R) = 0$, which is a contradiction since
$\Soc(R)$ is essential as a right ideal and $y\ne 0$.

Assume now that $\Prosimp_R$ is finite and let
$\{\mathbf{S}_{1},\ldots,\mathbf{S}_{n}\}$ be the canonical
partition of $\Simp_R$. For every $i\in\{1,\ldots,n\}$ and
$U\in\mathbf{S}_i$, by applying again finite induction and Lemma
\ref{lem orderreg} we can choose an idempotent $y_U\in L_{h(U)}$
which satisfies the following conditions:
\begin{gather*}
y_UR/y_UL_{h(U)-1} \is U,  \\
\H P{y_UR}R = 0\quad \text{for all $P\in\Prosimp_R$ such that
$P\nin\mathbf{S}_i$}.
\end{gather*}
We may then consider the ideal $R_i = \sum\{Ry_UR\mid
U\in\mathbf{S}_i\}$ and it is clear that $U = U(Ry_UR) = UR_i$. We
claim that $R$ decomposes as
\begin{equation}\label{eq decompR}
    R = R_1\oplus\cdots\oplus R_n.
\end{equation}
First, since $R$ is regular, in order to prove that the sum
$R_1+\cdots+R_n$ is direct it is sufficient to show that if $i\ne
j$, then $R_iR_j = 0$. Thus, take $U\in\mathbf{S}_i$ and
$V\in\mathbf{S}_j$ with $i\ne j$. If $K$ is a simple right ideal
contained in $y_UR$, then $K\is P$ for a unique $P\in\Prosimp_R$.
Necessarily $P\in\mathbf{S}_i$ and therefore $\H P{y_VR}R = 0$. By
using the fact that $\Soc(R)$ is projective, we infer that
\[
\Soc(Ry_UR)\Soc(Ry_VR) = \Soc(Ry_UR)\cap\Soc(Ry_VR) = 0
\]
and hence $(Ry_UR)(Ry_VR) = (Ry_UR)\cap(Ry_VR) = 0$ by the
essentiality of the socle. Finally, since $U =
U(R_1\oplus\cdots\oplus R_n)$ for every simple module $U_R$, we
conclude that the equality \eqref{eq decompR} holds.
There is a complete set $\{e_1,\ldots,e_n\}$ of pairwise
orthogonal and central idempotents such that $e_iR = R_i$ for al $i$
and it follows from the above that $\Simp_{e_iR} = \mathbf{S}_i$ for
al $i$.
\end{proof}

\begin{re}\label{re nonregular}
It is worth of note that the assumption of regularity of the ring
$R$ cannot be dropped in Proposition \ref{pro finmin}. Indeed, with
\cite[Example 4.8]{Bac:15} we presented an indecomposable Artinian
algebra $R$ for which $\Simp_R$ consists of two connected
components; yet, $\Simp_R$ is finite.
\end{re}

\section{Comparability.}\label{sect comparability}

We keep the same setting and notations of the previous section.  In the literature
on regular rings we find two conditions involving comparability
between principal right ideals which play a central role in the
structure theory of these rings. Precisely, a regular ring $R$
satisfies the \emph{comparability} axiom if, given $x,y\in R$, one
has that either $xR\lesssim yR$ or $yR\lesssim xR$, while $R$
satisfies the \emph{general comparability} axiom if, given $x,y\in
R$, there exists some central idempotent $e$ such that $exR\lesssim
eyR$ and $(1-e)yR\lesssim (1-e)xR$ (see \cite{Good:3}). An
additional axiom, which makes sense when $R$ is semiartinian and
regular, was introduced in \cite{BacCiamp:14}: $R$ satisfies the
\emph{restricted comparability} axiom if, given $x,y\in R$, the condition $h(x)< h(y)$ implies that $xR\lesssim yR$. Comparability implies general comparability. If $R$ is a regular and semiartinian ring satisfying comparability, then it satisfies also restricted comparability. Indeed, if $x,y\in R$ with $h(x)< h(y)$, it is not the case that $yR\lesssim xR$ otherwise, since $x\in L_{h(x)}$, it would follow that $y\in L_{h(x)}$ too, that is $h(y)\le h(x)$. Thus $xR\lesssim yR$. As we know from
Theorem \ref{orderreg}, the natural partial order of $\Simp_R$ can
be expressed in terms of the existence of an imbedding between
certain principal right ideals; thus, it appears quite natural to
ask if, given a semiartinian and regular ring $R$, there is any
relationship between the above axioms and properties of the poset
$\Simp_R$. The results which follow give some answer to this
question.

\begin{pro}\label{comparability}
Let $R$ be a semiartinian and regular ring. Then $R$ satisfies the
restricted comparability axiom if and only if the following
condition holds:
\begin{equation}\label{eq lengthOK}
    \text{for every $U,V\in\Simp_R$, if $h(U)<h(V)$, then $U\prec V$.}
\end{equation}
In particular $R$ satisfies the comparability axiom if and only if $\Simp_R$ is a chain. If $R$ satisfies the restricted comparability axiom, then $R$ is well behaved.
\end{pro}
\begin{proof}
The ``only if'' part follows immediately from Theorem \ref{orderreg}. In order to prove the ``if'' part, we first observe that, for every ordinal $\a$, the Loewy chain of the ring $R/L_{\a}$ is $(L_{\g}/L_{\a})_{\a\le\g}$ and each primitive ideal of $R/L_{\a}$ has the form $P/L_{\a}$ for a unique primitive ideal $P$ of $R$. Consequently, if $R$ satisfies \eqref{eq lengthOK}, then the same holds for $R/L_{\a}$. Given an ordinal $\a$, let $P(\a)$ denote the sentence
\[
\text{``\,\,If $x,y\in R$ and $\a+1 = h(x)<h(y)$, then $xR\lesssim yR$\,\,''}\,.
\]
Then the proof of the first part of the proposition will be complete once we have shown that $P(\a)$ is true for every ordinal $\a$. Let $y\in R$ be such that $h(y) = \b+1$. Then
there is a decomposition $yR = y_1R\oplus\cdots\oplus y_nR$, where
each $y_iR/y_iL_\b$ is simple and $h(y_iR/y_iL_\b) = \b+1$. If $x\in
L_1 = \Soc(R)$, namely $h(x) = 1$, then $xR = P_1\oplus\cdots\oplus
P_m$, where each $P_j$ is simple with $h(P_j) = 1$. Thus, given
$j\in\{1,\ldots,m\}$ and $i\in\{1,\ldots,n\}$, it follows from the
assumption that $P_j\prec y_iR/y_iL_\b$ and we infer from Theorem
\ref{orderreg} that $P_j^k\lesssim y_iR$ for every positive integer
$k$. This is enough to infer that $xR\lesssim yR$ and so the statement $P(0)$ is true. Next, given an ordinal $a>0$, assume that $P(\b)$ is true for every $\b<\a$ and take $x,y\in R$ such that $\a+1 = h(x)<h(y)$. Then $0\ne x+L_{\a}\in L_{\a+1}/L_{\a} = \Soc(R/L_{\a})$, while $y+L_{\a}\nin \Soc(R/L_{\a})$. Since the ring $R/L_{\a}$ satisfies
\eqref{eq lengthOK}, we can apply the above argument and infer that $xR/xL_{\a}\lesssim yR/yL_{\a}$. It follows from
\cite[Proposition 2.20]{Good:3} that there are $x',x''\in xR$, $y',y''\in yR$ and decompositions
\[
xR = x'R\oplus x''R, \qquad yR = y'R\oplus y''R,
\]
where $x'R\is y'R$ and $x''\in L_{\a}$. Necessarily $h(y'') = h(y)$ and, since $h(x'')\le\a$, it follows that $h(x'')<h(x)<h(y) = h(y'')$. From the inductive hypothesis we infer that $x''R\lesssim y''R$ and therefore $xR\lesssim yR$. We conclude that $P(\a)$ is true.

If $R$ satisfies the comparability axiom, then $\BL_2(R)$ is a chain
by \cite[Proposition 8.5]{Good:3}. Consequently $\Prim_R$ is a chain
as well and so is $\Simp_R$. Conversely, if this latter condition
holds, then $\BL_2(R)$ is a chain because every ideal of $R$ is the
intersection of primitive ideals. The proof that, consequently, $R$
satisfies the comparability axiom is identical to the proof of
\cite[Proposition 4]{BacCiamp:14}.

Assume that $R$ satisfies the restricted comparability axiom. If
$U\in\Simp_R$ and $h(U) = 1$, then $U$ is minimal and so $\l(U) =
1$. Given a successor ordinal $\a+1$, assume that $\l(U) = h(U)$
whenever $h(U)<\a+1$, let $U\in\Simp_R$ be such that $h(U) = \a+1$
and suppose that $\l(U) < h(U)$. Inasmuch as $\l(U)$ is a successor
ordinal less than the Loewy length of $R$, there exists
$V\in\Simp_R$ such that $h(V) = \l(U)$ and, from the inductive
hypothesis, we have that $\l(V) = h(V) = \l(U)$. Thus $U$ and $V$
are not comparable. On the other hand, let $x\in L_{\a+1}\setminus
L_{\a}$ be such that $xR/xL_{\a+1}\is U$ and chose $y\in
L_{h(V)}\setminus L_{h(V)-1}$ such that $yR/yL_{h(V)-1}\is V$. Since
$h(V)<\a+1$, by the hypothesis $yR\lesssim xR$ and we infer that $\H
V{xR/xL_{h(V)-1}}R \ne 0$. It follows from Theorem \ref{orderreg}
that $V\prec U$: a contradiction. We conclude that $\l(U) = h(U)$
and the proof is complete.
\end{proof}

\begin{pro}\label{gencomparability}
Let $R$ be a semiartinian and regular ring. If $R$ satisfies the
general comparability axiom, then $\Simp_R$ is the union of pairwise
disjoint maximal chains. Conversely, if $\Simp_R$ is the union of
\emph{finitely many} pairwise disjoint maxi\-mal chains, then $R$
satisfies the general comparability axiom and is well behaved.
\end{pro}
\begin{proof}
Inasmuch as $\Simp_R$ is artinian, it is sufficient to show that
$\{U\preccurlyeq\}$ is a chain whenever $U$ is a minimal element of
$\Simp_R$. However this follows from \cite[Theorem 8.20]{Good:3}, combined with Proposition \ref{comparability},
since $r_R(U)$ is a prime ideal for every $U\in\Simp_R$.

Conversely, assume that $\Simp_R$ is the union of \emph{finitely
many} pairwise disjoint maximal chains
$\{\mathbf{S}_{1},\ldots,\mathbf{S}_{n}\}$, which are necessarily
the connected components of $\Simp_R$. Then $\Prosimp_R$ is finite
and, according to Proposition \ref{pro finmin}, $R$ decomposes as in
\eqref{eq decompR}, where every $R_i$ is a semiartinian and regular
ring such that $\Simp_{R_i}$ is a chain. By Proposition
\ref{comparability} every $R_i$ satisfies the comparability axiom,
therefore $R$ satisfies the general comparability axiom. Now,
observe that if $U\in\Simp_R$, then $U = UR_i$ for a unique $i$, while $UR_j = 0$ if $j\ne i$.
Consequently, since each ring $R_i$ well behaved by Proposition
\ref{comparability} and
\[
\Soc_\a(R) = \Soc_\a(R_1)\oplus\cdots\oplus\Soc_\a(R_n)
\]
for every ordinal $\a$, it
is an easy matter to conclude that $R$ is well behaved.
\end{proof}

The following example shows that it is not possible to remove the
finiteness condition from Proposition \ref{gencomparability}.

\begin{ex}\label{ex nongeneralcompar}
There exists a semiartinian and regular ring $R$, with Loewy length
2 and all primitive factors artinian (hence $\Simp_R$ is the union
of pairwise disjoint maximal chains), which does not satisfy the
general comparability axiom.
\end{ex}
\begin{proof}
Given a field $F$, set $R_n = \BM_2(F)$ for every positive integer
$n$ and consider the following regular subring of the direct product
$T = \prod_{n>0}R_n$:
\[
R = K\oplus L\oplus\left(\bigoplus_{n>0}R_n\right),
\]
where
\begin{gather*}
K = \left\{k\in T\mid \text{there is $a\in F$ such that $k_n =
\bigl(
\begin{smallmatrix}
a&0 \\0&0
\end{smallmatrix}
\bigr)$ for all $n>0$}\right\}, \\
L = \left\{l\in T\mid \text{there is $a\in F$ such that $l_n =
\bigl(
\begin{smallmatrix}
0&0 \\0&a
\end{smallmatrix}
\bigr)$ for all $n>0$}\right\}.
\end{gather*}
We observe that
\[
\Soc(R) = \bigoplus_{n>0}R_n \quad\text{and}\quad R/\Soc(R)\is
F\times F,
\]
therefore $R$ is semiartinian with Loewy length 2 and has all
primitive factors artinian. If we set
\[
u = \bigl(\bigl(
\begin{smallmatrix}
1&0 \\0&0
\end{smallmatrix}
\bigr), \bigl(
\begin{smallmatrix}
1&0 \\0&0
\end{smallmatrix}
\bigr), \ldots\bigr)\,\,, \qquad v = \bigl(\bigl(
\begin{smallmatrix}
0&0 \\0&1
\end{smallmatrix}
\bigr), \bigl(
\begin{smallmatrix}
0&0 \\0&1
\end{smallmatrix}
\bigr), \ldots\bigr),
\]
then
\[
U = (uR+\Soc(R))/\Soc(R) \quad\text{and}\quad V =
(vR+\Soc(R))/\Soc(R)
\]
are non-isomorphic simple $R$-modules and $\Simp_{R/\Soc(R)} =
\{U,V\}$. Now an idempotent $e\in\Soc(R)$ is central if and only if
all its nonzero coordinates equal $\bigl(
\begin{smallmatrix}
1&0 \\0&1
\end{smallmatrix}
\bigr)$, while all remaining central idempotents of $R$ are of the
form $1-e$, where $e$ is a central idempotent of $\Soc(R)$. If $e$
is a central idempotent of $\Soc(R)$, then it is clear that $euR \is
evR$, but if $(1-e)vR$ were subisomorphic to $(1-e)uR$, since
$(1-e)v$ and $(1-e)u$ do not belong to $\Soc(R)$, we would get
\[
V = ((1-e)vR+\Soc(R))/\Soc(R) \subis ((1-e)uR+\Soc(R))/\Soc(R) = U,
\]
hence a contradiction; similarly, $(1-e)uR$ is not subisomorphic to
$(1-e)vR$. We conclude that $R$ does not satisfy the general
comparability axiom.
\end{proof}

\section{A very special well ordered chain of subrings of $\CFM_{X}(D)$.}\label{sect wochsbr}

With this section we begin the setup which will bring us to the construction of regular and semiartinian rings, starting from an artinian poset. We set the scenario by taking a ring $D$ (although our final concern will be the case in which $D$ is a division ring, unless otherwise stated we do not assume anything about $D$, apart associativity and presence of a multiplicative identity), a \underline{transfinite ordinal} $X$ and the ring $Q = \CFM_X(D)$ of all $X\times X$-matrices with entries in $D$ whose columns have finite support.

\begin{notations}
With the above setting, we adopt the following notations:
\begin{itemize}
    \item We denote by $\mathbf{0}$ and $\mathbf{1}$ the zero and the unital matrices respectively.
     \item   If $\za\in Q$ and $x,y\in X$, we use the
symbol $\za(x,y)$ to denote the entry at the intersection of the
$x$-th row with the $y$-th column of $\za$ (i. e. the $(x,y)$-entry of $\za$), instead of the more
traditional symbol $\za_{xy}$; since we often use more complex
arrays, other than single letters, in order to designate the position of the entries of the matrices we deal with, our choice should guarantee a better readability. If $Y,Z\sbs X$, then $\za(Y,Z)$ is be the $(Y,Z)$-block of $\za$, that is the
    submatrix $(\za(x,y))_{y\in Y, z\in Z}$ of $\za$.
      \item For every $Y\sbs X$, we denote with $\ze_Y$ the
idempotent diagonal matrix such that $\ze_Y(x,x)$ is $1$ if $x\in Y$
and is $0$ otherwise. If $x,y\in X$, we write $\ze_x$ instead of $\ze_{\{x\}}$, while $\ze_{x,y}$ stands for the matrix whose $(x,y)$-entry is $1$ and all others are zero; so, in particular $\ze_{x} = \ze_{x,x}$.
    \item $\BF\BR_X(D)$ and $\BF\BM_X(D)$ denote respectively the subset of $Q$ of all matrices having only finitely many nonzero rows and the subset of $Q$ of all matrices having only finitely many nonzero entries.
\end{itemize}
\end{notations}

$\BF\BR_X(D)$ is an ideal of $Q$ which is of a special
interest for us; as a right ideal, it is generated by the set
$\{\ze_{x}\mid x\in X\}$ of pairwise orthogonal idempotents and we have the equalities
\begin{gather}
\ze_{x}Q = \ze_{x}\BF\BR_X(D), \label{eq FR1} \\
\BF\BR_X(D) = \bigoplus\{\ze_{y}Q \mid y\in X\} = \BF\BR_X(D)\ze_{x}\BF\BR_X(D) \label{eq FR2}.
\end{gather}
Moreover $\BF\BR_X(D)$ is fully invariant; this follows from a more general result of Del R\`{\i}o and Sim\`{o}n (see \cite[Lemma, 7]{DelRioSimon:01}) although, for the case $X = \omega$, it was a byproduct of a theorem of Camillo (see \cite{Camillo:01} and \cite{AbramsSimon:01}).
As a consequence, if $R$ is any ring and $\map{\f,\psi}QR$ are two ring isomorphisms, then $\f(\BF\BR_X(D))$$ = \psi(\BF\BR_X(D))$; let's say that the elements of this latter ideal are the \emph{finite-ranked} elements of $R$.
If we consider a free module $M_D$ with a basis of cardinality $|X|$, the map which
assigns to each endomorphism of $M$ its associated matrix with
respect to $B$ is a ring isomorphism from $\End(M_D)$ to $Q$, which
restricts to an isomorphism from the ideal of finite rank
endomorphisms to the ideal $\BF\BR_X(D)$. If $D$ is a division ring
(thus $M_D$ is a vector space), then it is well known that $Q$ is
regular, left selfinjective and  $\BF\BR_X(D) = \Soc(Q)$. We shall consider $D$ as a
subring of $Q$ by identifying each element of $D$ with the
corresponding scalar matrix in $Q$. We call $D$-{\em subring\/} of $Q$ every (not necessarily unital) subring $S$ which is closed with respect
to both right and left multiplication by elements of $D$, namely it
is a $(D,D)$-submodule of $Q$. Of course, if $S$ is a $D$-subring of
$Q$, then $S$ is a unital subring if and only if $D\sbs S$; moreover
every ideal of $Q$ is a $D$-subring, while not every subring (unital
or not) is a $D$-subring. As far as $\BF\BM_{X}(D)$ is concerned, it is a \emph{left} ideal of $Q$, which is not a right ideal, and for every $x\in X$ the following hold:
\begin{gather}
Q\ze_{x} = \BF\BM_X(D)\ze_{x}, \label{eq FM1} \\
\BF\BM_X(D) = \bigoplus\{Q\ze_{y} \mid y\in X\} = \BF\BM_X(D)\ze_{x}\BF\BM_X(D). \label{eq FM2}
\end{gather}
In the sequel it will be useful to bear in mind the obvious observation that every matrix in $\BF\BM_{X}(D)$ is a finite sum of matrices of the form $d\ze_{x,y} = \ze_{x,y}d$ for $d\in D$ and $x,y\in X$.

Finally we observe that both $\BF\BR_X(D)$ and $\BF\BM_{X}(D)$ are pure as left ideals of $Q$; indeed, if $\mathbf{0}\ne\za\in\BF\BR_X(D)$ and $Y$ is the subset of $X$ of
those $x$ such that the $x$-th row of $\za$ is not zero, then
$\ze_Y\in\BF\BM_X(D)$ and $\za = \ze_Y\za$.

Given any ordinal $\xi\le X$, our program in this section is to define a family $(Q_\a)_{\a\le\xi}$ of unital subrings of $Q$ having the following features: (a) if $\a<\xi$, then $Q_\a$ is isomorphic to $Q$, (b) by denoting with $F_\a$ the ideal of $Q_\a$ of all finite-ranked elements when $\a<\xi$, then $Q_{\b}\cap F_{\a} = 0$ whenever $\a<\b\le\xi$. Our construction heavily bears on ordinal arithmetic; however, since ordinal arithmetic is not so frequently used in ring theory, we think useful to list here some of the basic facts we
shall use, omitting their proof (see \cite{Jech:1} or
\cite{LevyAz:1}, for example).

First recall that every ordinal $\a$ is just the set whose elements
are all ordinals $\b$ such that $\b<\a$; in particular
$\a\not\in\a$, while $\b<\a$ exactly means  $\b\in\a$. An initial ordinal, that is an ordinal $\al$ such
that $|\a|<|\al|$ for every ordinal $\a<\al$, is called a cardinal
number; for every set $X$ there is a unique cardinal $\al$ such that
$|X| = |\al|$ and one writes $|X| = \al$.

Ordinal addition, multiplication and exponentiation are defined as
follows: given an ordinal $\a$,
\begin{gather*}
\a+0 = \a,\quad \a+1 = \a\cup\{\a\}, \\
\a+(\b+1) = (\a+\b)+1 \quad\text{ for every ordinal }\b, \\
\a+\b = \sup\{\a+\g\mid \g<\b\} \quad\text{ for every limit ordinal
}\b\ne 0;
\end{gather*}
\begin{gather*}
\a\bullet 0 = 0, \\
\a\bullet(\b+1) = (\a\bullet\b)+\a \quad\text{ for every ordinal }\b, \\
\a\bullet\b = \sup\{\a\bullet\g\mid \g<\b\} \quad\text{ for every
limit ordinal }\b\ne 0;
\end{gather*}
\begin{gather*}
\a^{0} = 1, \\
\a^{\b+1} = \a^{\b}\bullet\a \quad\text{ for every ordinal }\b, \\
\a^{\b} = \sup\{\a^{\g}\mid \g<\b\} \quad\text{ for every limit
ordinal }\b\ne 0.
\end{gather*}
Ordinal arithmetic differs deeply from arithmetic of cardinals. For
example, if\linebreak $\omega = \al_{0}$, as ordinal exponential we
have that $2^\omega = \omega$, while $2^\omega$ is
uncountable\linebreak if we consider cardinal exponentiation. Since
in our work we always use \underbar{ordinal}
\underbar{exponentiation}, there will be no conflict with notations.
Note that $\a\bullet\b$ is isomorphic, as a well ordered set, to the
direct product $\a\times\b$ with the antilexicographic ordering.
If $\a$ and $\b$ are ordinals such that $\a<\b$, then there exists a
unique ordinal $\b-\a$ such that $\b = \a+(\b-\a)$. It follows that if $\a<\b<\g$, then $(\b-\a)+(\g-\b) = \g-\a$. Addition and
multiplication are both associative but are not commutative;
multiplication is distributive on the left with respect to addition,
but not on the right. All ordinals (resp. all nonzero ordinals) are
left cancellable with respect to addition (resp. multiplication),
but need not be right cancellable. If $\a,\b,\g$ are ordinals, then
$\a<\b$ if and only if $\g+\a<\g+\b$; if, in addition, $\g\ne 0$,
then $\a<\b$ if and only if $\g\bullet\a<\g\bullet\b$.

Using the definitions and induction it is easy to show that:
\[
0\bullet\a = 0 = \a\bullet 0\quad \text{ and }\quad1\bullet\a = \a =
\a\bullet 1\quad \text{ for every ordinal }\a;
\]
moreover, if $1<\a$ and $1<\b$, then $\a<\a\bullet\b$ and
$\b<\a\bullet\b$.

\begin{pro}\label{expmonoton}
If $\a$, $\b$, $\g$ are ordinals with $1<\g$,  then $\a<\b$ if and only if $\g^{\a}<\g^{\b}$.
\end{pro}

It is immediate from the definition that $\b$ is a limit ordinal if and only if $\a+\b$ is limit for every $\a$.

\begin{pro}\label{prodexplimit}
Given two ordinals $\a$ and $\b>0$, then both $\a\bullet\b$ and
$\a^{\b}$ are limit ordinals in case either $\a$ or $\b$ is limit.
\end{pro}

\begin{pro}\label{rulexp}
Given three ordinals $\a$, $\b$, $\g\ne 0$, the following equality
holds:
    \begin{equation}\label{eqrulexp}
    \g^{\a+\b} = \g^{\a}\bullet\g^{\b}.
    \end{equation}
    \end{pro}

Division with unique quotient and remainder between ordinals is
possible ``on the left'', as stated in the proposition which follows.
This possibility is actually the key of our construction;
we will make an extensive use of it without an explicit
mention.

\begin{pro}\label{ruleuclid}
Given two ordinals $\a$, $\b_1$ with $\b_1\ne 0$, there are unique
ordinals $\g$, $\a_1$ (called respectively the quotient and the remainder of
the division of $\a$ by $\b_1$) such that
\begin{equation}\label{eqruleuclid}
    \a = \b_1\bullet\g + \a_1 \quad\text{and}\quad\a_1<\b_1.
\end{equation}
\end{pro}

\begin{re}\label{remeuclid}
Let $\b_{1}, \b_{2}$ be nonzero ordinals. Given an ordinal $\a<\b_1\bullet\b_2$, it follows from Proposition \ref{ruleuclid} that there is a unique ordinal $\g$ such that $\a$ belongs to the right open interval
\[
[\b_1\bullet\g, \b_1\bullet\g+\b_1) =
\{\b_1\bullet\g+\a_1\mid\a_1<\b_1\};
\]
necessarily $\g<\b_2$, for if $\g\ge\b_2$, then
$\a<\b_1\bullet\b_2 \le\b_1\bullet\b_2+\a_1\le\b_1\bullet\g+\a_1 = \a$ and hence a contradiction. Thus the set $\{[\b_1\bullet\g, \b_1\bullet\g+\b_1)\mid\g<\b_2\}$ is a partition of $\b_1\bullet\b_2$. Also note that, for every $\g<\b_2$, the assignment $\a_1\mapsto \b_1\bullet\g+\a_1$ defines a bijection from $\b_1$ to $[\b_1\bullet\g, \b_1\bullet\g+\b_1)$. These observations will be crucial for the construction which is the objective of our work.
\end{re}

Another feature we shall rely on is the following $n$-th iterate of Proposition \ref{ruleuclid}.

\begin{pro}\label{ruleuclidd}
Let $\b_1,\ldots,\b_n$ be nonzero ordinals. For every ordinal $\a$
there are unique ordinals $\g$ and $\a_k<\b_{n-k+1}$ for $k = 1,\ldots,n$ such that
    \begin{equation}
    \a = \b_1\bullet\cdots\bullet\b_{n}\bullet\g +
    \b_1\bullet\cdots\bullet\b_{n-1}\bullet\a_1 + \cdots +
    \b_1\bullet\a_{n-1} + \a_n
    \end{equation}
    and $\g$ is the quotient of the division of $\a$ by
    $\b_1\bullet\cdots\bullet\b_{n}$.
    If $\b_{n+1}$ is another ordinal such that
    $\a<\b_1\bullet\cdots\bullet\b_{n}\bullet\b_{n+1}$, then
    $\g<\b_{n+1}$.
    \end{pro}
    \begin{proof}
    If $n = 1$, the first statement is merely Proposition
    \ref{ruleuclid} together with Remark \ref{remeuclid}. Suppose
inductively that
    the statement is true for some $n\ge 1$ and consider $n+1$ ordinals
    $\b_1,\ldots,\b_{n+1}$. Given $\a$, by Proposition
    \ref{ruleuclid} there are unique $\g$ and
    $\d<\b_1\bullet\cdots\bullet\b_{n+1}$ such that
    \begin{equation}\label{qq}
    \a = \b_1\bullet\cdots\bullet\b_{n+1}\bullet\g+\d.
    \end{equation}
    By the inductive hypothesis, there are
    $\a_1<\b_{n+1},\a_2<\b_{n}\ldots,\a_{n+1}<\b_{1}$ such that
    \[
    \d = \b_1\bullet\cdots\bullet\b_{n}\bullet\a_{1} +
    \b_1\bullet\cdots\bullet\b_{n-1}\bullet\a_2 + \cdots +
    \b_1\bullet\a_{n} + \a_{n+1}.
    \]
    As a result
    \[
    \a = \b_1\bullet\cdots\bullet\b_{n+1}\bullet\g +
    \b_1\bullet\cdots\bullet\b_{n}\bullet\a_{1} + \cdots +
    \b_1\bullet\a_{n} + \a_{n+1}.
    \]
    Suppose that also
    \[
    \a = \b_1\bullet\cdots\bullet\b_{n+1}\bullet\g' +
    \b_1\bullet\cdots\bullet\b_{n}\bullet\a'_{1} + \cdots +
    \b_1\bullet\a'_{n} + \a'_{n+1},
    \]
    where $\a'_1<\b_{n+1},\ldots,\a'_{n+1}<\b_{1}$. Using the left
distributivity of
    multiplication with respect to the addition, we infer from
    uniqueness of the quotient and remainder of the division of $\a$ by
    $\b_1$ that $\a_{n+1} = \a'_{n+1}$ and
    \begin{gather*}
    \b_2\bullet\cdots\bullet\b_{n+1}\bullet\g +
    \b_2\bullet\cdots\bullet\b_{n}\bullet\a_{1}+\cdots+\a_{n}
    \\
    = \b_2\bullet\cdots\bullet\b_{n+1}\bullet\g' +
    \b_2\bullet\cdots\bullet\b_{n}\bullet\a'_{1}+\cdots+\a'_{n}.
    \end{gather*}
    Again from the inductive hypothesis it follows that $\g = \g'$
    and $\a_k = \a'_k$ for $1\le k\le n$. Concerning the last statement,
if $\b_{n+1}$ is
    another ordinal such that
    $\a<\b_1\bullet\cdots\bullet\b_{n}\bullet\b_{n+1}$, then it follows
from Proposition
    \ref{ruleuclid} and Remark \ref{remeuclid} that $\g<\b_{n+1}$ and the
proof is
    complete.
    \end{proof}

        \begin{pro}\label{cardaleph}
    Given an ordinal $\xi$ and an \underbar{infinite} cardinal $\al$ such
that
    $\xi\le\al$, if $0<\a\le\xi$ then, as ordinal exponential,
    \[
    |\al^{\a}| = \al.
    \]
       \end{pro}
    \begin{proof}
    The equality being obvious if $\a = 1$, suppose that $1<\a\le\xi$
    and $ |\al^{\b}| = \al$ for all nonzero $\b<\a$. If $\a = \b+1$ for
    some $\b$, then
    \[
    |\al^{\a}| = |\al^{\b+1}| = |\al^{\b}\bullet\al| = |\al^{\b}\times\al|
    = |\al\times\al| = \al.
    \]
    If $\a$ is limit, then $\al^{\a} = \sup\{\al^{\b}\mid\b<\a\} =
    \bigcup\{\al^{\b}\mid\b<\a\}$, therefore
    \[
    |\al^{\a}| = \sup(\{|\al^{\b}|\mid \b<\a\}\cup\{|\a|\}) = \al.
    \]
        \end{proof}

In order to obtain results which are general enough to be readily used in the subsequent sections,
throughout the remaining part of this section we assume that
\[
X = \al^{\xi}\bullet\bet,
\]
where $\al$ is a given \underbar{infinite} cardinal, $\bet$ is a second nonzero cardinal such that $\bet\le\al$ and $\xi$ is an ordinal such that $\xi\le\al$. We want to stress that we are using ordinal exponentiation and multiplication. It is clear from Proposition \ref{cardaleph} that $|X| = \al$.

    We say that a partition $\CP$ of $X$ is an $\al$-{\em
partition\/} if $|Y| = \al$ for all $Y\in\CP$; we denote by
$\BP_{\al}(X)$ the set of all such partitions. Given a cardinal
$\al'\le\al$, we say that a partition $\CQ$ of $X$ is $\al'$-{\em
coarser\/} than a partition $\CP\in\BP_{\al}(X)$ if each element of
$\CQ$ is the union of $\al'$ elements of $\CP$; if it is the case,
then it is clear that $\CQ\in\BP_{\al}(X)$ and $\al'\le|\CP|$. Using
the natural ordering and the arithmetical properties of ordinals we
can define a sequence of partitions $\{\CP_{\a}\mid 0<\a\le\xi\}$ of
the set $X$, in such a way that each $\CP_{\a}$ is an
$\al$-partition and $\CP_{\b}$ is $\al$-coarser than $\CP_{\a}$
whenever $0<\a<\b$. Precisely, for every $\a\le\xi$ and
$\l<\al^{\xi-\a}\bullet\bet$ let us consider in $X$ the right open
interval
\[
    X_{\a,\l} \defug [\al^{\a}\bullet\l,\,\,\, \al^{\a}\bullet\l +
    \al^{\a}) = \{\al^{\a}\bullet\l+\r\mid\r<\al^{\a}\}.
\]
Observing that $X =
\al^{\a}\bullet\left(\al^{\xi-\a}\bullet\bet\right)$, according to
Remark \ref{remeuclid} the set
    \begin{equation}\label{partition}
    \CP_{\a} \defug \{X_{\a,\l}\mid \l<\al^{\xi-\a}\bullet\bet\}
    \end{equation}
is a partition of $X$ and $|X_{\a,\l}| = |\al^{\a}|$, hence
$|X_{\a,\l}| = \al$ by Proposition \ref{cardaleph}. Thus $\CP_{\a}$
is an $\al$-partition of $X$. An element $x\in X$ belongs to $X_{\a,\l}$ if and only if $\l$ is the quotient of the division (on the left) of $x$ by $\al^{\a}$. We can extend the definition of the partition $\CP_{\a}$ to the case $\a = 0$ by observing that $X_{0,\l} = \{\l\}$ for every $\l<\al^{\xi}\bullet\bet$. Thus $\CP_{0}$ is just the trivial partition of $X$ in which each member is a singleton.

\begin{lem}\label{lpartpart}
If $\a<\b\le\xi$, then $\CP_\b$ is $\al$-coarser than $\CP_\a$; specifically
\begin{equation}\label{partpart}
  X_{\b,\l} =
\bigcup\{X_{\a,\,\,\al^{\b-\a}\bullet\l+\mu}\mid\mu<\al^{\b-\a}\}
\end{equation}
for every $\l <\al^{  \xi-\b}\bullet\bet$.
\end{lem}
\begin{proof}
Given $\l <\al^{  \xi-\b}\bullet\bet$, suppose that $x\in
X_{\b,\l}$, namely  $x = \al^{\b}\bullet\l+\r$ for some
$\r<\al^{\b}$. Then it follows from Proposition \ref{ruleuclidd}
that there are unique $\mu<\al^{\b-\a}$ and $\s<\al^{\a}$ such that
\[
x = \al^{\a}\bullet\al^{\b-\a}\bullet\l + \al^{\a}\bullet\mu + \s
\in X_{\a,\,\,\al^{\b-\a}\bullet\l+\mu}.
\]
Conversely, take any $\mu<\al^{\b-\a}$ and observe that
\[
X_{\a,\,\,\al^{\b-\a}\bullet\l+\mu} =
[\al^{\b}\bullet\l+\al^{\a}\bullet\mu,
\al^{\b}\bullet\l+\al^{\a}\bullet\mu + \al^{\a}).
\]
Obviously
$\al^{\b}\bullet\l\le\al^{\b}\bullet\l+\al^{\a}\bullet\mu$; on the
other hand, since $\mu<\al^{\b-\a}$ and $\al^{\b-\a}$ is a limit
ordinal by Proposition \ref{prodexplimit}, then $\mu+1<\al^{\b-\a}$
and consequently
\begin{gather*}
\al^{\b}\bullet\l+\al^{\a}\bullet\mu + \al^{\a} =
\al^{\b}\bullet\l+\al^{\a}\bullet(\mu + 1) \\
< \al^{\b}\bullet\l + \al^{\a}\bullet\al^{\b-\a} = \al^{\b}\bullet\l
+\al^{\b}.
\end{gather*}
This shows that $X_{\a,\,\,\al^{\b-\a}\bullet\l+\mu}\sbs X_{\b,\l}$,
as wanted.
\end{proof}

\begin{notation}\label{notquorem}
Given $x\in X$ and $\a\le\xi$, we shall denote by $x_{\a,q}$ and
$x_{\a,r}$ respectively the quotient and the remainder of the (left)
division of $x$ by $\al^{\a}$, namely the unique ordinals such that
$x_{\a,r}<\al^{\a}$ and
\begin{equation}\label{decomp}
x = \al^{\a}\bullet x_{\a,q} + x_{\a,r}.
\end{equation}
Note that $x_{\a,q}<\al^{\xi-\a}\bullet\bet$ by Proposition
\ref{ruleuclidd}.
\end{notation}

Let us consider the ring $Q = \CFM_X(D)$ and, for every $\a\le\xi$,
let us consider the subset $Q_{\a}$ of $Q$ consisting of those
matrices $\za$ satisfying the following condition:
\begin{equation}\label{magicmatrix}
\za(x,y) = \d(x_{\a,r},y_{\a,r})\,\za(\al^{\a}\bullet
x_{\a,q},\al^{\a}\bullet y_{\a,q}) \quad\text{ for all }x,y\in X
\end{equation}
(here and in the sequel $\d$ stands for the ``Kronecker delta''
function). Thus $Q_{\a}$ consists of those matrices $\za\in Q$ such
that, for every $\l,\mu<\al^{\xi-\a}\bullet\bet$, the block
$\za(X_{\a,\l},X_{\a,\mu})$ is a scalar
$\al^{\a}\times\al^{\a}$-matrix. It is clear that $Q_{0} = Q$ and $D\sbs
Q_{\a}$ for all $\a$.

\begin{theo}\label{pro-chainflr}
Given an ordinal $\xi>0$, a cardinal $\bet>0$ and a ring $D$, let
$\al$ be the first \underbar{infinite} cardinal such that
$\sup\{|\xi|,\bet\}\le\al$, set $X = \al^{\xi}\bullet\bet$ and
consider the ring $Q = \CFM_X(D)$. Then, with the above notations,
the following properties hold:
\begin{enumerate}
    \item For every $\a\le\xi$ there is a unital monomorphism
$\f_\a\colon\CFM_{\al^{\xi-\a}\bullet\bet}(D)$$\to Q$ of rings such
that $\Imm(\f_\a) = Q_{\a}$; in particular, if $\a<\xi$, then
$Q_{\a}$ is a unital $D$-subring of $Q$ isomorphic to $Q$.
    \item Given $\a\le\xi$, let us consider the $D$-subrings $F_\a =
\f_\a(\FR_{\al^{\xi-\a}\bullet\bet}(D))$ of $Q_{\a}$ and $G_\a =
\f_\a(\FM_{\al^{\xi-\a}\bullet\bet}(D))$ of $Q_{\a}$. Then a matrix $\zb\in Q_{\a}$
belongs to $F_\a$ if and only if it satisfies the following
condition:
\begin{align*}
    (\star)\quad &\text{there are
$\l_1,\ldots,\l_n\in\al^{\xi-\a}\bullet\bet$ such that if the
$x$-th row of $\zb$ is} \\
&\text{not zero, then $x\in X_{\a,\l_1}\cup\cdots\cup X_{\a,\l_n}$},
\end{align*}
while $\zb$ belongs to $G_\a$ if and only if it satisfies the following
condition:
\begin{align*}
    (\star\star)\quad &\text{there are
$\l_1,\ldots,\l_n\in\al^{\xi-\a}\bullet\bet$ such that if the entry
$\zb(x,y)$ of $\zb$ is not} \\
&\text{zero, then $x,y\in X_{\a,\l_1}\cup\cdots\cup X_{\a,\l_n}$}.
\end{align*}
\item If $\a<\b\le\xi$, then $Q_\b\sbs Q_\a$\,. \item If
$\a_1<\ldots<\a_n<\b\le\xi$, then
\[
\left[F_{\a_1}+\cdots+F_{\a_n}\right]\cap Q_{\b} = 0;
\]
consequently the set $\{F_{\a}\mid \a\le\xi\}$ of $(D,D)$-submodules
of $Q$ is independent and so is, in turn, the set $\{G_{\a}\mid \a\le\xi\}$.
\end{enumerate}
\end{theo}
\begin{proof}
(1) Given $\a\le\xi$, let us define the map
$\map{\f_\a}{\CFM_{\al^{\xi-\a}\bullet\bet}(D)}{Q}$ as follows: for
all $x,y\in X$
\begin{equation}\label{a}
\f_{\a}(\za)(x,y) = \d(x_{\a,r},y_{\a,r})\,\za(x_{\a,q},y_{\a,q}).
\end{equation}
Then, given $\za\in \CFM_{\al^{\xi-\a}\bullet\bet}(D)$ and $x,y\in
X$, we have that
\begin{align*}
\f_{\a}(\za)(x,y) &=
\d(x_{\a,r},y_{\a,r})\,\d(0,0)\,\za(x_{\a,q},y_{\a,q}) \\
                  &= \d(x_{\a,r},y_{\a,r})\,\f_{\a}(\za)(\al^{\a}\bullet
                        x_{\a,q},\al^{\a}\bullet y_{\a,q}),
\end{align*}
therefore $\Imm(\f_{\a})\sbs Q_{\a}$. Conversely, given $\zb\in
Q_{\a}$, let $\za\in \CFM_{\al^{\xi-\a}\bullet\bet}(D)$ be the
matrix defined by $\za(\l,\mu) =
\zb(\al^{\a}\bullet\l,\al^{\a}\bullet\mu)$ for all
$\l,\mu<\al^{\xi-\a}\bullet\bet$. Then for every $x,y\in X$ we
have
\begin{align*}
\f_{\a}(\za)(x,y) &= \d(x_{\a,r},y_{\a,r})\,\za(x_{\a,q},y_{\a,q}) \\
                  &= \d(x_{\a,r},y_{\a,r})\, \zb(\al^{\a}\bullet
x_{\a,q},\al^{\a}\bullet y_{\a,q}) \\
                  &= \zb(x,y);
\end{align*}
consequently $\zb = \f_{\a}(\za)$ and hence $Q_{\a} =
\Imm(\f_{\a})$. It is clear that $\f_{\a}(1) = 1$ and $\f_{\a}$ is a
homomorphism of additive groups. Given $\za, \zb\in
\CFM_{\al^{\xi-\a}\bullet\bet}(D)$, for all $x, y\in X$ we have that
\begin{align*}\label{ringmorph}
\f_{\a}(\za\zb)(x,y) &=
\d(x_{\a,r},y_{\a,r})\,(\za\zb)(x_{\a,q},y_{\a,q}) \\
 &= \sum_{\mu<\al^{\xi-\a}\bullet\bet}
\d(x_{\a,r},y_{\a,r})\,\za(x_{\a,q},\mu)\,\zb(\mu,y_{\a,q}) \\
&= \sum_{\substack{\mu<\al^{\xi-\a}\bullet\bet\\\r<\al^{\a}}}
\d(x_{\a,r},\r)\,\d(\r,y_{\a,r})\,
\za(x_{\a,q},\mu)\,\zb(\mu,y_{\a,q}) \\
&= \sum_{z\in X} \d(x_{\a,r},z_{\a,r})\,\d(z_{\a,r},y_{\a,r})\,
\za(x_{\a,q},z_{\a,q})\,\zb(z_{\a,q},y_{\a,q}) \\
&= \left(\f_{\a}(\za)\,\f_{\a}(\zb)\right)(x,y),
\end{align*}
hence $\f_{\a}$ is a ring homomorphism. Finally, if $\za\in
\CFM_{\al^{\xi-\a}\bullet\bet}(D)$ and $\za(\l,\mu)\ne 0$ for some
$\l,\mu<\al^{\xi-\a}\bullet\bet$, then $(\f_{\a}(\za))(x,y) \ne 0$
whenever $x = \al^{\a}\bullet\l+\r$ and $y = \al^{\a}\bullet\mu+\r$
for some $\r<\al^{\a}$; this shows that $\f_{\a}$ is injective. As a
result, if $\a<\xi$,  since $|\al^{\xi-\a}\bullet\bet| = \al$ by
Proposition \ref{cardaleph}, we have that $Q_{\a}\is
\CFM_{\al^{\xi-\a}\bullet\bet}(D)\is Q$.

(2) Let $\zb\in Q_\a$ and take $\za\in
\CFM_{\al^{\xi-\a}\bullet\bet}(D)$ such that $\zb = \f_\a(\za)$. If
$\zb\in F_{\a}$, that is $\za\in\FR_{\al^{\xi-\a}\bullet\bet}(D)$,
then there are $\l_1,\ldots,\l_n\in\al^{\xi-\a}\bullet\bet$ such
that the $\l$-th row of $\za$ is not zero only if $\l = \l_i$ for
some $i$. Consequently, if $x,y\in X$, by \eqref{a} we see that
$\zb(x,y)\ne 0$ only if $\za(x_{\a,q},y_{\a,q}) \ne 0$, only if
$x_{\a,q} = \l_i$ for some $i$, only if $x\in
X_{\a,\l_1}\cup\cdots\cup X_{\a,\l_n}$. Similarly, if $\zb\in G_{\a}$, namely $\za\in\FM_{\al^{\xi-\a}\bullet\bet}(D)$,
then there are $\l_1,\ldots,\l_n\in\al^{\xi-\a}\bullet\bet$ such
that the $(\l,\mu)$-entry of $\za$ is not zero only if $\l = \l_i$ and $\mu = \l_j$ for
some $i,j$. Consequently, if $x,y\in X$, again from \eqref{a} we see that
$\zb(x,y)\ne 0$ only if $\za(x_{\a,q},y_{\a,q}) \ne 0$, only if
$x_{\a,q} = \l_i$ and $y_{\a,q} = \l_j$ for some $i,j$, only if $x,y\in
X_{\a,\l_1}\cup\cdots\cup X_{\a,\l_n}$. Conversely, assume that
$\zb$ satisfies ($\star$) and let $\l,\mu<\al^{\xi-\a}\bullet\bet$
be such that $\za(\l,\,\mu) \ne 0$. By taking $x =
\al^{\a}\bullet\l$ and $y = \al^{\a}\bullet\mu$, we infer from \eqref{a} that $\zb(x,y) = \d(0,0)\,\za(\l,\mu) =
\za(\l,\mu) \ne 0$, therefore $x\in X_{\a,\l_i}$ for some $i$ and
hence $\l = x_{\a,q} = \l_i$. Thus $\za$ has only a finite number of
nonzero rows and so $\zb\in F_{\a}$. A similar argument shows that if $\zb$ satisfies ($\star\star$), then $\za$ has only a finite number of
nonzero entries and so $\zb\in G_{\a}$.

(3) Suppose that $\a<\b\le\xi$ and let $\za\in Q_\b$. Since $X =
\al^{\a}\bullet\al^{\b-\a}\bullet(\al^{\xi-\b}\bullet\bet)$, it
follows from Proposition \ref{ruleuclidd} that for every $x\in X$
there is a unique $x'<\al^{\b-\a}$ such that
\[
x = \al^{\b}\bullet x_{\b,q} + \al^{\a}\bullet x' + x_{\a,r},
\]
from which
\[
x_{\a,q} = \al^{\b-\a}\bullet x_{\b,q} + x' \quad\text{ and }
x_{\b,r} = \al^{\a}\bullet x' + x_{\a,r}.
\]
As a result, since $\za\in Q_\b$, for every $x,y\in X$ we have the
following equalities:
\begin{align*}
\za(x,y)&= \d(x_{\b,r},y_{\b,r})\,\za(\al^{\b}\bullet x_{\b,q},
           \al^{\b}\bullet y_{\b,q}) \\
          &= \d(x',y')\,\d(x_{\a,r},y_{\a,r})\,
          \za(\al^{\b}\bullet x_{\b,q},
            \al^{\b}\bullet y_{\b,q}) \\
          &= \d(x_{\a,r},y_{\a,r})\,\d(\al^{\a}\bullet x',\al^{\a}\bullet
y')\,
          \za(\al^{\b}\bullet x_{\b,q},
            \al^{\b}\bullet y_{\b,q}) \\
          &= \d(x_{\a,r},y_{\a,r})\,
          \za(\al^{\b}\bullet x_{\b,q}+\al^{\a}\bullet x',
           \al^{\b}\bullet y_{\b,q}+\al^{\a}\bullet y') \\
          &= \d(x_{\a,r},y_{\a,r})\,\za(\al^{\a}\bullet x_{\a,q},
          \al^{\a}\bullet y_{\a,q}).
\end{align*}
Thus \eqref{magicmatrix} holds and hence $\za\in Q_{\a}$.

(4) Assume that $\a_1<\ldots<\a_n<\b\le\xi$ and that there are
non-zero elements $\za_1\in F_{\a_1},\ldots,\za_n\in F_{\a_n},\zb\in
Q_{\b}$ such that
\[
\za_1 + \cdots + \za_n = \zb.
\]
If $x_0,y_0\in X$ are such that $\zb(x_0,y_0) \ne 0$, then there are
$Y,Z\in\CP_\b$ such that $x_0\in Y$, $y_0\in Z$ and all the rows of
the block $\zb(Y,Z)$ are nonzero. Note that $Y$ is the union of a
subset $\CY$ of $\CP_{\a_n}$ of cardinality $\al$, because $\CP_\b$
is $\al$-coarser than $\CP_{\a_n}$; thus, by the above, for each
$U\in\CY$ and each $x\in U$ the $x$-th row of $\zb$ is not zero. On
the other hand the assumptions on $\za_1,\ldots,\za_n$, together
with the previously shown property (2) and the fact that
$\CP_{\a_{i+1}}$ is $\al$-coarser than $\CP_{\a_i}$ for $1\le i<n$,
imply that there are $Y_1,\ldots,Y_k\in\CP_{\a_n}$ such that the
$x$-th row of any $\za_i$ is not zero only if $x\in Y_1\cup\ldots
\cup Y_k$. As a result the $x$-th row of $\zb$ is not zero only if
$x\in Y_1\cup\ldots \cup Y_k$: a contradiction since $|\CY|=\al$ is
infinite.
\end{proof}

\begin{re}\label{elementidempot}
Given $\a<\xi$, we have the set $\{\ze_{\l}\mid
\l\in\al^{\xi-\a}\bullet\bet\}$ of pairwise orthogonal idempotents
which generates $\FR_{\al^{\xi-\a}\bullet\bet}(D)$ as a right ideal
of $\CFM_{\al^{\xi-\a}\bullet\bet}(D)$; each $\ze_{\l}$
gene\-rates $\FR_{\al^{\xi-\a}\bullet\bet}(D)$ as a (two-sided)
ideal. As a result, because of the embedding $\f_{\a}$, we have the
set
\[
\{\ze_{X_{\a,\l}} = \f_{\a}(\ze_{\l})\mid
\l\in\al^{\xi-\a}\bullet\bet\} = \{\ze_{Y}\mid Y\in\CP_{\a}\}
\]
of pairwise orthogonal idempotents of the ring $Q_{\a}$. For every $Y\in\CP_{\a}$ we have
the equalities
\begin{gather*}
\ze_{Y}Q_{\a} = \ze_{Y}F_{\a}, \\
F_{\a} = \bigoplus\{\ze_{Z}Q_{\a}\mid Z\in\CP_{\a}\} = F_{\a}\ze_{Y}F_{\a}
\end{gather*}
and, similarly,
\begin{gather*}
Q_{\a}\ze_{Y} = G_{\a}\ze_{Y},  \\
G_{\a} = \bigoplus\{Q_{\a}\ze_{Z}\mid Z\in\CP_{\a}\} = G_{\a}\ze_{Y}G_{\a}
\end{gather*}
(see \eqref{eq FR1}, \eqref{eq FR2}, \eqref{eq FM1} and \eqref{eq FM2}).
\end{re}

\begin{remnotation}\label{elementidempot1}
For every $\a<\xi$, given $Y,Z\in\CP_{\a}$ we shall denote by $\ze_{Y,Z}$ the matrix such that the $(Y,Z)$-block is the unital $\al^{\a}\times\al^{\a}$-matrix, while all other entries are zero. As $X = X_{\a,\l}$ and $Y = X_{\a,\mu}$ for unique $\l,\mu\in\al^{\xi-\a}\bullet\bet$, then $\ze_{Y,Z} = \f_{\a}(\ze_{\l,\mu})$; more explicitly: for every $x,y\in X$
\[
\ze_{Y,Z}(x,y) = \left\{
                   \begin{array}{ll}
                     1, & \hbox{if $x = \al^{\a}\bullet\l+\r$, $x = \al^{\a}\bullet\mu+\r$ for some $\r<\al^{\a}$;} \\
                     0, & \hbox{otherwise.}
                   \end{array}
                 \right.
\]
Each matrix in $G_{\a}$ is a finite sum of matrices of the form $d\ze_{Y,Z} = \ze_{Y,Z}d$, for $d\in D$ and $Y,Z\in\CP_{\a}$.
\end{remnotation}

\section{Representing artinian partially ordered sets over $\CFM_X(D)$.}
\label{sectrepresartin}

Let us call a \emph{polarized (artinian) poset} an ordered pair $(I,I')$, where $I$ is an artinian poset and $I'$ is a lower subset of $I$. However, in order to simplify notation, from now on we shall use the single letter $I$ in order to designate a polarized artinian poset, while the symbol $I'$ will denote the prescribed lower subset of $I$. Starting from a polarized artinian poset $I$, a ring $D$ and an appropriately sized transfinite ordinal $X$, our main objective in the present section is to associate to each element $i\in I$ a (not necessarily unital) $D$-subring $H_i$ of $Q = \CFM_X(D)$, in such a way that $\CH = \{H_i\mid i\in I\}$ is independent as a set of $(D,D)$-submodules of $Q$ and present the following features: if $i$ is a maximal element of $I$, then $H_i$ is isomorphic to $D$; if $i$ is not maximal and belongs to $I'$ (resp. to $I\setminus I'$), then
$H_i$ is isomorphic to $\FR_X(D)$ (resp. to $\FM_X(D)$); moreover $H_i H_j = 0$ if and
only if $i,j$ are not comparable, while both $H_i H_j$ and $H_j H_i$ are nonzero and
are contained in $H_i$ if $i\le j$.

In order to reach this goal we need a preliminary setup, in which Theorem \ref{pro-chainflr} will
play a central role. This setup will concern just artinian posets; polarized artinian posets will enter the scene only after the setup is ready, so that the above rings $H_i$ can be introduced and we are able to prove that they have the above outlined behavior.

\begin{notations}\label{not main}
In what follows \emph{$I$ is a given artinian poset} and, by keeping the notations introduced in the previous sections, we set the following data and further notations:
\begin{itemize}
    \item   $\xi$ is the dual classical Krull dimension of $I$.
    \item $\CM$ is the set of all maximal chains of $I$; we consider the cardinal
$\bet \defug |\CM|$ and we choose a bijection $\chi\mapsto A_{\chi}$
from $\bet$ to $\CM$.
    \item For every $i\in I$, $\CM_i$ is the set of all maximal chains of $I$ which include $i$:
\[
\CM_i \defug \{A\in\CM\mid i\in A\}.
\]
    \item Given $i\in I$, the binary relation $\sim_{i}$ in $\CM_i$
defined by
\[
A\sim_{i} B \quad\text{\em if and only if }\quad A\cap\{\le i\} =
B\cap\{\le i\}
\]
is clearly an equivalence; set $\,\,\BD_i = \CM_i/\!\!\!\sim_{i}$
and note that there is an obvious one to one correspondence between
the elements of $\,\,\BD_i$ and the maximal chains of $\{\le i\}$.
    \item Denoting by $\al$ the first \underbar{infinite} cardinal such that $\al\ge\sup\{|I|,\bet\}$,
we consider the ordinal
\[
X \defug \al^{\xi+1}\bullet\bet\,.
\]
Note that $|X| = \al$ by Proposition \ref{cardaleph}.
    \item $\CP_\a$ is the partition of $X$ defined by \eqref{partition}, for
all $\a\le\xi+1$.
    \item Given $\chi<\bet$, $i\in I$, $\CA\in\BD_i$, we set
\begin{align*}
X_{\chi} &\defug X_{\xi+1,\chi},\quad\text{so
that}\quad\{X_{\chi}\mid
                 \chi<\bet\}= \CP_{\xi+1};\\
\bet_{\CA} &\defug \{\chi<\bet\mid A_{\chi}\in\CA\}; \\
X_{\CA} &\defug \bigcup\{X_{\chi}\mid \chi\in\bet_{\CA}\}=
\{\al^{\xi+1}\bullet\chi+\t\mid\chi\in\bet_{\CA},
\t<\al^{\xi+1}\}; \\
X_i &\defug \bigcup\{X_{\CA}\mid \CA\in\BD_i\} =
\bigcup\{X_{\chi}\mid A_{\chi}\in\CM_i\}.
\end{align*}
Note that, since $\l(i)<\xi+1$, every $X_{\chi}$ is a disjoint union
of $\al$ members of $\CP_{\l(i)}$, each of which has the form
$X_{\l(i),\l}$ for a unique $\l\in\al^{\xi+1-\l(i)}\bullet\bet$ (see
Lemma \ref{lpartpart}). Set
\begin{align*}
\L_{\CA} &\defug \{\l<\al^{\xi+1-\l(i)}\bullet\bet\mid
X_{\l(i),\l}\sbs X_{\CA}\}; \\
\CQ_{\CA} &\defug \{X_{\l(i),\l}\mid \l\in\L_{\CA}\} = \{Y\in\CP_{\l(i)}\mid Y\sbs X_{\CA}\}.
\end{align*}
As a consequence $|\L_{\CA}| = |\CQ_{\CA}| = \al$ and so
$\CQ_{\CA}\in\BP_{\al}(X_{\CA})$; moreover
\begin{equation}\label{prelem}
    X_{\CA} = \{\al^{\l(i)}\bullet\l+\r\mid\l\in\L_{\CA},\r<\al^{\l(i)}\} = \bigcup\CQ_{\CA}.
\end{equation}
\end{itemize}
\end{notations}

\begin{lem}\label{lemlem}
Given $i\in I$ and $\CA\in\BD_i$, with the above notations we have
\begin{equation}\label{llem}
  \L_{\CA} = \{\al^{\xi+1-\l(i)}\bullet\chi+\s\mid
\chi\in\bet_{\CA},\s<\al^{\xi+1-\l(i)}\}.
\end{equation}
\end{lem}
\begin{proof}
Let $\l = \al^{\xi+1-\l(i)}\bullet\chi+\s$ for some $\chi<\bet$ and
$\s<\al^{\xi+1-\l(i)}$. Then it follows from \eqref{partpart} that
$X_{\l(i),\l}\sbs X_{\xi+1,\chi} = X_{\chi}$. Consequently
$X_{\l(i),\l}\sbs X_{\CA}$ if and only if $X_{\chi}\sbs X_{\CA}$,
namely $\l\in\L_{\CA}$ if and only if $\chi\in\bet_{\CA}$.
\end{proof}

\begin{pro}\label{pro max}
Two elements $i,j\in I$ are comparable if and only if $X_{i}\cap X_{j}\ne\emptyset$. Consequently, if every maximal chain of $I$ is bounded by a maximal element and $\ZM(I)$ denotes the set of all maximal elements of $I$, then the set $\{X_{m}\mid m\in\ZM(I)\}$ is a partition of $X$.
\end{pro}
\begin{proof}
First note that $X_{i}\cap X_{j}\ne\emptyset$ if and only if there is $\chi\in\bet$ such that $X_{\chi}\sbs X_{i}\cap X_{j}$, if and only if there is $\chi\in\bet$ such that $A_{\chi}\in \CM_{i}\cap \CM_{j}$. By the Hausdorff Maximal Principle the latter condition holds if and only if $i$ and $j$ are comparable. Assume now that every maximal chain of $I$ is bounded by a maximal element. Given $\chi\in\bet$, there is $m\in\ZM(I)$ such that $m\in A_{\chi}$; hence $A_{\chi}\in\CM_{m}$ and so $X_{\chi}\sbs X_{m}$. Since the sets $X_{\chi}$ are the members of the partition $\CP_{\xi+1}$ of $X$ and each $X_{m}$ is a union of such sets, the last statement of the proposition follows from the above proven first statement.
\end{proof}

Given $i\in I$ and $\CA, \CA'\in\BD_i$, let us choose $A'\in\CA'$
and let us consider the map
\[
\lmap{f_{\CA'\CA}}{\CA}{\CA'}
\]
defined by
\[
f_{\CA'\CA}(A) = (A'\cap\{\le i\})\cup(A\cap\{i\le\}).
\]
Since $A'\cap\{\le i\} = A''\cap\{\le i\}$ for all $A',A''\in\CA'$,
we see that $f_{\CA'\CA}$ does not depend on the choice of the chain
$A'\in\CA'$. Straightforward computations show that
\begin{equation}\label{ro}
\text{ for all } \CA,\CA',\CA''\in\BD_i\,\,, \quad f_{\CA\CA} =
1_{\CA}\quad\text{and}\quad f_{\CA''\CA'}\,f_{\CA'\CA} =
f_{\CA''\CA};
\end{equation}
in particular each $f_{\CA'\CA}$ is a bijection. Observe that
$f_{\CA'\CA}$ induces the bijection
\[
\lmap{g_{\CA'\CA}}{\bet_{\CA}}{\bet_{\CA'}}
\]
defined as follows: if $\chi\in\bet_{\CA}$, then $g_{\CA'\CA}(\chi)$
is the unique element of $\bet_{\CA'}$ such that
$A_{g_{\CA'\CA}(\chi)} = f_{\CA'\CA}(A_{\chi})$. It follows
immediately from \eqref{ro} that
\begin{equation}\label{roo}
\text{ for all } \CA,\CA',\CA''\in\BD_i \quad g_{\CA\CA} =
1_{\bet_{\CA}}\quad\text{and}\quad g_{\CA''\CA'}\,g_{\CA'\CA} =
g_{\CA''\CA}.
\end{equation}

\begin{lem}\label{lem-partition2}
With the above notations, if $i,j\in I$, then the following hold:
\begin{enumerate}
\item Given $\CA\in\BD_i$ and $\CB\in\BD_j$, if $\CA\cap\CB \ne \emptyset$, then either $\CB\sbs\CA$ or $\CA\sbs\CB$.
\item If $i$ and $j$ are not comparable, then $\CA\cap\CB =
\emptyset$ whenever $\CA\in\BD_i$ and $\CB\in\BD_j$.
\item If $i<j$, then every $\CA\in \BD_i$ contains some $\CB\in \BD_j$. Moreover, if $\CB'\in \BD_j$ and $\CA\cap\CB' \ne \emptyset$, then $\CB'\sbs\CA$.
\item  Assume that $i<j$. If $\CA, \CA'\in\BD_i$, $\CB\in\BD_j$ and $\CB\sbs\CA$,
then $f_{\CA'\CA}(\CB) \in\BD_j$ and, by setting $\CB' =
f_{\CA'\CA}(\CB)$, for all $B\in\CB$ we have
\[
f_{\CA'\CA}(B) = f_{\CB'\CB}(B).
\]
Consequently $g_{\CA'\CA}(\chi) = g_{\CB'\CB}(\chi)$ for all
$\chi\in\bet_{\CB}\sbs \bet_{\CA}$.
\end{enumerate}
\end{lem}
\begin{proof}
(1) and (2). Let $A\in\CA\cap\CB$. Then $i,j\in A$, say $i\le j$. If
$B\in\CB$, that is $B\sim_{j}A$, then necessarily $i\in B$ and
$B\sim_{i}A$. This shows that $\CB\sbs\CA$. Similarly $j\le i$
implies $\CA\sbs\CB$.

(3) Suppose that $i<j$ and let $\CA\in\BD_i$. Given $A\in\CA$, by the Hausdorff's Maximal Principle there is some $B\in\CM$ such that $(A\cap\{\le i\})\cup\{j\}\sbs B$; since $B\sim_{i}A$, then $B\in\CA$. If $\CB$ is the unique
element of $\BD_j$ such that $B\in\CB$, then $\CB\sbs\CA$. Next, let $\CB'\in \BD_j$ and assume that there is some $A\in \CA\cap\CB'$. Then for every $B\in\CB'$ we have that $B\sim_{j}A$ and, since $i\in A$, we infer that $B\sim_{i}A$ as well and therefore $B\in\CA$, proving that $\CB'\sbs\CA$.

(4) Suppose that $i<j$ and let $\CA, \CA'\in\BD_i$, $\CB\in\BD_j$ be
such that $\CB\sbs\CA$. By the definition of $f_{\CA'\CA}$ it is
clear that $f_{\CA'\CA}(\CB)\sbs\CB'$ for some $\CB'\in \BD_j$ and,
according to (1), we must have $\CB'\sbs\CA'$. Similarly, there is
$\CB''\in\BD_j$ such that $f_{\CA\CA'}(\CB')\sbs\CB''\sbs\CA$. On
the other hand we have
\[
\CB = f_{\CA\CA'}(f_{\CA'\CA}(\CB)) \sbs f_{\CA\CA'}(\CB')
\sbs\CB'';
\]
this forces $\CB = \CB''$ and consequently $f_{\CA'\CA}(\CB) =
\CB'$. Finally, choose any $B'\in\CB'$. If $B\in\CB$, it follows
from the above that $B\cap[i,j] = B'\cap[i,j]$, therefore
\[
f_{\CB'\CB}(B) = (B'\cap\{\le j\})\cup(B\cap\{j\le\}) = (B'\cap\{\le
i\})\cup(B\cap\{i\le\}) = f_{\CA'\CA}(B),
\]
as wanted.
\end{proof}

\begin{re}\label{remarklem-partition2}
Let $i,j\in I$ be such that $i<j$ and, according to Lemma \ref{lem-partition2},
take $\CA\in\BD_i$, $\CB\in\BD_j$ such that $\CB\sbs\CA$. If $Y\in\CQ_{\CB}$,
then $Y$ is the union of $\al$ elements of $\CQ_{\CA}$. In fact, since $\l(i)<\l(j)$
and $\CQ_{\CB}\sbs\CP_{\l(j)}$, then $Y$ is the union of $\al$ elements of $\CP_{\l(i)}$.
But if $Z\in \CP_{\l(i)}$ and $Z\sbs Y$, then $Z\sbs Y\sbs X_{\CB}\sbs X_{\CA}$ and so $Z\in\CQ_{\CA}$.
\end{re}

The next step toward our construction is to define, for every $i\in
I$, appropriate families of bijections
\[
(\lmap{t_{\CA'\CA}}{X_{\CA}}{X_{\CA'}})_{\CA, \CA'\in\BD_i}
\quad\text{ and }\quad (\lmap{t_{\CA}}{X_{\CA}}{X})_{\CA\in\BD_i}
\]
such that
\begin{equation}\label{tCA}
t_{\CA''\CA} = t_{\CA''\CA'}\,t_{\CA'\CA},\quad t_{\CA\CA} =
1_{X_{\CA}}\quad\text{ and }\quad t_{\CA} = t_{\CA'}\,t_{\CA'\CA}
\end{equation}
for all $\CA, \CA', \CA''\in\BD_i$. First observe that, for any
$\CA\in\BD_i$, by the definition of $X_{\CA}$ we have $x\in X_{\CA}$
if and only if $x_{\xi+1,q}\in\bet_{\CA}$ (see Notations
\ref{notquorem}). Thus, given $\CA, \CA'\in\BD_i$, for every $x\in
X_{\CA}$ we can define
\[
t_{\CA'\CA}(x) \defug \al^{\xi+1}\bullet
g_{\CA'\CA}(x_{\xi+1,q})+x_{\xi+1,r},
\]
noting that the second member actually belongs to $X_{\CA'}$.
Straightforward computations with the use of \eqref{roo} show that
the first two equalities of \eqref{tCA} hold for every
$\CA,\CA',\CA''\in\BD_i$ and so each $t_{\CA'\CA}$ is a bijection.
It is clear that $t_{\CA'\CA}$ restricts to a bijection from
$X_{\chi}$ to $X_{g_{\CA'\CA}(\chi)}$ for all $\chi\in\bet_{\CA}$;
moreover from (4) of Lemma \ref{lem-partition2} we obtain the
following corollary.

\begin{cor}\label{lem-partition3}
Assume that $i<j$. If $\CA,\CA'\in\BD_i$, $\CB\in\BD_j$ and
$\CB\sbs\CA$, by setting $\CB' = f_{\CA'\CA}(\CB)$, for every $x\in
X_{\CB}$ we have
\[
t_{\CA'\CA}(x) = t_{\CB'\CB}(x).
\]
\end{cor}

Next, given $i\in I$ and $\CA,\CA'\in\BD_i$, let us consider the bijection
\[
\lmap{k_{\CA'\CA}}{\L_{\CA}}{\L_{\CA'}}
\]
defined by
\[
k_{\CA'\CA}(\al^{\xi+1-\l(i)}\bullet\chi+\s) =
\al^{\xi+1-\l(i)}\bullet g_{\CA'\CA}(\chi)+\s
\]
for all $\chi\in\bet_{\CA}$ and $\s<\al^{\xi+1-\l(i)}$ (see Lemma
\ref{lemlem}). Again from \eqref{roo} we infer that
\begin{equation}\label{rooo}
\text{ for all } \CA,\CA',\CA''\in\BD_i \quad k_{\CA\CA} =
1_{\L_{\CA}}\quad\text{and}\quad k_{\CA''\CA'}\,k_{\CA'\CA} =
k_{\CA''\CA}.
\end{equation}

Now, let us choose an equivalence class $\CA_{i}\in\BD_i$ and a bijection
\[
\lmap{k_{\CA_{i}}}{\L_{\CA_{i}}}{\al^{\xi+1-\l(i)}\bullet\bet}
\]
(this can be done since both $\L_{\CA_{i}}$ and
$\al^{\xi+1-\l(i)}\bullet\bet$ have cardinality $\al$ by Lemma
\ref{lemlem} and Proposition \ref{cardaleph}) and, for each
$\CA\in\BD_i$, let us consider the bijection
\[
\lmap{k_{\CA}\defug
k_{\CA_{i}}\,k_{\CA_{i}\CA}}{\L_{\CA}}{\al^{\xi+1-\l(i)}\bullet\bet},
\]
namely
\[
k_{\CA}(\al^{\xi+1-\l(i)}\bullet\chi+\s) =
k_{\CA_{i}}(\al^{\xi+1-\l(i)}\bullet g_{\CA_{i}\CA}(\chi)+\s)
\]
for all $\chi\in\bet_{\CA}$ and $\s<\al^{\xi+1-\l(i)}$ (see again
Lemma \ref{lemlem}). Finally, let us define the map
$\map{t_{\CA}}{X_{\CA}}{X}$ by setting
\[
t_{\CA}(\al^{\l(i)}\bullet\l+\r) = \al^{\l(i)}\bullet k_{\CA}(\l)+\r
\]
for every $\l\in\L_{\CA}$ and $\r<\al^{\l(i)}$ (see \eqref{prelem}).
Using Proposition \ref{ruleuclidd} and the fact that $k_{\CA}$ is a
bijection it is easy to see that $t_{\CA}$ is a bijection. We claim
that
\[
t_{\CA} = t_{\CA_{i}}t_{\CA_{i}\CA}.
\]
Indeed, taking \eqref{prelem} and Lemma \ref{lemlem} into account, let
$\chi\in\bet_{\CA}$, $\s<\al^{\xi+1-\l(i)}$, $\r<\al^{\l(i)}$ and
consider $\l = \al^{\xi+1-\l(i)}\bullet\chi+\s$. Then we have:
\begin{align*}
t_{\CA}\left(\al^{\l(i)}\bullet\l+\r\right)
        &= \al^{\l(i)}\bullet k_{\CA}(\l)+\r \\
        &= \al^{\l(i)}\bullet k_{\CA_{i}}\left(\al^{\xi+1-\l(i)}\bullet g_{\CA_{i}\CA}(\chi)+\s\right)+\r \\
        &= t_{\CA_{i}}\left(\al^{\l(i)}\bullet\left(\al^{\xi+1-\l(i)}\bullet
                      g_{\CA_{i}\CA}(\chi)+\s\right)+\r\right) \\
        &= t_{\CA_{i}}\left(\al^{\xi+1}\bullet
                      g_{\CA_{i}\CA}(\chi)+\al^{\l(i)}\bullet\s+\r\right) \\
        &= t_{\CA_{i}}t_{\CA_{i}\CA}\left(\al^{\xi+1}\bullet\chi+\al^{\l(i)}\bullet\s+\r\right) \\
        &= t_{\CA_{i}}t_{\CA_{i}\CA}\left(\al^{\l(i)}\bullet\al^{\xi+1-\l(i)}
                        \bullet\chi+\al^{\l(i)}\bullet\s+\r\right) \\
        &= t_{\CA_{i}}t_{\CA_{i}\CA}\left(\al^{\l(i)}\bullet\l+\r\right),
\end{align*}
proving our claim. Now, let $\CA,\CA'\in\BD_{i}$. Since the first
two equalities of \eqref{tCA} hold for every
$\CA,\CA',\CA''\in\BD_i$, from $t_{\CA'} =
t_{\CA_{i}}t_{\CA_{i}\CA'}$ we infer that $t_{\CA_{i}} =
t_{\CA'}t_{\CA'\CA_{i}}$; consequently
\[
t_{\CA} = t_{\CA_{i}}t_{\CA_{i}\CA} =
t_{\CA'}t_{\CA'\CA_{i}}t_{\CA_{i}\CA} = t_{\CA'}t_{\CA'\CA}
\]
and therefore the third equality of \eqref{tCA} holds for all $\CA,
\CA'\in\BD_{i}$.

\begin{re}\label{remarkt}
    Because of the definition of $t_{\CA}$, the assignment
\[
X_{\l(i),\l}\mapsto t_{\CA}(X_{\l(i),\l}) = X_{\l(i),k_{\CA}(\l)}
\]
for $\l\in\L_{\CA}$ defines a bijection from $\CQ_{\CA}$ to
$\CP_{\l(i)} = \{X_{\l(i),\l}\mid\l<\al^{\xi+1-\l(i)}\bullet\bet\}$.
Consequently, by \eqref{tCA} the assignment
\[
X_{\l(i),\l}\mapsto t_{\CA'\CA}(X_{\l(i),\l})
\]
gives a bijection from $\CQ_{\CA}$ to $\CQ_{\CA'}$.
\end{re}

As in Section 1, for a given ring $D$ let us consider the ring $Q =
\CFM_{X}(D)$ and, for each $\a<\xi+1$, let $Q_{\a}$ be the subring
of $Q$ consisting of those matrices $\za$ satisfying
\eqref{magicmatrix}. For each $i\in I$ let us denote by $S_i$ the
subset of $Q$ of those matrices $\za$ such that
\begin{equation}\label{mamagicmatrix}
  \ze_{X_{\CA}}\za = \ze_{X_{\CA}}\za\ze_{X_{\CA}} = \za\ze_{X_{\CA}}
  \quad\text{ for all }\CA\in\BD_{i}
\end{equation}
and, if $\CA, \CA'\in\BD_i$, then
\[
\za(x,y) = \za(t_{\CA'\CA}(x),t_{\CA'\CA}(y))
\]
for all $x,y\in X_{\CA}$. Roughly speaking, $S_i$ consists of those
matrices which have zero entries outside the
$(X_{\CA},X_{\CA})$-blocks for $\CA\in\BD_i$ (which are mutually
disjoint) and, if $\CA,\CA'\in\BD_i$, the
$(X_{\CA'},X_{\CA'})$-block coincides with the
$(X_{\CA},X_{\CA})$-block ``\,up to the bijection $t_{\CA'\CA}$\,''. As
we are going to see, if we consider the idempotent diagonal matrix
$\ze_{X_i}$, then $S_i$ is actually a unital $D$-subring of
$\ze_{X_i}Q\ze_{X_i}$ isomorphic to $Q$.

\begin{pro}\label{subrings}
With the above notations, for every $i\in I$ there is a unital
$D$-linear ring mono\-morphism
$\map{\psi_i}{Q}{\ze_{X_i}Q\ze_{X_i}}$ such that
\begin{equation}\label{sbrs}
\psi_i(Q) = S_i
\end{equation}
and
\begin{equation}\label{sbrss}
\psi_i(Q_\a) \sbs S_i\cap Q_\a \quad\text{ for all }\a\le\l(i).
\end{equation}
Moreover, for every $i,j\in I$ the following properties hold:
\begin{enumerate}
\item $S_iS_j = 0$ if and only if $i,j$ are not comparable. \item
If $i\le j$, then $S_iS_j\cup S_jS_i\sbs S_i$.
\end{enumerate}
\end{pro}
\begin{proof}
Given $i\in I$, let us define the map
\[
\lmap{\psi_i}{Q}{\ze_{X_i}Q\ze_{X_i}}
\]
as follows: given $\za\in Q$, for every $x,y\in X$
\[
\psi_i(\za)(x,y) =
\begin{cases}\za(t_{\CA}(x),t_{\CA}(y))\text{ if }x,y\in X_{\CA}\text{ for
some }\CA\in\BD_i, \\
0\text{ otherwise }.
\end{cases}
\]
It is clear that $\psi_i$ is an homomorphism of $(D,D)$-bimodules
and, by using \eqref{tCA}, we see easily that $\psi_i(Q) \sbs S_i$.
Let $\za\in Q$ and assume that $\za(u,v)\ne 0$ for some $u,v\in X$.
Given $\CA\in\BD_i$, we have $\psi_i(\za)(x,y)\ne 0$ for $x =
t_{\CA}^{-1}(u)$ and $y = t_{\CA}^{-1}(v)$; this shows that $\psi_i$
is a monomorphism. Next, let $\za, \zb\in Q$ and $x,y\in X$. If
$x,y\in X_{\CA}$ for some $\CA\in\BD_i$, using the fact that
$t_{\CA}$ is a bijection and recalling that the subsets $X_{\CA}$
are mutually disjoint for $\CA$ ranging in $\BD_i$ we get the
following:
\begin{align*}
\psi_i(\za\zb)(x,y) &= (\za\zb)(t_{\CA}(x),t_{\CA}(y)) =
      \sum_{u\in X}[(\za)(t_{\CA}(x),u)]\,[(\zb)(u,t_{\CA}(y))] \\
&= \sum_{z\in
X_{\CA}}[(\za)(t_{\CA}(x),t_{\CA}(z))]\,[(\zb)(t_{\CA}(z),
      t_{\CA}(y))] \\
&= \sum_{z\in X_{\CA}}[\psi_i(\za)(x,z)]\,[\psi_i(\zb)(z,y)] \\
&= \sum_{z\in X}[\psi_i(\za)(x,z)]\,[\psi_i(\zb)(z,y)] \\
&= \left(\psi_i(\za)\psi_i(\zb)\right)(x,y).
\end{align*}
If there is no $\CA\in\BD_i$ such that $x,y\in X_{\CA}$, through the
same guidelines we obtain that
\[
\left(\psi_i(\za)\psi_i(\zb)\right)(x,y) = \sum_{z\in
X}[\psi_i(\za)(x,z)]\,[\psi_i(\zb)(z,y)] = 0 = \psi_i(\za\zb)(x,y).
\]
Since $\psi_i(1) = \ze_{X_i}$, we conclude that $\psi_i$ is a unital
ring homomorphism. Finally, let $\zc\in S_i$ and define the matrix
$\za\in Q$ as follows: choose any $\CA\in\BD_i$ and, for every
$u,v\in X$, set
\[
\za(u,v) = \zc(t_{\CA}^{-1}(u),t_{\CA}^{-1}(v)).
\]
Using again \eqref{tCA} it is immediate to check that $\psi_i(\za) =
\zc$ and thus $\psi_i(Q) = S_i$.

In order to establish \eqref{sbrss}, given any $\a\le\l(i)$ and
$\zb\in Q_\a$, we must show that the matrix $\za = \psi_i(\zb)$
satisfies \eqref{magicmatrix}. First observe that, given any $x\in
X$, both $x$ and $\al^{\a}\bullet x_{\a,q}$ belong to the same
member $X_{\a,x_{\a,q}}$ of the partition $\CP_{\a}$; on the other hand, given
$\CA\in\BD_i$, since $\CP_{\xi+1}$ is coarser than $\CP_{\a}$ and
$X_{\CA}$ is a union of members of $\CP_{\xi+1}$, we have that either
$X_{\a,x_{\a,q}}\sbs X_{\CA}$ or $X_{\a,x_{\a,q}}\cap X_{\CA} =
\emptyset$. We infer that $x\in X_{\CA}$ if and only if
$\al^{\a}\bullet x_{\a,q}\in X_{\CA}$. Accordingly, given $x,y\in
X$, if there is no $\CA\in\BD_i$ such that $x,y\in X_{\CA}$, then
both members of the equality in \eqref{magicmatrix} are zero. Assume
that $x,y\in X_{\CA}$ for some $\CA\in\BD_i$ and note that,
according to Proposition \ref{ruleuclidd}, we have the
decompositions
\begin{gather*}
x = \al^{\a}\bullet\al^{\l(i)-\a}\bullet\al^{\xi+1-\l(i)} \bullet
x_1 + \al^{\a}\bullet\al^{\l(i)-\a}\bullet x_2 +
\al^{\a}\bullet x_3 + x_{4} \\
= \al^{\l(i)}\bullet\left(\al^{\xi+1-\l(i)} \bullet x_1 + x_2\right)
+ \al^{\a}\bullet x_3 + x_{4}
\end{gather*}
for unique $x_1<\bet$, $x_2<\al^{\xi+1-\l(i)}$,
$x_3<\al^{\l(i)-\a}$, $x_4<\al^{\a}$. By setting $x_{5} = \al^{\xi+1-\l(i)} \bullet x_1 + x_2$ and comparing with the
decomposition \eqref{decomp} we see that
\[
x_{\a,q} = \al^{\l(i)-\a}\bullet x_{5} + x_3\quad\text{ and }\quad
x_{\a,r} = x_4.
\]
We observe that $x\in X_{x_1} = X_{\xi+1,x_1}$, therefore $X_{x_1}\sbs X_{\CA}$ and
so $x_1\in\bet_{\CA}$. Consequently, it follows from Lemma \ref{lemlem} that $x_{5}\in\L_{\CA}$ and then we may consider the ordinal
\[
x_{6} = \al^{\l(i)-\a}\bullet k_{\CA}(x_{5}) + x_3.
\]
We now obtain that
\begin{align*}
t_{\CA}(x) &= t_{\CA}(\al^{\l(i)}\bullet x_{5} + \al^{\a}\bullet x_3
+
x_{\a,r}) \\
           &=\al^{\l(i)}\bullet k_{\CA}(x_{5}) + \al^{\a}\bullet x_3 +
x_{\a,r} \\
           &= \al^{\a}\bullet x_{6} + x_{\a,r}
\end{align*}
and a similar computation shows that
\[
t_{\CA}(\al^{\a}\bullet x_{\a,q}) = \al^{\a}\bullet x_{6}.
\]
After processing $y$ in the same way, from all above we infer
finally:
\begin{align*}
\za(x,y) &= \zb(t_{\CA}(x),t_{\CA}(y)) \\
         &= \zb(\al^{\a}\bullet x_{6}+x_{\a,r},
         \,\al^{\a}\bullet y_{6}+y_{\a,r}) \\
         &= \d(x_{\a,r},y_{\a,r})\,\zb(\al^{\a}\bullet x_{6},
         \,\al^{\a}\bullet y_{6}) \\
         &= \d(x_{\a,r},y_{\a,r})\,
         \zb(t_{\CA}(\al^{\a}\bullet x_{\a,q}),
         t_{\CA}(\al^{\a}\bullet y_{\a,q})) \\
         &= \d(x_{\a,r},y_{\a,r})\,
         \za(\al^{\a}\bullet x_{\a,q},
         \al^{\a}\bullet y_{\a,q}).
\end{align*}
This proves that $\za\in Q_{\a}$.

(1) Let $i,j\in I$ and assume that $i,j$ are not comparable. Then,
given $\CA\in\BD_i$ and $\CB\in\BD_j$, we have $\CA\cap\CB =
\emptyset$ by (3) of Lemma \ref{lem-partition2}, therefore $X_i\cap
X_j = \emptyset$. As a consequence, if $\za\in S_i$ and $\zb\in
S_j$, then $\za\zb =
\ze_{X_{i}}\za\ze_{X_{i}}\ze_{X_{j}}\zb\ze_{X_{j}} = 0$. If, on the
contrary, $i\le j$ and $A_{\chi}$ is any maximal chain such that
$i,j\in A_{\chi}$, then $X_{A_{\chi}}\sbs X_{i}\cap X_{j}$ and hence
$X_{i}\cap X_{j}\ne\emptyset$. Consequently $\mathbf{0} \ne
\ze_{X_{i}}\,\ze_{X_{j}} = \ze_{X_{j}}\,\ze_{X_{i}}\in S_iS_j\cap
S_jS_i$.

(2) Suppose that $i<j$, let $\za\in S_i$, $\zb\in S_j$ and assume
that $0\ne (\za\zb)(x,y) = \sum_{z\in X}\za(x,z)\zb(z,y)$ for some
$x,y\in X$. Then $\za(x,z) \ne 0 \ne \zb(z,y)$ for some $z\in X$ and
therefore $x,z\in X_{\CA}$, $z,y\in X_{\CB}$ for some $\CA\in\BD_i$,
$\CB\in\BD_j$; necessarily $\CB\sbs\CA$ in view of property (2) of
Lemma \ref{lem-partition2} and this shows that the matrix $\za\zb$
has zero entries outside the $(X_{\CA},X_{\CA})$-blocks for
$\CA\in\BD_i$. Suppose that $\CA,\CA'\in\BD_i$ and let us prove that
\begin{equation}\label{eqblockss}
(\za\zb)(x,y) = (\za\zb)(t_{\CA'\CA}(x),t_{\CA'\CA}(y))
\end{equation}
for all $x,y\in X_{\CA}$. By using \eqref{tCA} and (4) of Lemma
\ref{lem-partition2}, we see that there is no $\CB\in\BD_j$ such
that $y\in X_{\CB}$ if and only if there is no $\CB'\in\BD_j$ such
that $t_{\CA'\CA}(y)\in X_{\CB'}$; if it is the case, since $\zb\in
S_j$, both members of \eqref{eqblockss} are zero. Otherwise there is
$\CB\in\BD_j$ such that $y\in X_{\CB}$; necessarily $\CB\sbs\CA$ by
Lemma \ref{lem-partition2} and, by setting $\CB' = f_{\CA'\CA}(\CB)$
and using Corollary \ref{lem-partition3}, we may compute as follows:
\begin{align*}
(\za\zb)(x,y) &= \sum_{z\in X_{\CA}}\za(x,z)\,\zb(z,y) =
\sum_{z\in X_{\CB}}\za(x,z)\,\zb(z,y) \\
&= \sum_{z\in X_{\CB}} [\za(t_{\CA'\CA}(x),t_{\CA'\CA}(z))]\,
[\zb(t_{\CB'\CB}(z),t_{\CB'\CB}(y))] \\
&= \sum_{u\in X_{\CB'}}
[\za(t_{\CA'\CA}(x),u)]\,[\zb(u,t_{\CB'\CB}(y))] \\
&= \sum_{u\in X_{\CA'}}
[\za(t_{\CA'\CA}(x),u)]\,[\zb(u,t_{\CA'\CA}(y))] \\
&= (\za\zb)(t_{\CA'\CA}(x),t_{\CA'\CA}(y)).
\end{align*}
Thus \eqref{eqblockss} holds for all $x,y\in X_{\CA}$, showing that
$\za\zb\in S_i$. The proof that $\zb\za\in S_i$ is similar.
\end{proof}

We are now in a position to associate to a given polarized artinian poset $I$ the set $\CH = \{H_i\mid i\in I\}$ of (possibly non-unital) subrings of $Q$, satisfying the conditions we outlined at the beginning of the present section. For every $i\in I$ let us define
the $D$-subring $H_i$ of $Q$ as follows:
\[
H_i = \left\{
        \begin{array}{ll}
          \psi_i\left(F_{\l(i)}\right), & \hbox{if $i$
\underbar{is not} a maximal element of $I$ and $i\in I'$;} \\
          \psi_i\left(G_{\l(i)}\right), & \hbox{if $i$
\underbar{is not} a maximal element of $I$ and $i\nin I'$;} \\
          \psi_i\left(D\right) = \ze_{X_{i}}D, & \hbox{if $i$ \underbar{is} a maximal
element of $I$.}
        \end{array}
      \right.
 \]
(for each ordinal $\a\le\xi$, the non-unital $D$-subrings $F_{\a}$ and $G_{\a}$ of $Q$
are defined in (2) of Theorem \ref{pro-chainflr}). Of course $H_i\ne
H_j$ if $i\ne j$; also note that, apart from the trivial case in
which $I$ is a singleton, $H_i$ is not a unital subring of $Q$. It
is clear that $H_i$ has a multiplicative identity, given by
$\ze_{X_i}$, if and only if $i$ is a maximal element of $I$.

Given $i\in I$, we know that $G_{\l(i)}$ contains the set
$\{\ze_{Y}\mid Y\in\CP_{\l(i)}\}$ of pairwise orthogonal idempotents
which generate $F_{\l(i)}$ as a right ideal and $G_{\l(i)}$ as a left ideal of $Q_{\l(i)}$ (Remark
\ref{elementidempot}); the images of these idempotents,
under the action of the imbedding $\psi_i$, will be relevant in
order to analyze the features of the subrings $H_i$ and the way they
interact each other. Firstly we need to introduce two additional notations.

\begin{notations}\label{not varie}
Given $i\in I$,
$\CA\in\BD_{i}$, $V\in\CQ_{\CA}$ and $Y\in\CP_{\l(i)}$, we define the following subsets of $X_{i}$ (see Remark \ref{remarkt}):
\begin{gather*}
\overline{V} \defug
\bigcup\left\{t_{\CA'\CA}(V)\mid\CA'\in\BD_{i}\right\}, \\
Y(i) \defug
\bigcup\left\{t_{\CA'}^{-1}(Y)\mid\CA'\in\BD_{i}\right\}.
\end{gather*}
\end{notations}
Clearly $V = t_{\CA\CA}(V)\sbs\overline{V}$; moreover it follows from \eqref{tCA} that
\begin{equation}\label{overlined}
   \overline{V} =
\overline{t_{\CA'\CA}(V)} = \left(t_{\CA}(V)\right)(i) \quad\text{
for all $\CA,\CA'\in\BD_{i}$ and $V\in\CQ_{\CA}$},
\end{equation}
while
\begin{equation}\label{overlinedd}
    Y(i) = \overline{t_{\CA}^{-1}(Y)} = \overline{t_{\CA'}^{-1}(Y)}\quad \text{for all
 $\CA,\CA'\in\BD_{i}$ and $Y\in\CP_{\l(i)}$}.
\end{equation}
As a consequence we have the equalities
\begin{equation}\label{overlineddd}
 \{Y(i)\mid Y\in\CP_{\l(i)}\} = \{\overline{V}\mid
V\in\CQ_{\CA}\} = \{\overline{W}\mid W\in\CQ_{\CB}\}
\end{equation}
for all $\CA,\CB\in\BD_{i}$.

Due to the definition of $\psi_{i}$, for every
$Y\in\CP_{\l(i)}$ we have
\[
 \psi_{i}\left(\ze_{Y}\right) = \ze_{Y(i)}.
 \]

\begin{lem}\label{lem idempprim Hi}
With the above notations, $\left\{ \psi_{i}\left(\ze_{Y}\right) =
\ze_{Y(i)}\left| Y\in\CP_{\l(i)}\right.\right\}$ is a set of pairwise
orthogonal idempotents of $H_i$ and
\[
\{\ze_{Y(i)}\mid Y\in\CP_{\l(i)}\} = \{\ze_{\overline{V}}\mid
V\in\CQ_{\CA}\} = \{\ze_{\overline{W}}\mid W\in\CQ_{\CB}\}
\]
for every $\CA,\CB\in\BD_i$. Moreover, given $j\in I$, for every $Y\in\CP_{\l(i)}$ and $Z\in\CP_{\l(j)}$ the following hold:
\begin{enumerate}
  \item If $i,j$ are not comparable, then $Y(i)\cap Z(j) = \vu$.
  \item If $i<j$ and $Y(i)\cap Z(j) \ne \vu$, then $Y(i)\sbs Z(j)$.
\end{enumerate}
 \end{lem}
 \begin{proof}
 The first statement is a consequence of \eqref{overlineddd} and the fact that $\psi_{i}$ is injective. Given $j\in I$, assume that $Y(i)\cap Z(j) \ne \vu$. Then there are $\CA\in\BD_i$, $\CB\in\BD_j$ such that $t_{\CA}^{-1}(Y)\cap t_{\CB}^{-1}(Z) \ne \vu$. This implies that $X_{\CA}\cap X_{\CB} \ne \vu$ and hence $\CA\cap\CB\ne\vu$. As a result $i$ and $j$ are comparable by (2) of Lemma \ref{lem-partition2}, say $i<j$. Thus $\CP_{\l(i)}$ is coarser than $\CP_{\l(i)}$ and, since $t_{\CA}^{-1}(Y)\in\CP_{\l(i)}$ and $t_{\CB}^{-1}(Z)\in\CP_{\l(i)}$, we infer that $t_{\CA}^{-1}(Y)\sbs t_{\CB}^{-1}(Z)$. Given any $\CA'\in\BD_i$, by setting $\CB' =
  f_{\CA\CA'}(\CB)$, we have that $\CB'\in\BD_j$ by (4) of Lemma \ref{lem-partition2}. Thus, by using Corollary \ref{lem-partition3} we obtain
  \[
 t_{\CA'}^{-1}(Y) = t_{\CA'\CA}(t_{\CA}^{-1}(Y)) = t_{\CB'\CB}(t_{\CA}^{-1}(Y)) \sbs t_{\CB'\CB}(t_{\CB}^{-1}(Z)) =  t_{\CB'}^{-1}(Z) \sbs Z(j).
 \]
 We conclude that $Y(i) \sbs Z(j)$.
 \end{proof}

 \begin{lem}\label{HsbsFr}
 Assume that $i$ \underbar{is not} a maximal element of $I$. Then
 \begin{equation}\label{H_ii}
     H_{i} = \left\{
               \begin{array}{ll}
                 \bigoplus\left\{\ze_{Y(i)}H_{i}\left|
Y\in\CP_{\l(i)}\right.\right\}, & \hbox{if $i\in I'$;} \\
                 \bigoplus\left\{H_{i}\ze_{Y(i)}\left|
Y\in\CP_{\l(i)}\right.\right\}, & \hbox{if $i\nin I'$}
               \end{array}
             \right.
     \end{equation}
and
     \begin{equation}\label{H_i}
     H_{i} = H_{i}\,\ze_{Y(i)}H_{i}
     \end{equation}
for every $Y\in\CP_{\l(i)}$. Moreover, given $\za\in S_i$, if $i\in I'$, then
 the following conditions are equivalent:
 \begin{enumerate}
     \item $\za\in H_{i}$.
     \item $\ze_{X_{\CA}}\za\,\ze_{X_{\CA}}\in F_{\l(i)}$ for all
     $\CA\in\BD_i$.
     \item There exist $Y_{1},\ldots,Y_{n}\in\CP_{\l(i)}$ such that the
     $x$-th row of $\za$ is not zero only if $x\in
     Y_{1}(i)\cup\cdots\cup Y_{n}(i)$; equivalently
     \[
     \za = \left(\ze_{Y_{1}(i)}+\cdots+\ze_{Y_{n}(i)}\right)\za.
     \]
     \end{enumerate}
If, on the contrary, $i\nin I'$, then the above conditions (1), (2), in which $F_{\l(i)}$ is replaced by $G_{\l(i)}$, are equivalent to the following one:
\begin{enumerate}
\setcounter{enumi}{3}
     \item There exist $Y_{1},\ldots,Y_{n}\in\CP_{\l(i)}$ such that the entry $\za(x,y)$ of $\za$ is not zero only if $x,y\in
     Y_{1}(i)\cup\cdots\cup Y_{n}(i)$; equivalently
     \[
     \za = \left(\ze_{Y_{1}(i)}+\cdots+\ze_{Y_{n}(i)}\right)\za = \za\left(\ze_{Y_{1}(i)}+\cdots+\ze_{Y_{n}(i)}\right).
     \]
     \end{enumerate}
 \end{lem}
 \begin{proof}
 The first statement follows from  Remark \ref{elementidempot}, while the equivalence (1)$\Leftrightarrow$(3) is clear from \eqref{H_ii}. Next, suppose that $i\in I'$, assume (3) and let $\CA\in\BD_i$. Using
\eqref{overlinedd} we see that
\begin{align*}
[Y_{1}(i)\cup\cdots\cup Y_{n}(i)]\cap X_{\CA} &=
\left[\overline{t_{\CA}^{-1}(Y_{1})}\cup\cdots\cup
\overline{t_{\CA}^{-1}(Y_{n})}\right]\cap X_{\CA} \\
         &= t_{\CA}^{-1}(Y_{1})\cup\cdots\cup t_{\CA}^{-1}(Y_{n}).
         \end{align*}
         Since $t_{\CA}^{-1}(Y_{r})\in\CQ_{\CA}\sbs\CP_{\l(i)}$ for all
$r\in\{1,\ldots,n\}$, it follows from (2) of Theorem
\ref{pro-chainflr} that $\ze_{X_{\CA}}\za\,\ze_{X_{\CA}}\in
F_{\l(i)}$. Conversely, suppose (2) and let $x,y\in X$ be such that
$\za(x,y)\ne 0$. Then $x,y\in X_{\CA}$ for some $\CA\in\BD_i$ and,
by the assumption and (2) of Theorem \ref{pro-chainflr}, there are
$V_1,\ldots,V_n\in\CP_{\l(i)}$ such that $x\in V_1\cup\cdots\cup
V_n$; necessarily $V_1,\ldots,V_n\in\CQ_{\l(i)}$, because $x\in
X_{\CA}$. By setting $Y_r = t_{\CA}(V_r)\in\CP_{\l(i)}$ for
$r\in\{1,\ldots,n\}$, we conclude that $x\in Y_{1}(i)\cup\cdots\cup
Y_{n}(i)$, taking \eqref{overlinedd} into account.

The proof of the equivalence (2)$\Leftrightarrow$(3) is similar, by taking again Theorem
\ref{pro-chainflr} into account.
\end{proof}

\begin{re}
    Observe that, in general, we have $\ze_{X_{\CA}}H_i\,\ze_{X_{\CA}}\not\sbs H_i$,
unless $\BD_i = \{\CA\}$. If $i$ \underbar{is not} a maximal element
of $I$, then $H_i\sbs F_{\l(i)}$ if and only if $\BD_i$ is finite,
that is, if and only if $\{\le i\}$ has finitely many maximal
chains.
\end{re}

 \begin{lem}\label{nonzerochain}
 Let $j_1<\cdots<j_n$ be a finite chain of $I$ with $n>1$, let $\za_1\in
\nolinebreak H_{j_1}$,\linebreak\ldots,$\za_n\in H_{j_n}$ and choose
$\CA_1\in\BD_{j_1},\ldots,\CA_n\in\BD_{j_n}$ such that
 $\CA_1\sps\cdots\sps\CA_n$ (see (2) of Lemma
 \ref{lem-partition2}). If $\za_n\ne\mathbf{0}$, then the
$(X_{\CA_n}\times X_{\CA_n})$-block
 of $\za = \za_1+\cdots+\za_n$ is not zero; in particular
 there is some $x\in X_{\CA_n}$ such that the $x$-th row of $\za$
 is not zero and coincides with the $x$-th row of $\za_n$.
 \end{lem}
\begin{proof}
For each $r\in\{1,\ldots,n\}$ let us denote by $Y_r$ the subset of
those $u\in X_{\CA_r}$ such that the $u$-th row  of $\za_r$ is not
zero. For $r<n$ the element $j_r$ is not maximal, therefore it
follows from (2) of Theorem \ref{pro-chainflr} and Lemma
\ref{HsbsFr} that $Y_r$ is the (disjoint) union of finitely many
elements of $\CQ_{\CA_r}\sbs\CP_{\l(j_r)}$. Assume that
$\za_n\ne\mathbf{0}$. Then for every $\CA\in\BD_{j_n}$ the
$(X_{\CA}\times X_{\CA})$-block of $\za_n$ is not zero and hence, in
particular, $Y_n\ne\emptyset$. Since $\CP_{\l(j_s)}$ is
$\al$-coarser than $\CP_{\l(j_r)}$ when $r<s\le n$, in particular
$Y_n$ is the (disjoint) union of $\al$ elements of
$\CP_{\l(j_{n-1})}$; on the other hand, by the above
$Y_1\cup\cdots\cup Y_{n-1}$ is contained in the (disjoint) union of
finitely many elements of $\CP_{\l(j_{n-1})}$. Consequently
\[
Y_n\setminus(Y_1\cup\cdots\cup Y_{n-1})\ne\emptyset
\]
and therefore, if $x\in Y_n\setminus(Y_1\cup\cdots\cup Y_{n-1})$,
the $x$-th row of $\za$ coincides with the $x$-th row of $\za_n$,
which is not zero.
\end{proof}

 \begin{lem}\label{nonzerochain2}
 Let $J$ be a finite subset of $I$, let $\za = \sum_{j\in
J}\za_j$, where $\za_j\in H_j$ for $j\in J$, let $j_1<\cdots<j_n$ be
a maximal chain of $J$ and let
$\CA_1\in\BD_{j_1},\ldots,\CA_n\in\BD_{j_n}$ be as in Lemma
 \ref{nonzerochain}. Then the $(X_{\CA_n}\times X_{\CA_n})$-blocks of
$\za$ and $\za_{j_1}+\cdots+\za_{j_n}$ coincide.
 \end{lem}
\begin{proof}
 First note that, given $j\in J$, if there is some $\CB\in\BD_j$ such
that $X_{\CA_n}\cap X_{\CB}\ne\emptyset$, namely $\CA_n\cap
\CB\ne\emptyset$, then $\CA_r\cap\CB\ne\emptyset$ for all
$r\in\{1,\ldots,n\}$ and therefore it follows from Lemma
\ref{lem-partition2} that $j$ is comparable with every $j_r$. As a
result $j\in\{j_1,\ldots,j_n\}$, because this latter is a maximal
chain of $J$. This implies that if $j\in
J\setminus\{j_1,\ldots,j_n\}$, then the $(X_{\CA_n}\times
X_{\CA_n})$-block of every matrix in $H_{j}$ is zero and,
consequently, the $(X_{\CA_n}\times X_{\CA_n})$-blocks of $\za$ and
$\za_{j_1}+\cdots+\za_{j_n}$ coincide.
\end{proof}

\begin{theo}\label{lemposetring}
Let $I$ be a polarized artinian poset having at least two elements. With the above notations, $\CH = \{H_i\mid i\in I\}$ is an independent set of
$(D,D)$-submodules of $Q$ which satisfy the following conditions:
\begin{enumerate}
\item Every $H_i$ is a non-unital subring of $Q$; it has an
identity, given by $\ze_{X_i}$, if and only if $i$ is a maximal
element of $I$. \item $H_iH_j = 0$ if $i,j$ are not comparable;
\item Given $i\in I$, if $J\sbs\{i\le\}$ and
$\mathbf{0}\ne\za\in\bigoplus_{j\in J}H_{j}$, then
\[
0\ne H_i\za\sbs H_i \quad\text{and}\quad 0\ne \za H_i\sbs H_i;
\]
moreover there are $Y,Z\in\CP_{\l(i)}$ such that
\[
\mathbf{0}\ne\ze_{Y(i)}\za\in H_i \quad\text{and}\quad \mathbf{0}\ne
\za\ze_{Z(i)}\in H_i.
\]
\end{enumerate}
\end{theo}
\begin{proof}
Assume that $J$ is a finite subset of $I$, suppose that $\za = \sum_{j\in
J}\za_j$, where $\mathbf{0}\ne \za_j\in H_j$ for $j\in J$, let us
choose a maximal chain $j_1<\cdots<j_n$ of $J$ and let
$\CA_1\in\BD_{j_1},\ldots,\CA_n\in\BD_{j_n}$ be such that
 $\CA_1\sps\cdots\sps\CA_n$. Then by Lemma
\ref{nonzerochain2} the $(X_{\CA_n}\times X_{\CA_n})$-blocks of
$\za$ and $\za' = \za_{j_1}+\cdots+\za_{j_n}$ coincide and, on the
other hand, the $(X_{\CA_n}\times X_{\CA_n})$-block of $\za'$ is not
zero by Lemma \ref{nonzerochain}. As a consequence $\za
\ne\mathbf{0}$ and this proves the independence of $\CH$.

(1) If $i\in I$ and $I$ is not a maximal element, then $H_i\is
F_{\l(i)}\is \BF\BR_X(D)$ or $H_i\is
G_{\l(i)}\is \BF\BM_X(D)$ as rings, depending on the fact that $i$ is in $I'$ or not, therefore $H_i$ is a ring without an
identity. If, on the contrary, $i$ is a maximal element, then $H_i = \psi_i(D)\is D$ and $\ze_{X_i} = \psi_i(1)$ is an identity
for $H_i$. Now $X_i\ne X$, because $I$ has at least two elements,
consequently $H_i$ is not an unital subring of $Q$.

(2) follows from the property (1) of Proposition \ref{subrings},
since $H_i\sbs S_i$ for all $i\in I$.

(3) It is clearly sufficient to take $i,j\in I$ with $i<j$, two nonzero elements $\za\in H_j$, $\zb\in H_b$ and show that $\za\zb$ and $\zb\za$ are both in $H_i$.
First, according to Proposition \ref{subrings} we have that $\za\zb\in S_i$ and $\zb\za\in S_i$. Given $\CA\in\BD_i$, we have from
\eqref{mamagicmatrix} that
\begin{equation}\label{equa}
\ze_{X_{\CA}}(\za\zb)\ze_{X_{\CA}} =
\ze_{X_{\CA}}\za\ze_{X_{\CA}}\zb\ze_{X_{\CA}}.
\end{equation}
By Lemma \ref{HsbsFr} we have that either $\ze_{X_{\CA}}\zb\ze_{X_{\CA}}\in
F_{\l(i)}$, or $\ze_{X_{\CA}}\zb\ze_{X_{\CA}}\in
G_{\l(i)}$, according to the fact that $i\in I'$ or not.
Since $\ze_{X_{\CA}}\in Q_{\l(i)}$ and $\za\in
Q_{\l(j)}\sbs Q_{\l(i)}$ and both $F_{\l(i)}$ and $G_{\l(i)}$ are \emph{left} ideals of
$Q_{\l(i)}$, we infer that the first member of \eqref{equa} belongs
to $F_{\l(i)}$, or to $G_{\l(i)}$ respectively. As a result $\za\zb\in H_i$, again by Lemma
\ref{HsbsFr}. If $i\in I'$, since $F_{\l(i)}$ is also a \emph{right} ideal of
$Q_{\l(i)}$, the same argument as above shows that $\zb\za\in H_i$. Assume that $i\nin I'$, so that $j\nin I'$ as well. In order to show that $\zb\za\in H_i$ also in this case, it is sufficient to consider the case in which $\zb = \ze_{Y(i)}$ and $\za = \psi_{j}(\ze_{V,W})$ for some $Y\in\CP_{\l(i)}$ and $V,W\in\CP_{\l(j)}$ (see Remark and Notation \ref{elementidempot1}).
Since $\psi_{j}(\ze_{V,W}) = \psi_{j}(\ze_{V}\ze_{V,W}) = \psi_{j}(\ze_{V})\psi_{j}(\ze_{V,W}) = \ze_{V(j)}\psi_{j}(\ze_{V,W})$, if $Y(i)\cap V(j) = \vu$, then $\ze_{Y(i)}\psi_{j}(\ze_{V,W}) = \mathbf{0}$. Otherwise, according to (2) of Lemma \ref{lem idempprim Hi} we have that $Y(i)\sbs V(j)$. Given $\CA\in\BD_{i}$, we claim that
\begin{equation}\label{eq equa2}
    \ze_{X_{\CA}}\ze_{Y(i)}\psi_{j}(\ze_{V,W})\ze_{X_{\CA}} =
\ze_{t^{-1}_{\CA}(Y),Z}\in G_{\l(i)}
\end{equation}
for a suitable $Z\in\CP_{\l(i)}$; it will follow from Lemma \ref{HsbsFr} that $\zb\za =$\linebreak $ \ze_{Y(i)}\psi_{j}(\ze_{V,W})\in H_i$. Since $Y(i)\cap X_{\CA} = t^{-1}_{\CA}(Y)$, it follows from $Y(i)\sbs V(j)$ that $t^{-1}_{\CA}(Y)\sbs t^{-1}_{\CB}(V)$ for a necessarily unique $\CB\in\BD_{j}$ and, given $x,y\in X$, we have that
\[
[\ze_{Y(i)}\psi_{j}(\ze_{V,W})\ze_{X_{\CA}}](x,y) \ne 0 \quad \text{only if $x\in t^{-1}_{\CA}(Y)$}.
\]
There are $\chi<\al^{\xi+1-\l(i)}\bullet\beth$ and $\l,\mu<\al^{\xi+1-\l(j)}\bullet\beth$ such that $Y = X_{\l(i),\chi}$, $V = X_{\l(j),\l}$ and $W = X_{\l(j),\mu}$. Let $x\in t^{-1}_{\CA}(Y)$, so that $x = \al^{\l(i)}\bullet k^{-1}_{\CA}(\chi)+\t$ for a unique $\t<\al^{\l(i)}$. Since $x\in t^{-1}_{\CB}(V)$ as well, there is a unique $\s<\al^{\l(j)}$ such that $x = \al^{\l(j)}\bullet k^{-1}_{\CB}(\l)+\s$. Also, $\s = \al^{\l(i)}\bullet\s'+\t$ for a unique $\s'<\al^{\l(j)-\l(i)}$ (see Remark \ref{remeuclid}) and so
\[
x = \al^{\l(i)}\bullet\left(\al^{\l(j)-\l(i)}\bullet k^{-1}_{\CB}(\l)+\s'\right)+\t.
\]
Consequently, for every $y\in X_{\CA}$ we have
\begin{align*}
[\ze_{t^{-1}_{\CA}(Y)}\psi_{j}(\ze_{V,W})\ze_{X_{\CA}}](x,y) &= \psi_{j}(\ze_{V,W})(x,y) \\
&=
\left\{
  \begin{array}{ll}
    1, & \hbox{if $y = \al^{\l(j)}\bullet k^{-1}_{\CB}(\mu)+\s$;} \\
    0, & \hbox{otherwise.}
  \end{array}
\right. \\
&= \left\{
     \begin{array}{ll}
       1, & \hbox{$y =\al^{\l(i)}\bullet\left(\al^{\l(j)-\l(i)}\bullet k^{-1}_{\CB}(\mu)+\s'\right)+\t$;} \\
       0, & \hbox{otherwise.}
     \end{array}
   \right.
\end{align*}
If we take $Z = X_{\l(i),\nu}$, where $\nu = \al^{\l(j)-\l(i)}\bullet k^{-1}_{\CB}(\mu)+\s'$, we conclude that \eqref{eq equa2} holds.

As far as the last statement is concerned, suppose again that $\za = \sum_{j\in
J}\za_j$, where $J$ is a finite subset of $I$ and $\mathbf{0}\ne \za_j\in H_j$ for $j\in J$, let us consider a maximal chain $j_1<\cdots<j_n$ of $J$ and
take $\CA\in\BD_{i}$, $\CA_1\in\BD_{j_1},\ldots,\CA_n\in\BD_{j_n}$
such that $\CA\sps\CA_1\sps\cdots\sps\CA_n$. As seen in the first
part of the present proof, the $(X_{\CA_n}\times X_{\CA_n})$-block
of $\za$ is not zero. Let $x,y\in X_{\CA_{n}}$ be such that
$\za_{xy}\ne\mathbf{0}$. Since $X_{\CA_{n}}\sbs X_{\CA}$, there are
(necessarily unique) $V,W\in\CQ_{\CA}$ such that $x\in V$ and $y\in
W$. If $Y,Z$ are the unique elements of $\CP_{\l(i)}$ such that
$Y(i) = \overline{V}$ and $Z(i) = \overline{W}$ (see
\eqref{overlined}), then $\ze_{Y(i)}$ and $\ze_{Z(i)}\in H_i$; both
$\ze_{Y(i)}\za$ and $\za\ze_{Z(i)}$ are nonzero and belong to $H_i$.
\end{proof}

\section{The ring $D_{I}$.}\label{sect DI}

As in the second half of the previous section, we assume that a polarized artinian poset $I$ is given. The $D$-subrings $H_i$ (for $i\in I$) of $Q = \CFM_X(D)$ we
have introduced in the previous section can be used in a natural way as
building blocks to construct further $D$-subrings of $Q$,
this time starting from \emph{subsets} of $I$. Indeed, given a
subset $J\sbs I$, if we consider the $(D,D)$-submodule $H_J$ of $Q$
defined by
\[
H_J \defug \bigoplus_{j\in J}H_j,
\]
then it follows from Theorem \ref{lemposetring} that $H_J$ is a
$D$-subring; it may fail to be a unitary subring of $Q$ and it may
even lack multiplicative identity. Of course we set $H_{\emptyset} = 0$. If we define the subset $X_{J}$
by setting
\[
X_J \defug \bigcup\{X_{i}\mid i\in J\} = \bigcup\{X_{\chi}\mid
A_{\chi}\cap J\ne\emptyset\},
\]
then $X_{J}$ is the smallest subset of $X$ such that every matrix in
$H_{J}$ has zero entries outside the $(X_{J}\times X_{J})$-block. We observe that if a matrix $\zu \in Q$ acts as a multiplicative identity on $H_J$, then the following equalities hold as well:
\begin{equation}\label{eq multunitHJ}
    \zu\ze_{X_J} = \ze_{X_J} = \ze_{X_J}\zu.
\end{equation}
In fact, given $x\in X_J$, there are $j\in J$, $\CA\in\BD_j$ and
$Y\in\CQ_{\CA}$ such that $x\in Y$ (see \eqref{prelem}). Inasmuch as
$\ze_{\overline{Y}}\in H_j\sbs H_J$ by Lemma \ref{lem idempprim Hi}, we have that
$\ze_{\overline{Y}}\zu = \ze_{\overline{Y}} =
\zu\,\ze_{\overline{Y}}$ and, since $x\in {\overline{Y}}$, we infer that
$\zu(x,y) = \delta(x,y) = \zu(y,x)$ for every $y\in X$, which proves \eqref{eq multunitHJ}.
As a result, $H_J+\ze_{Z}\,D$ is the smallest $D$-subring of $Q$ which has a multiplicative identity (given by $\ze_{Z}$) and contains $H_J$ as an ideal; we denote it by $D_{I,J}$:
\begin{equation}\label{Jring}
D_{I,J} \defug H_J+\ze_{X_J}\,D.
\end{equation}
In case, $J=I$, we simply write $D_{I}$ instead of $D_{I,I}$. With the next result, we give necessary and
sufficient conditions under which $H_J = D_{I,J}$. As we shall see, in this context a relevant role is played by the set $\max{J}$ defined by
\[
\max{J} \defug \ZM(I)\cap\{J\le\} = \ZM(\{J\le\}),
\]
(recall that $\ZM(I)$ denotes the set of all maximal elements of $I$), namely the set of those maximal elements of $I$ which follow some element of $J$. Of course it may happen that
$J\not\sbs\{\le\max{J}\}$, in particular that $\max{J} =
\emptyset$. If every element of $I$ is bounded by a maximal element or, equivalently, all maximal chains of $I$ have a greatest element, then it is clear that  $X_{J}\sbs X_{\max{J}}$; this inclusion is an equality if and only if, given $m\in\max{J}$, every maximal chain of $I$ which is bounded by above by $m$ contains an element of $J$. Obviously this is the case if $\max{J}\sbs J$, in particular when $J$ is an upper subset of $I$; in this latter case it is clear that $\max{J} = \ZM(J)$.

\begin{def}\label{def finshelt}
We say that a subset $J$ of $I$ is \emph{finitely sheltered in} $I$ if the following three conditions hold:
\[
\max{J} \text{ is finite,} \quad J\sbs\{\le\max{J}\}\quad \text{ and }\quad \max{J}\sbs J.
\]
\end{def}

If $J$ is an upper subset of $I$, then $J$ is finitely sheltered in $I$ if and only if $J$ has a finite cofinal subset; in particular $I$ is finitely sheltered in $I$ exactly when $I$ has a finite cofinal subset and, if it is the case, then every subset of $I$ is finitely sheltered in $I$.

\begin{pro}\label{multunit}
If $\vu\ne J\sbs I$, then the following conditions are equivalent:
\begin{enumerate}
\item $H_J$ has a multiplicative identity.
\item $D_{I,J} = H_J$.
\item $H_J\cap (\ze_{X_{J}}D)\ne 0$.
\item $J$ is finitely sheltered in $I$.
\end{enumerate}
If any (and hence all) of these conditions holds and $\max{J} =
\{m_1,\ldots,m_r\}$, then
\begin{equation}\label{multuniteq}
  \ze_{X_J} =  \ze_{X_{\max{J}}} = \ze_{X_{m_{1}}}+\cdots+\ze_{X_{m_{r}}}.
\end{equation}
Consequently, either $D_{I,J} = H_J$, or the sum in \eqref{Jring} is direct.
\end{pro}
\begin{proof}
The equivalence between (1) and (2) follows from the previous observation, while the implication
(2)$\Rightarrow$(3) is obvious.

  (3)$\Rightarrow$(4). Suppose that there is a finite subset $F\sbs J$ and nonzero matrices
  $\zd_{i}\in H_{i}$, for $i\in F$, such that
  \begin{equation}\label{bobobo}
     \mathbf{0} \ne\zd = \sum_{i\in F}\zd_{i}\in \ze_{X_{J}}D
\end{equation}
and let us prove first that $X_F = X_J$. Clearly $F\sbs J$ implies that $X_F\sbs X_J$. On the other hand, given $x\in X_J$, since the $x$-th row of $\zd$ is not zero then, for some $i\in F$, the $x$-th row of $\zd_i$ is not zero. This means that there exists $\CA\in\BD_i$ such that $x\in X_{\CA}\sbs X_i\sbs X_F$. Thus $X_J\sbs X_F$. Next, given any maximal chain $i_1<\ldots<i_r$ of $F$, we claim that $i_r$ must be a maximal element of $I$, so that $\max{F}\sbs F$. Indeed, if $\CA_1\in\BD_{i_1},\ldots,\CA_r\in\BD_{i_r}$ are such that $\CA_1\sps\cdots\sps\CA_r$ (see (2) of Lemma \ref{lem-partition2}), then it follows from Lemma \ref{nonzerochain2} that the $(X_{\CA_r}\times X_{\CA_r})$-blocks of $\zd$ and $\zd_{i_1}+\cdots+\zd_{i_r}$ coincide. If $i_r$ is not maximal then, with the help of (2) of Lemma \ref{HsbsFr} and (2) of Theorem \ref{pro-chainflr}, we see that there are $Y_1,\ldots,Y_s\in \CQ_{\CA_r}\sbs\CP_{\l(i_r)}$ such that, given $x\in X_{\CA_r}$, the $x$-th row of $\zd$ is not zero only if $x\in Y_1\cup\cdots\cup Y_s$. Since the $(X_{\CA_r}\times X_{\CA_r})$-block of $\zd$ is a nonzero scalar matrix and $X_{\CA_r}$ is the union of $\al$ elements of $\CP_{\l(i_r)}$, we have a contradiction and our claim is proved.

Now, let $j\in J$ and take any $x\in X_j$. Then there is a unique
$\chi\in\bet$ such that $x\in X_{\chi}$ and $j\in A_{\chi}$. As $X_J
= X_F$, we infer that $X_{\chi}\sbs X_F$ and so $A_{\chi}\cap
F\ne\emptyset$. Thus $A_{\chi}\cap F$ is a maximal chain of $F$
which is bounded from above by an element $m\in\max{F}$, as we have seen previously. As a result $A_{\chi}$
itself is bounded from above by $m$ and this proves that
$J\sbs\{\le\max{F}\}\sbs\{\le\max{J}\}$.

  Finally, let us show that $\max{J}\sbs F$, from which it will follow that $\max{J} = \max{F}$ and so $\max{J}$ is finite.
  Assume, on the contrary, that there is some maximal element $m$ of $I$ such that $m\not\in F$ but $j<m$
for some $j\in J$. As a consequence, according to Lemma
\ref{lem-partition2} there is some $\CB\in\BD_{m}$ which is
contained in some $\CA\in\BD_{j}$ and hence $X_{\CB}\sbs X_{\CA}\sbs
X_j\sbs X_J$. We observe that if $i_{1},\ldots,i_{r}$ are those
elements of $F$ such that the $(X_{\CB},X_{\CB})$-blocks of
$\zd_{i_{1}},\ldots,\zd_{i_{r}}$ are not zero, then $r\ge 1$ and
 $i_{1},\ldots,i_{r}\in\{<m\}$. Indeed, if $t\in\{1,\ldots,r\}$, then there is some $\CC\in\BD_{i_{t}}$ such that
 $\CC\cap\CB\ne\emptyset$ and hence $i_{t}$ and $m$ are related by Lemma \ref{lem-partition2}; since $m$ is maximal in $I$, then necessarily $i_{t}<m$. We have now
 \[
 \mathbf{0}\ne\zd_{X_{\CB}\,X_{\CB}} =
 (\zd_{i_{1}})_{X_{\CB}\,X_{\CB}}+\cdots+(\zd_{i_{r}})_{X_{\CB}\,X_{\CB}}
 \]
 and, inasmuch as $i_{1},\ldots,i_{r}$ are not maximal elements of $I$ and
 $\l(i_{1}),\ldots,\l(i_{r})$$<\l(m)$, we infer from  (2) of Lemma \ref{HsbsFr} and (2) of Theorem \ref{pro-chainflr} that there are
 $Y_{1},\ldots,Y_{s}\in \CQ_{\CB}\sbs\CP_{\l(m)}$ such that, given
$x\in X_{\CB}$, the
 $x$-th row of $\zd_{X_{\CB}\,X_{\CB}}$ is not zero only if
$x\in
 Y_{1}\cup\cdots\cup Y_{s}$. This leads to a contradiction; in fact, since
 $\zd\in \ze_{X_{J}}D$ and $X_{\CB}\sbs X_{J}$, for every $x\in X_{\CB}$ the
 $x$-th row of $\zd$ is not zero and $X_{\CB}$ is the union of $\al$
 members of $\CP_{\l(m)}$. This shows that $\max{J}\sbs F$, as
wanted.

    (4)$\Rightarrow$(1) Assume (4) and set $\max{J} =
\{m_1,\ldots,m_r\}$. Then it follows from Proposition
\ref{pro max} that $\{X_{m_{1}},\ldots,X_{m_{r}}\}$ is a partition of $X_J$. Thus \eqref{multuniteq} holds and, since $\ze_{X_{m_{t}}}\in H_{m_t}$ for all $t\in\{1,\ldots,r\}$, it follows that $\ze_{X_{J}} \in H_J$ and therefore $D_{I,J} = H_J$.
\end{proof}

Formally speaking, the assignment $I\mapsto D_I$ cannot be considered as a map from the class of all pairs $(I,I')$, where $I$ is an artinian posets and $I'$ is a lower subset of $I$, to the class of $D$-rings. In fact, the construction that leads us to the ring $D_I$ bears first on the choice of a bijection $\Xi\colon\chi\mapsto A_{\chi}$ from the cardinal $\beth = |\CM|$ to the set $\CM$ of all maximal chains of $I$, next on the choice of a family $\CF = (\CA_{i})_{i\in I}$, where each $\CA_{i}$ is an equivalence class modulo $\sim_{i}$ and finally on the choice of a family of bijections $\CK = (\map{k_{\CA_{i}}}{\L_{\CA_{i}}}{\al^{\xi+1-\l(i)}\bullet\bet})_{i\in I}$. Thus the ring $D_I$ strictly depends on the ordered quintuple $(I,I',\Xi,\CF,\CK)$, so that our construction realizes actually a function from the class of all such quintuples. As one might expect, if we take a second quintuple $(J,J',\Pi,\CG,\CL)$ from this class, every order isomorphism $\map fIJ$ such that $J' = f(I')$ induces a canonical $D$-ring isomorphism from $D_I$ to $D_J$. This is not immediately obvious and, in what follows, we show how it works. Let $\CN$ be the set of all maximal chains of $J$, so that $|\CM| = \beth = |\CN|$, and write  let $B_{\chi} = \Pi(\chi)$ for every $\chi\in\beth$. As we did with Notations \ref{not main}, for every $j\in J$ we may consider the equivalence relation $\sim_{j}$ in the set $\CN_j$ of all maximal chains of $J$ which contain $j$ and we get the corresponding quotient set $\BE_{j}$. Note that $f$ induces an obvious bijection $\map{\overline{f}}{\BD_{i}}{\BE_{f(i)}}$ For every $\CB\in\BE_{j}$ we have the set $\beth'_{\CB} = \{\chi\in\beth\mid B_{\chi}\in\CB\}$ and we can define the subsets $X'_{\CB}$, $X'_{j}$ of $X$, as well as the sets $\Lambda'_{\CB}$, $\CQ'_{\CB}$ and the map $\map{t'_{\CB'\CB}}{X'_{\CB}}{X'_{\CB'}}$ for every $\CB,\CB'\in\BE_{j}$ exactly as we did with $X_{\CA}$, $X_{i}$, $\Lambda_{\CA}$, $\CQ_{\CA}$ and $\map{t_{\CA'\CA}}{X_{\CA}}{X_{\CA'}}$ for every $i\in I$ and $\CA,\CA'\in\BD_{i}$. Next, for every $j\in J$ let us choose $\CB_{j}\in\BE_{j}$ and a bijection
\[
\lmap{k'_{\CB_{j}}}{\L'_{\CB_{j}}}{\al^{\xi+1-\l(j)}\bullet\bet}
\]
and, for every $\CB,\CB'\in\BE_{j}$, let us consider the bijections
\[
\lmap{t'_{\CB}}{X'_{\CB}}{X}, \quad \lmap{t'_{\CB'\CB}}{X'_{\CB}}{X'_{\CB'}},
\]
defined in the same fashion as $t_{\CA}$ and $t_{\CA'\CA}$. Thus, for every $j\in J$ we can define the $D$-ring $S'_{j}$, analogous to $S_{i}$ for $i\in I$, and the $D$-ring monomorphism
\[
\lmap{\psi'_j}{Q}{\ze_{X'_j}Q\ze_{X'_j}}
\]
analogous to $\psi_i$, as in Proposition \ref{subrings}, so that $S'_{j} = \psi'_j(Q)$. Through a slight notational transgression, we may view $\psi_i$ and $\psi'_j$ as isomorphisms from $Q$ to $S_{i}$ and $S'_{j}$ respectively. For each $i\in I$ let us consider the $D$-ring isomorphism
\[
\lmap{\a_i}{S_{i}}{S'_{f(i)}}
\]
defined by
\[
\a_i = \psi'_{f(i)}\psi_{i}^{-1}.
\]
Given $\CA\in\BD_{i}$, let us consider the bijection
\[
\lmap{s_{\CA}}{X_{\CA}}{X'_{\overline{f}(\CA)}}
\]
defined by $s_{\CA} = t'^{-1}_{f(\CA)}t_{\CA}$. Let $\za\in S_{i}$ and $\CB\in\BE_{f(i)}$. Then $\CB = \overline{f}(\CA)$ for a unique $\CA\in\BD_{i}$ and for every $x,y\in X'_{\CB}$ we have:
\begin{equation}\label{eq computiso}
\begin{aligned}
\a_{i}(\za)(x,y) &= [(\psi'_{f(i)}\psi_{i}^{-1})(\za)](x,y) = \psi_{i}^{-1}(\za)(t'_{\CB}(x),t'_{\CB}(y)) \\
 &= \psi_{i}^{-1}(\za)(t_{\CA}(s^{-1}_{\CA}(x)),t_{\CA}(s^{-1}_{\CA}(y))) \\
 &= \za(s^{-1}_{\CA}(x),s^{-1}_{\CA}(y)).
\end{aligned}
\end{equation}
We claim that if $i,j\in I$ and $i<j$, given $\za\in S_{i}$, $\zb\in S_{j}$ the equality
\begin{equation}\label{eq iso1}
    \a_{i}(\za)\a_{j}(\zb) = \a_{i}(\za\zb)
\end{equation}
holds, that is,
\begin{equation}\label{eq iso2}
    [\a_{i}(\za)\a_{j}(\zb)](x,y) = \a_{i}(\za\zb)(x,y)
\end{equation}
for every $x,y\in X$. Indeed, first note that both members of \eqref{eq iso1} belong to $S'_{f(i)}$ by Proposition \ref{subrings}, (2); thus, given $x,y\in X$, if there is no $\CB\in\BE_{f(i)}$ such that $x,y\in X'_{\CB}$, then both members of \eqref{eq iso2} are zero. Assume, on the contrary, that $x,y\in X'_{\CB}$ for some $\CB\in\BE_{f(i)}$ and let $\CA\in\BD_{i}$ be such that $\CB = \overline{f}(\CA)$. Then, by using \eqref{eq computiso} we have:
\begin{align*}
[\a_{i}(\za)\a_{j}(\zb)](x,y)
        &= \sum_{z\in X}\a_{i}(\za)(x,z)\,\a_{j}(\zb)(z,y) \\
        &= \sum_{z\in X'_{\CB}}\a_{i}(\za)(x,z)\,\a_{j}(\zb)(z,y) \\
        &= \sum_{z\in X'_{\CB}}\za(s^{-1}_{\CA}(x),s^{-1}_{\CA}(z))\,\zb(s^{-1}_{\CA}(z),s^{-1}_{\CA}(y)) \\
        &= \sum_{w\in X_{\CB}}\za(s^{-1}_{\CA}(x),w)\,\zb(w,s^{-1}_{\CA}(y)) \\
        &= (\za\zb)(s^{-1}_{\CA}(x),s^{-1}_{\CA}(y)) \\
        &= [\a_{i}(\za\zb)](x,y).
\end{align*}
Thus the equality \eqref{eq iso1} is proven.

Assume that both $I$, $J$ are polarized in such a way that $J' = f(I')$ and, for every $j\in J$, define $H'_{j}$ as follows:
\[
H'_j = \left\{
        \begin{array}{ll}
          \psi'_j\left(F_{\l(j)}\right), & \hbox{if $j$
\underbar{is not} a maximal element of $J$ and $j\in J'$;} \\
          \psi'_j\left(G_{\l(j)}\right), & \hbox{if $j$
\underbar{is not} a maximal element of $J$ and $j\nin J'$;} \\
          \psi'_j\left(D\right), & \hbox{if $j$ \underbar{is} a maximal
element of $J$.}
        \end{array}
      \right.
 \]
Then, for any $K\sbs J$, we can define $H'_{K} \defug\bigoplus_{j\in K}H'_j$. For every $i\in I$ it is clear that $\za\in H_i$ if and only if $\a_i(\za)\in H'_{f(i)}$, therefore we can define the $D$-module isomorphism
\[
\lmap{\bigoplus_{i\in I}\a_i}{H_{I}}{H'_{J}},
\]
which extends to a $D$-module isomorphism
\[
\lmap{\a}{D_{I}}{D_{J}}.
\]
This is obvious if $I$ has a finite cofinal subset, since in this case we have that $D_{I} = H_I$ end $D_J = H'_J$ by Proposition \ref{multunit}; otherwise $D_{I} = H_I\oplus \ze_{X}D$ end $D_J = H'_J\oplus \ze_{X}D$, therefore $\a = \left(\bigoplus_{i\in I}\a_i\right)\oplus 1_{\ze_{X}D}$. Now it follows from \eqref{eq iso1} that $\a$ is a $D$-ring isomorphism.

\begin{re}\label{remark iso}
If $I$ is any artinian poset and $J\sbs I$, then the two rings $D_{I,J}$ and $D_{J} = D_{J,J}$ can be different and may even be non-isomorphic. This latter case occurs if, for example, $\sup(|I|, |\CM|)> \al_0$ and $\sup(|J|, |\CN|)<\sup(|I|, |\CM|)$, where $\CM$ and $\CN$ are the sets of all maximal chains of $I$ and $J$ respectively.
\end{re}

\section{Upper subsets of $I$ versus ideals of the ring $D_{I}$.}\label{sect idDI}

If $K\sbs J\sbs I$, then it is clear that $H_K$ and $H_{J\setminus
K}$ are complementary direct summands of $H_J$ as
$(D,D)$-submodules, but they need not be ideals of $D_{I,J}$. In this
connection the case in which $K$ is an upper subset of $J$ is of a
particular interest, mainly due to the following result.

\begin{pro}\label{uppersbs}
Assume that $\vu\ne K\sbs J\sbs I$. Then the following properties hold:
\begin{enumerate}
    \item $H_{J\setminus K}$ is an ideal of $D_{I,J}$ if and only if $K$ is an
upper subset of $J$.
    \item If $K$ is an upper subset of $J$, then there is a unique (unital)
surjective $D$-linear ring
    homomorphism
\[
\lmap{\f_{K,J}}{D_{I,J}}{D_{I,K}}
\]
such that
\begin{equation}\label{mapfKJ}
\f_{K,J}(\za'+\za''+\ze_{X_J}d) = \za'+\ze_{X_K}d \quad\text{ for
all\,\, $\za'\in H_K,\,\, \za''\in H_{J\setminus K},\,\, d\in D$};
\end{equation}
moreover
\begin{equation}\label{kermapfKJ}
    \Ker(\f_{K,J}) = \begin{cases} H_{J\setminus K}+(\ze_{X_J}-\ze_{X_K})D,\text{ if $K$
    is finitely sheltered in $I$, } \\
    H_{J\setminus K},\text{ if $K$ is not finitely sheltered in $I$, }
    \end{cases}
 \end{equation}
therefore $\f_{K,J}$ induces an isomorphism of $D$-rings
\[
D_{I,K} \is \begin{cases} D_{I,J}/\left[H_{J\setminus K}+(\ze_{X_J}-\ze_{X_K})D\right],\text{ if $K$
    is finitely sheltered in $I$, } \\
D_{I,J}/H_{J\setminus K},\text{ if $K$
    is not finitely sheltered in $I$. }
    \end{cases}
\]
\end{enumerate}
\end{pro}

\begin{proof}
(1) Assume that $H_{J\setminus K}$ is an ideal of $D_{I,J}$ and take
$j\in J$, $k\in K$ with $k\le j$. If $j\not\in K$, then $H_j\sbs
H_{J\setminus K}$ and it follows from Theorem \ref{lemposetring}
that $0\ne H_jH_{k}\sbs H_{k}\cap H_{J\setminus K} = 0$, hence a
contradiction. Thus necessarily $j\in K$. Conversely, suppose that
$K$ is an upper subset of $J$ and let $j\in J$, $k\in J\setminus K$.
Then exactly one of the following possibilities occurs: a) $j\not\in
K$, b) $j\in K$ and $j,k$ are unrelated, c) $j\in K$ and $j>k$. In
all cases it follows from Theorem \ref{lemposetring} that
$H_jH_k\cup H_kH_j\sbs H_{J\setminus K}$. Since $H_{J\setminus K}$
is already a $(D,D)$-submodule of $D_{I,J}$, this is sufficient to
conclude that it is an ideal of $D_{I,J}$.

(2) Assume now that $K$ is an upper subset of $J$. If $J$ is finitely sheltered in $I$, then $D_{I,J} = H_{J} = H_{K}\oplus
H_{J\setminus K}$. Since $K$ is an upper subset of $J$, then $K$ is finitely sheltered in $I$ as well and hence $D_{I,K} = H_K$. Set $\max{J} =
\{m_{1},\ldots,m_{r},m_{r+1},\ldots,m_{r+s}\}$, where
$\{m_{1},\ldots,m_{r}\} = \max{K}$. It is clear that every maximal chain of $J$ is bounded by an element of $\max{J}$, thus Proposition
\ref{pro max} tells us that both
$\{X_{m_{1}},\ldots,X_{m_{r+s}}\}$ and $\{X_{m_{1}},\ldots,X_{m_{r}}\}$ are partitions of $X_{J}$ and $X_{K}$, respectively. It follows that $\ze_{X_{m_{1}}},\ldots,\ze_{X_{m_{r}}}, \ze_{X_{m_{r+1}}},\ldots,\ze_{X_{m_{r+s}}}$ are pairwise orthogonal idempotents and we have that
\[
\ze_{X_{K}} = \ze_{X_{m_{1}}}+\cdots+\ze_{X_{m_{r}}}\in H_{K}, \quad \ze_{X_{J\setminus K}}=
 \ze_{X_{m_{r+1}}}+\cdots+\ze_{X_{m_{r+s}}}\in H_{J\setminus K},
\]
hence $\ze_{X_{J}} = \ze_{X_{K}}+\ze_{X_{J\setminus K}}$. As a result, given $\za'\in
H_K$, $\za''\in H_{J\setminus K}$ and $d\in D$, we may write
\[
\za'+\za''+\ze_{X_{J}}d = \za'+\ze_{X_{K}} d+\za''+\ze_{X_{J\setminus K}} d
\]
and, since $\za'+\ze_{X_{K}} d\in H_{K}$ and $\za''+\ze_{X_{J\setminus K}} d\in H_{J\setminus
K}$, we infer that \eqref{mapfKJ} defines $\f_{K,J}$ as the projection
of $D_{I,J}$ onto $H_{K} = D_{I,K}$ parallel to $H_{J\setminus K}$. If $J$ is not finitely sheltered in $I$, then
    \begin{equation}\label{eq decomp}
    D_{I,J} =
H_{K}\oplus H_{J\setminus K}\oplus \ze_{X_J}D
\end{equation}
by Proposition \ref{multunit} and there exists a unique $D$-linear map
$\map{\f_{K,J}}{D_{I,J}}{D_{I,K}}$ satisfying \eqref{mapfKJ}.
In any case, we have that $\f_{K,J}$ is well defined by mean of
\eqref{mapfKJ} and, by using the fact that $H_{J\setminus K}$ is an
ideal of $D_{I,J}$, it is an easy matter to show that
$\f_{K,J}$ is a ring homomorphism.

Finally, if $K$ is finitely sheltered in $I$, then
$\ze_{X_{\max{K}}} = \ze_{X_K}\in H_{K}$ by Proposition \ref{multunit} and so
$\ze_{X_{J}}-\ze_{X_{K}}\in D_{I,J}$; consequently
$\Ker(\f_{K,J}) = H_{J\setminus K}+(\ze_{X_{J}}-\ze_{X_{K}})D$. If, on the contrary, $K$ is not finitely sheltered in $I$, then $J$ is not finitely sheltered in $I$ as well and we have the decomposition \eqref{eq decomp}; in this case it is clear that $\Ker(\f_{K,J}) = H_{J\setminus K}$.
\end{proof}

Given any subset $J$ of $I$ and any ordinal $\a <\xi$, the $\a$-th layer $J_{\a}$ is a lower subset of $J$, therefore it follows from Proposition \ref{uppersbs} that $H_{J_{\a}}$ is an ideal of the ring $D_{I,J}$. By considering the set $J_1$ of all minimal elements of $J$, the corresponding ideal $H_{J_1}$ will play a special role, mainly due to the next result.

\begin{pro}\label{uppersbss}
If $\emptyset \ne J\sbs I$, then $H_{J_1}$ is essential as a right ideal and is pure as a left ideal of the ring $D_{I,J}$.
\end{pro}
\begin{proof}
Suppose that $\mathbf{0}\ne\za\in D_{I,J}$ and assume first that $\za\in
H_J$. Then there is a smallest finite nonempty subset $K\sbs J$ such that $\za\in H_K$.
According to Corollary \ref{lengthlayerr} we can choose some $j\in
J_1$ such that $K' = K\cap\{j\le\}\ne\emptyset$ and so $\za =
\za'+\za''$ for unique nonzero elements $\za'\in H_{K'}$ and
$\za''\in H_{K\setminus K'}$. Consequently it follows from Theorem
\ref{lemposetring} that there is an idempotent $\ze\in H_j$ such
that $\mathbf{0}\ne\za'\ze\in H_j\sbs H_{J_1}$, while $\za''\ze =
\mathbf{0}$. As a result $\mathbf{0}\ne\za\ze\in H_{J_1}$.

If $D_{I,J} = H_J$, the above argument shows that $H_{J_1}$ is essential
as a right ideal of $D_{I,J}$. Otherwise, according to Proposition \ref{multunit} we have that $D_{I,J} = H_J\oplus \ze_{X_J}D$. If $\mathbf{0}\ne\za\in \ze_{X_J}D$, given any $j\in J_{1}$ and $Y\in\CP_{\l(j)}$, we have that $\ze_{Y(j)}\in H_j\sbs H_{J_{1}}$ (see Lemma \ref{lem idempprim Hi}) and so $\mathbf{0}\ne\za\,\ze_{Y(j)}\in H_{J_{1}}$, because $Y(j)\sbs X_J$. In order to complete the proof it remains to consider the case in which $\za = \zb+ \zd$, where $\mathbf{0}\ne\zb\in H_J$ and $\mathbf{0}\ne\zd\in \ze_{X_J}D$. Let $K$ be the smallest finite subset of $J$ such that
$\zb\in H_K$. Again from Proposition \ref{multunit} we have that either $\max{J}$ is infinite, or $J\nsbs\{\max{J}\}$, or $\max{J}\nsbs J$. In the first and third cases it is clear that $\max{J}\setminus K\ne
\emptyset$. In the second case, $J$ contains at least an infinite chain. Thus, in all cases we can choose an element $m\in J$ such that $m\nleqslant k$ for every $k\in K$ and Corollary \ref{lengthlayerr} allows us to take some $j\in J_{1}$ in such a way
that $j\le m$. By Lemma \ref{lem-partition2} there are $\CB\in\BD_{m}$ and $\CC\in\BD_{j}$ such that $\CB\sbs\CC$.
We claim that there is some $V\in\CQ_{\CB}$ such that the $x$-th row
of $\zb$ is zero when $x\in V$. This is clear if $\CB\cap\CA =
\emptyset$ for all $\CA\in\BD_{k}$ with $k\in K$. Otherwise, let
$k_1,\ldots,k_r$ be those elements of $K$ such that
$\CB\cap\CA_{t}\ne\emptyset$ for some $\CA_{t}\in\BD_{k_t}$, where
$t\in\{1,\ldots,r\}$. Because of the choice of $m$, it follows from Lemma
\ref{lem-partition2} that $k_1,\ldots,k_r$ and $m$ are
pairwise comparable and we may assume that $k_1<\cdots<k_r<m$. Let
us consider the unique decomposition $\zb = \zb'+\zb''$, where $\zb'\in
H_{\{k_{1},\ldots,k_{r}\}}$ and $\zb''\in
H_{K\setminus\{k_{1},\ldots,k_{r}\}}$. Noting that $\CA\cap\CB =
\emptyset$ whenever $\CA\in\BD_{k}$ for some $k\in
K\setminus\{k_{1},\ldots,k_{r}\}$, we see that the $(X_{\CB}\times
X_{\CB})$-blocks of $\zb$ and $\zb'$ coincide. Inasmuch as
$k_{1},\ldots,k_{r}$ are not maximal, by applying Lemma
\ref{HsbsFr}to the single components of $\zb'$ in $H_{k_{1}}$, \ldots, $H_{k_{n}}$ we see that there are
$V_{1},\ldots,V_{s}\in\CQ_{\CB}\sbs\CP_{\l(m)}$ such that if $x\in
X_{\CB}$ and the $x$-th row of $\zb'$ is not zero, then $x\in
V_{1}\cup\cdots\cup V_{s}$. Thus, since $\CQ_{\CB}$ contains $\al$
elements of $\CP_{\l(m)}$, there is $V\in\CQ_{\CB}$ such that the
$x$-th row of $\zb$ is zero when $x\in V$ and our claim is
established.

Now, pick any $W\in\CQ_{\CC}$ such that $W\sbs V$ (see Remark \ref{remarklem-partition2}) and consider the idempotent $\ze_{\overline{W}}\in H_{j}$. By the above,
$\zb\,\ze_{\overline{W}}$ has zero $x$-th row for $x\in W$, while
$\zd\,\ze_{\overline{W}}$ has nonzero $x$-th row \emph{for all}
$x\in W$, because $W\sbs\overline{W}$ and $\overline{W}\sbs X_{J}$. This shows that
$\za\,\ze_{\overline{W}} = (\zb+\zd)\,\ze_{\overline{W}}$ is not zero and belongs to
$H_{J_{1}}$, completing the proof that $H_{J_1}$ is
essential as a right ideal of $D_{I,J}$.

Finally, let $j\in J_1$ and note that $H_j$ is an ideal of $D_{I,J}$ by (1) of Proposition \ref{uppersbs}. We
have that $H_j = \psi_j(F_{\l(j)})$ and $\psi_j$ is a ring
monomorphism; consequently, since $F_{\l(j)}$ is left pure, we infer
that $H_j$ is left pure as well; in particular, for every $\za\in H_j$ there is an idempotent $\ze\in H_j$ such that $\za = \ze\za$. Assume that $\za\in H_{J_1} = \bigoplus_{j\in J_1}H_j$, that is, $\za = \za_1+\cdots+\za_n$ for some  $\za_1\in H_{j_1},\ldots,\za_n\in H_{j_n}$ with $j_1,\ldots,j_n\in J_1$. Then there are appropriate idempotents $\ze_1\in H_{j_1},\ldots,\ze_n\in H_{j_n}$ such that $\za_r = \ze_r\za_r$ for all $r = 1,\ldots,n$ and it follows from Theorem \ref{lemposetring} that these idempotents are pairwise orthogonal. As a result $\ze = \ze_1+\cdots+\ze_n$ is an idempotent of $H_{J_1}$ such that $\za = \ze\za$, hence $H_{J_1}$ is left pure.
\end{proof}

\section{Semiartinian unit-regular rings are coming, finally!}

The setup we need is now complete for use. In this final section, starting from a given \underline{{\em nonempty\/}} polarized artinian poset $I$, we only have to specialize the ring $D$ and check that the corresponding ring $D_{I}$, as defined in the previous section, is a semiartinian and regular ring which satisfies the conditions we had announced. Thus, by keeping the same data, assumptions and notations so far introduced, \underline{{\em from now on we assume that $D$  is a division ring\/}}. Now the ring $Q = \CFM_{X}(D)$, and hence $Q_{\a}$ for every
$\a\le\xi$, is regular, prime and right selfinjective; moreover
$F_{\a} = \Soc(Q_{\a})$, so that each $H_i$ is a simple and semisimple ring, no matter if $i\in I'$ or not; it has a
multiplicative identity if and only if $i$ is a maximal element of
$I$, in which the case $H_i$ is isomorphic to $D$.

\begin{lem}\label{socRJ}
    Assume that $\vu\ne J\sbs I$. Then $D_{I,J}$ is a regular ring and
    \[
    H_{J_{1}} = \Soc(D_{I,J}).
    \]
    \end{lem}
    \begin{proof}
Since $J\setminus J_{1}$ is an upper subset of $J$, then $H_{J_{1}}$ is an ideal of $D_{I,J}$ by Proposition \ref{uppersbs}. Moreover $H_{J}$ is Von Neumann regular, because it is a direct sum of simple and semisimple rings. If $D_{I,J}\ne H_{J}$, then $D_{I,J}/H_{J}\cong D$ is regular, thus $D_{I,J}$ is itself regular (see \cite[Lemma 1.3]{Good:3}). Let $i\in J_{1}$. Since $\{i\}$ is a lower subset of $J$, it follows from Proposition \ref{uppersbs} that $H_{i}$ is an ideal of $D_{I,J}$. On the other hand, in view of our assumptions $H_{i}$ is a semisimple ring (possibly without identity), thus we infer that $H_{J_{1}} = \bigoplus_{i\in J_{1}}H_{i} \sbs\Soc(D_{I,J})$. Since $H_{J_{1}}$ is an essential right ideal of $D_{I,J}$ by Proposition \ref{uppersbss}, the opposite inclusion holds and the equality follows.
\end{proof}

As we have seen in Lemma \ref{HsbsFr}, if $i\in I\setminus
\max{I}$, then $H_i$ contains the set $\BE_i = \{\ze_{Y(i)}\mid
Y\in\CP_{\l(i)}\}$ of pairwise orthogonal idempotents, each of which
generates $H_i$ as an ideal of itself (see \eqref{H_i}). This time, having chosen $D$ as a division ring, \emph{each $\ze_{Y(i)}$ is primitive}. To every element $i\in
I$ we associate an idempotent $\zu_i\in D_{I}$ and a right $D_{I}$-module $U_i$ with the following
rules: if $i\in I\setminus \max{I}$, we \emph{choose} $\zu_i\in
\BE_i$, while if $i\in\max{I}$, then we set $\zu_i \defug
\ze_{X_i}$. Next, set
\[
U_{i} \defug\left(\left.\zu_{i}D_{I} +
H_{I_{\l(i)-1}}\right)\right/H_{I_{\l(i)-1}}.
\]

\begin{pro}\label{simples}
For every $i\in I$ the right $D_{I}$-module $U_i$ is simple and
\begin{equation}\label{rrUi}
r_{D_{I}}(U_{i}) = \left\{
                   \begin{array}{l}
                     H_{I\setminus\{i\le\}}+\left(\mathbf{1}-\ze_{X_{\{i\le\}}}\right)D,\,  \hbox{if $\{i\le\}$ is finitely sheltered in $I$;} \\
                     H_{I\setminus\{i\le\}},\,  \hbox{if $\{i\le\}$ is not finitely sheltered in $I$.}
                   \end{array}
                 \right.
\end{equation}
If $I$ has a finite cofinal subset, then $r_{D_{I}}(U_{i}) = H_{I\setminus\{i\le\}}$ for every $i\in I$.
\end{pro}
\begin{proof}
Firstly, it follows from \eqref{H_i} that
\begin{equation}\label{Ui}
    U_i = \left(\left.\zu_{i}H_i +
H_{I_{\l(i)-1}}\right)\right/H_{I_{\l(i)-1}} = U_iH_i.
\end{equation}
Suppose that $0\ne x\in U_i$. Then $x = \zu_i\za+H_{I_{\l(i)-1}}$ for some nonzero $\za\in H_{i}$. Inasmuch as $\zu_{i}H_i$ is a minimal
right ideal of the ring $H_i$, then $\zu_{i}\za H_i = \zu_{i}H_i$
and consequently $xD_{I} =
U_i$ by \eqref{Ui}, proving that $U_i$ is a simple $D_{I}$-module.
Next, taking into account that $\{i\le\}$ is an upper subset of $I$, if we specialize Proposition \ref{uppersbs} by setting $J = I$ and $K = \{i\le\}$, we see that the second member of the equality \eqref{rrUi} is
precisely the kernel of the ring epimorphism
$\map{\f = \f_{K,J}}{D_{I}}{D_{I,\{i\le\}}} = H_{\{i\le\}}+\ze_{X_{\{i\le\}}}D$ defined by the rule \eqref{mapfKJ}. By \eqref{Ui} and Theorem
\ref{lemposetring} we have that $U_iH_j = 0$ when $j\in
I\setminus\{i\le\}$, therefore $H_{I\setminus\{i\le\}}\sbs
r_{D_{I}}(U_{i})$. If $\{i\le\}$ is finitely sheltered in $I$,
then $\ze_{X_{\{i\le\}}}$ is the multiplicative identity of the ring
$H_{\{i\le\}}$ according to Proposition \ref{multunit}; in
particular $H_i = H_i\ze_{X_{\{i\le\}}}$ and so
$\left(\mathbf{1}-\ze_{X_{\{i\le\}}}\right)D \sbs r_{D_{I}}(U_{i})$ by
\eqref{Ui}. This shows that $\Ker\left(\f\right) \sbs r_{D_{I}}(U_{i})$
and hence $U_i$ is canonically a simple right $D_{I,\{i\le\}}$-module.
Now it follows from Proposition \ref{uppersbss} that $D_{I,\{i\le\}}$
has essential socle given by $H_i$; since this latter is homogeneous
and regular, we infer that the ring $D_{I,\{i\le\}}$ is primitive.
Accordingly, since $U_i = U_iH_i$ we conclude that $U_i$ is faithful
as a simple right $D_{I,\{i\le\}}$-module and this establishes the
equality \eqref{rrUi}.

The last statement follows directly from the
equality \eqref{rrUi}, because if $I$ has a finite cofinal subset, then every subset of $I$ is finitely sheltered in $I$ and it follows from Proposition \ref{multunit} that $\mathbf{1}-\ze_{X_{\{i\le\}}} = \ze_{X_{I\setminus\{i\le\}}} \in H_{I\setminus\{i\le\}}$.
\end{proof}

We are now in a position to analyze the main features of the regular
ring $D_{I}$, the first of which is that it is semiartinian. Observe that $I^{\bullet\bullet}_\xi = \emptyset$ is finitely sheltered in $I$; thus we may consider the
ordinal $\xi_0$ defined by
\[
\xi_0
\defug\min\left\{\a\left|I^{\bullet\bullet}_\a\text{ is
finitely sheltered in $I$}\right\}\right. \le \xi.
\]
As we shall see, $\xi_0$ will be critical when determining the Loewy chain of
$D_{I}$. If $\xi_0 < \xi$ and $\max{(I^{\bullet\bullet}_{\xi_0})} =
\{k_1,\ldots,k_n\}$, we shall consider the idempotent
\[
\zf \defug\ze_{X_{k_1}}+\cdots+\ze_{X_{k_n}}.
\]
Remember that $\zf$ is the multiplicative identity of $H_{I^{\bullet\bullet}_{\xi_0}} =
D_{I,I^{\bullet\bullet}_{\xi_0}}$ by Proposition
\ref{multunit}.

If $\xi_0 <\xi$, then $I^{\bullet\bullet}_{\xi_0}$ has a finite cofinal subset and therefore $\xi$ must be a successor ordinal. In particular, it is clear that $\xi_0 = 0$ if and only if $I$ has a finite cofinal subset, in which the case $\zf = \mathbf{1}$ and $D_{I} = H_I$ by Proposition \ref{multunit}.

We shall need a couple of lemmas, the second of which is a direct consequence of \cite[Proposition 3]{Bac:7}.

\begin{lem}\label{lemhereditary}
Let $R$ be a ring with projective and essential right socle $L$ and
assume that $R$ has a subring $S$ such that $R = S+L$ and $R$ is left $S$-flat. If $S$ is right hereditary, then $R$ is right hereditary as
well.
\end{lem}
\begin{proof}
In order to prove that $R$ is right hereditary it is sufficient to
show that if $E$ is an injective right $R$-module with an essential
submodule $M$, then $E/M$ is $R$-injective. Firstly, by the flatness of $_SR$, the injectivity of $E_R$ implies that of $E_S$. Now it is readily seen that the
canonical right $R/L$-module structure on $E/M$, arising from the
fact that $EL\sbs M$, and the original structure of a factor
$R$-module restrict to the same $S$-module structure. As $S$ is
right hereditary, it follows that $E/M$ is $S$-injective and hence
$R/L$-injective. Finally, since $L$ is left pure in $R$, we conclude
that $E/M$ is $R$-injective.
\end{proof}

\begin{lem}\label{pro simplinject}
Let $R$ be a ring with a faithful simple and projective right $R$-module $S$ and let $Q = \BIE(S_R)$, so that $R$ can be identified with a dense subring of $Q$ and $\Soc(R) = R\cap\Soc(Q)$ is the trace of $S$ in $R$. Then $S_R$ is injective if and only if $\Soc(R) = \Soc(Q)$.
\end{lem}

We recall that a ring $R$ is \emph{unit-regular} if for every $x\in R$ there exists a unit $u\in R$ such that $x = xux$. It is well known that a regular ring $R$ is unit regular if and only if, given three finitely generated projective right $R$-modules $A,B,C$, the condition $A\oplus C\is B\oplus C$ implies $A\is B$. Another equivalent condition is that $R$ has stable range 1, meaning that if $a,b\in R$ and $aR+bR = R$, then there is some $c\in R$ such that $a+bc$ is a unit (see \cite[Chapter 4]{Good:3}).

\begin{theo}\label{theo semiartunitreg}
    With the above settings and notations, the ring $D_{I}$ satisfies the
    following properties:
    \begin{enumerate}
    \item For every ordinal $\a\le\xi$
\begin{equation}\label{loewy}
\Soc_\a(D_{I}) = \begin{cases}H_{I_\a},\,\,\text{ if $I$ has a finite cofinal subset
 or }\a\le\xi_0 \\
H_{I_\a}\oplus(\mathbf{1}-\zf)D,\,\,\text{ if }0<\xi_0<\a.
                            \end{cases}
\end{equation}
Thus the ring $D_{I}$ is semiartinian and its Loewy length is $\xi$
(resp. $\xi+1$) if $\xi_0<\xi$ (resp. $\xi_0 = \xi$).
\item If $i,j\in I$, then $U_i\preccurlyeq U_j$ if and only if $i\le j$ and we have
\begin{equation}\label{SimpR}
\Simp_{D_{I}} = \begin{cases} \{U_i\mid i\in I\}, \text{ if $I$ has a finite cofinal subset,} \\
\left\{U_i\mid i\in I\right\}\cup\{D_{I}/H_I\},\text{ otherwise.}
\end{cases}
\end{equation}
Thus $I$ and $\Simp_{D_{I}}$ are isomorphic posets if $I$ has a finite cofinal subset, otherwise the additional simple module
$D_{I}/H_I$ is a maximal element of $\Simp_{D_{I}}$ such that
\[
h(D_{I}/H_I) = \xi_0+1
\]
and, for
every $i\in I$,
\begin{equation}\label{SimpR1}
U_{i}\prec D_{I}/H_I \text{ if and only if
$\{i\le\}$ is not finitely sheltered in $I$}.
\end{equation}
\item $D_{I}$ is unit regular.
\item If $U\in\Simp_{D_{I}}$, then $U_{D_{I}}$ is injective if and only if $U$ is either a maximal element, or $U = U_i$ for some $i\in I'$. Consequently $D_{I}$ is a right $V$-ring if and only if $I' = I$. Moreover, $D_{I}$
is a right and left $V$-ring if and only if $\xi = 1$, if and only if all
primitive factor rings of $D_{I}$ are artinian.
\item If $\xi$ is a
natural number, in particular if $I$ is finite, then $D_{I}$ is (right
and left) hereditary.
\item If $I' = \emptyset$ and $I$ is at most countable, then the dimension of $D_{I}$ as a right and a left vector space over $D$ is countable.
\item $D_{I}$ is well behaved (Proposition and Definition \ref{pro lambda=h}) if and only if, for every $\a<\xi_0$, there is some $i\in I^{\bullet}_{\a+1}$ such that $\{i\le\}$ is not finitely sheltered.
\item $D_{I}$ is very well behaved (Definition \ref{def vwbehaved}) if and only if $I$, or equivalently $\Simp_{D_{I}}$, has finitely many maximal elements.
\end{enumerate}
\end{theo}

\begin{proof}
(1) Obviously \eqref{loewy} holds if $\a = 0$, while if $\a = 1$, then \eqref{loewy} follows directly from Lemma \ref{socRJ}. Suppose that $\a>1$, assume that either $I$ has a finite cofinal subset, or $\a\le\xi_0$, and assume inductively that $\Soc_\b(D_{I}) = H_{I_{\b}}$ for every ordinal $\b<\a$. If $\a$ is a limit ordinal, since $I_{\a} \defug \bigcup_{\b<\a}I_{\b}$, then we have
\[
\Soc_{\a}(D_{I}) = \bigcup_{\b<\a}\Soc_{\b}(D_{I}) = \bigcup_{\b<\a}H_{I_{\b}} = H_{I_{\a}}.
\]
Suppose that $\a = \b+1$ for some $\b$ and set $J = I_{\b}^{\bullet\bullet}$. Since $J$ is an upper subset of $I$, we can consider the surjective ring homomorphism $\map{\f_{J,I}}{D_{I}}{D_{I,J}}$ as in Proposition \ref{uppersbs}. If $\a\le\xi_0$, then $J$
is not finitely sheltered in $I$ and it follows from Proposition \ref{uppersbs} that
\begin{equation}\label{eq-KerfJI}
\Ker(\f_{J,I}) = H_{I\setminus J} = H_{I_{\b}} = \Soc_{\b}(D_{I}).
\end{equation}
If $I$ has a finite cofinal subset, then $J$ is finitely sheltered in $I$ and again Proposition \ref{uppersbs} tells us that
\[
\Ker(\f_{J,I}) =
H_{I\setminus J}+(\mathbf{1}-\ze_{X_J})D = H_{I_{\b}}+(\mathbf{1}-\ze_{X_J})D ;
\]
moreover we have that $\mathbf{1}-\ze_{X_J}\in H_{I_{\b}}$,
therefore \eqref{eq-KerfJI} again holds. Thus, in both
cases $\f_{J,I}$ induces an isomorphism $D_{I}/\Soc_{\b}(D_{I})\is
D_{I,J}$, which in turn restricts to the
canonical isomorphism $(H_{I_{\b+1}^{\bullet}}+\Soc_{\b}(D_{I}))/\Soc_{\b}(D_{I}) \is
H_{I_{\b+1}^{\bullet}}$. As a result we obtain
\[
H_{I_{\b+1}}/H_{I_{\b}} =
\left.\left(H_{I_{\b+1}^{\bullet}}+H_{I_{\b}}\right)\right/H_{I_{\b}} \is
H_{I_{\b+1}^{\bullet}} =
H_{\left(I_{\b}^{\bullet\bullet}\right)_{1}} =
\Soc(D_{I,J}),
\]
where the last equality comes from Lemma \ref{socRJ}. This shows that
\[
H_{I_{\b+1}}/H_{I_{\b}} = \Soc(D_{I}/\Soc_{\b}(D_{I}))
\]
and therefore
$\Soc_\a(D_{I}) = H_{I_{\a}}$.

Next, let us consider the case in which $0<\xi_0<\a$. Inasmuch as $\xi_0\le\b$, then $\max{J}$ is finite and so we may consider the (orthogonal) idempotents $\zg = \sum\left\{\ze_{X_k}\left|k\in \max{J}
\right\}\right.$ and $\zh = \sum\left\{\ze_{X_k}\left|k\in
\max{(I_{\xi_{0}}^{\bullet\bullet})}
\setminus\max{J}\right\}\right.$. Moreover $\zg\in
H_{J} = R_{J}$ by
Proposition \ref{multunit} and $\zh\in H_{I_{\b}}$, because
$\max{(I_{\xi_{0}}^{\bullet\bullet})}
\setminus\max{J} \sbs I_{\b}$. As a result, since $I$ is not finitely sheltered in $I$, by using again Proposition
\ref{multunit} and noting that $\zf = \zg+\zh$ we see that
\begin{equation}\label{eq loew}
D_{I} = H_I\oplus\mathbf{1}D = H_I\oplus(\mathbf{1}-\zf)D =
H_{I_{\b}}\oplus D_{I,J}\oplus(\mathbf{1}-\zg-\zh)D
\end{equation}
and hence $D_{I}/H_{I_{\b}}\is D_{I,J}\oplus
(\mathbf{1}-\zg)D$. Now, since $\zg$ is the multiplicative identity of $H_{J}$, it is immediately checked that $(\mathbf{1}-\zg)+H_{I_{\b}}$ is a central idempotent of $D_{I}/H_{I_{\b}}$. Consequently
\[
D_{I}/H_{I_{\b}}\is
D_{I,J}\times D
\]
as rings. We are then in a position to compute $\Soc_{\xi_{0}+1}(D_{I})$, by putting $\b = \xi_{0}$ in the above. Since $\Soc_{\xi_{0}}(D_{I}) = H_{I_{\xi_{0}}}$, as it follows from the first part of the proof, then we have that
\[
D_{I}/\Soc_{\xi_{0}}(D_{I}) = D_{I}/H_{I_{\xi_{0}}}\is
D_{I,I_{\xi_{0}}^{\bullet\bullet}}\times D
\]
as rings. Inasmuch as $\Soc\left(D_{I,I_{\xi_{0}}^{\bullet\bullet}}\right) = H_{I_{\xi_{0}+1}^{\bullet}}$ by Lemma \ref{socRJ} and $\zh = \mathbf{0}$ when $\b = \xi_{0}$, it follows from \eqref{eq loew} that
\[
\Soc_{\xi_{0}+1}(D_{I}) = H_{I_{\xi_{0}}}\oplus H_{I_{\xi_{0}+1}^{\bullet}}\oplus(\mathbf{1}-\zf)D = H_{I_{\xi_{0}+1}}\oplus(\mathbf{1}-\zf)D.
\]
Now, assume that $\a > \xi_0+1$ and suppose, inductively, that
\begin{equation}\label{eq loew1}
    \Soc_{\b}(D_{I}) = H_{I_{\b}}\oplus(\mathbf{1}-\zf)D
\end{equation}
whenever $\xi_0<\b<\a$. If $\a = \b+1$ for some $\b>\xi_0$, then it follows from \eqref{eq loew} and \eqref{eq loew1} that $D_{I}/\Soc_{\b}(D_{I}) \is
D_{I,J}$. Since $\Soc\left(D_{I,J}\right) = H_{I_{\b+1}^{\bullet}}$ by Lemma \ref{socRJ}, we infer that
\[
\Soc_{\a}(D_{I}) = H_{I_{\b}}\oplus H_{I_{\b+1}^{\bullet}}\oplus(\mathbf{1}-\zf)D = H_{I_{\a}}\oplus(\mathbf{1}-\zf)D,
\]
as wanted. Finally, if $\a$ is a limit ordinal, then we have that $H_{I_{\a}} = \bigcup_{\b<\a}H_{I_{\b}}$ and hence
\[
H_{I_{\a}}\cap(\mathbf{1}-\zf)D =
\left(\bigcup_{\b<\a}H_{I_{\b}}\right)\cap(\mathbf{1}-\zf)D =
\bigcup_{\b<\a}\left(H_{I_{\b}}\cap(\mathbf{1}-\zf)D\right) = 0.
\]
It follows that
\begin{gather*}
\Soc_{\a}(D_{I}) = \bigcup_{\b<\a}\Soc_{\b}(D_{I}) = \bigcup_{\b<\a}(H_{I_{\b}}\oplus(\mathbf{1}-\zf)D) = \left(\bigcup_{\b<\a}H_{I_{\b}}\right)\oplus(\mathbf{1}-\zf)D \\
 = H_{I_{\a}}\oplus(\mathbf{1}-\zf)D
\end{gather*}
and we are done.

As far as the Loewy length of $D_{I}$ is concerned, if $I$ has
a finite cofinal subset, that is $\xi_{0} = 0$, then $D_{I} = H_I =
H_{I_{\xi}}$ and so it follows from \eqref{loewy} that $D_{I}$ has Loewy length $\xi$. If $0<\xi_0<\xi$, then by \eqref{loewy} we have that
\[
\Soc_{\xi}(D_{I}) = H_{I_{\xi}}\oplus(\mathbf{1}-\zf)D =
H_{I}\oplus\mathbf{1}D = D_{I}
\]
and therefore $D_{I}$ has again Loewy length $\xi$. If $\xi_{0} =
\xi$, then \eqref{loewy} imply that
$D_{I}/\Soc_{\xi}(D_{I}) = D_{I}/H_{I_{\xi}} = D_{I}/H_{I} \is D$ and $D_{I}$ has Loewy
length $\xi+1$.

(2) Let $i,j\in I$ and assume that $i\le j$. Then $\{i\le\}\sps \{j\le\}$, that is
$I\setminus\{i\le\}\sbs I\setminus\{j\le\}$ and therefore $H_{I\setminus\{i\le\}}\sbs H_{I\setminus\{j\le\}}$. As a result, if $\{i\le\}$ is not finitely sheltered in $I$, then it follows from Proposition \ref{simples} that $r_{D_{I}}(U_{i}) \sbs r_{D_{I}}(U_{j})$.
Suppose that $\{i\le\}$, and hence $\{j\le\}$, is finitely sheltered in $I$ and set $\max{\{i\le\}} = \{m_{1},\ldots,m_{r},m_{r+1},\ldots,m_{s}\}$, where $\{m_{r+1},\ldots,m_{s}\} = \max{\{j\le\}}$. By Proposition \ref{pro max}, $\{X_{m_{1}},\ldots,X_{m_{r}},X_{m_{r+1}},\ldots,X_{m_{s}}\}$ and $\{X_{m_{r+1}},\ldots,X_{m_{s}},\}$ are partitions of $X_{\{i\le\}}$ and $X_{\{j\le\}}$ respectively, therefore we can consider the corresponding pairwise orthogonal idempotents $\ze_{X_{m_{1}}},\ldots,\ze_{X_{m_{r}}},$ $\ze_{X_{m_{r+1}}},\ldots,\ze_{X_{m_{s}}}$. We observe that $m_{1},\ldots,m_{r}\in I\setminus\{j\le\}$, therefore $\ze_{X_{m_{1}}},\ldots,\ze_{X_{m_{r}}}$ $\in H_{I\setminus\{j\le\}}$. Consequently, by taking again Proposition \ref{simples} into account, we see that
\begin{gather*}
\mathbf{1}-\ze_{X_{\{i\le\}}} = \mathbf{1}-\ze_{X_{m_{1}}}-\cdots-\ze_{X_{m_{r}}}-\ze_{X_{m_{r+1}}}-\cdots-\ze_{X_{m_{s}}} \\
= -\ze_{X_{m_{1}}}-\cdots-\ze_{X_{m_{r}}}+ \mathbf{1}-\ze_{X_{\{j\le\}}}\in r_{D_{I}}(U_{j}).
\end{gather*}
Thus we have again that $r_{D_{I}}(U_{i}) \sbs r_{D_{I}}(U_{j})$, proving that $i\le j$ implies $U_i\preccurlyeq U_j$. At this point, in order to show that the reverse implication holds, it is sufficient to prove that if $U_i\preccurlyeq U_j$, then $i$ and $j$ are comparable. However, if $r_{D_{I}}(U_{i}) \sbs r_{D_{I}}(U_{j})$, then necessarily $U_iH_j\ne 0$, otherwise we would get $U_j = U_jH_j = 0$. Since $U_i = U_iH_i$, then $H_iH_j\ne 0$ and so $i$ and $j$ are comparable by Theorem \ref{lemposetring}.

Let $V$ be any simple right $D_{I}$-module. If $VH_I = 0$, then $H_I\ne D_{I}$ and it is clear that $V\is D_{I}/H_I$. Otherwise $VH_I = V$ and we may consider the smallest ordinal $\a$ such that $VH_i\ne 0$ for some $i\in I^{\bullet}_{\a+1}$. According to Lemma \ref{HsbsFr} we have that $H_i = H_i\mathbf{u}_iH_i$, thus
\[
0\ne VH_i = VH_i\mathbf{u}_iH_i \sbs V\mathbf{u}_iH_i.
\]
Let $x\in V$ be such that $x\mathbf{u}_i\ne 0$. If $\mathbf{a}\in D_{I}$ and $\mathbf{u}_i\mathbf{a}\in H_{I_\a}$, then $x\mathbf{u}_i\mathbf{a} = 0$ by the choice of $\a$. Thus the assignment $\mathbf{u}_i\mathbf{a}+H_{I_\a}\mapsto x\mathbf{u}_i\mathbf{a}$ defines a nonzero $D_I$-linear map from $U_i$ to $V$ and consequently $V\is U_i$.

Finally, if $i\in I$ and $\{i\le\}$ is not finitely sheltered in $I$, it
follows from Proposition \ref{simples} that $U_{i}\preccurlyeq D_{I}/H_{I}$.
If, on the contrary, $\{i\le\}$ is finitely sheltered in $I$, then
$\ze_{X_{\{i\le\}}}\in H_{\{i\le\}}$ by Proposition \ref{multunit}. As a
result, by using again Proposition \ref{simples} we see that
\[
\mathbf{1} =
\left(\mathbf{1}-\ze_{X_{\{i\le\}}}\right)+\ze_{X_{\{i\le\}}} \in
r_{D_{I}}(U_{i})+H_I
\]
and hence $r_{D_{I}}(U_{i})+H_I = D_{I}$, proving that
$r_{D_{I}}(U_{i})\nsbs H_I$, namely $U_{i}\not\preccurlyeq
D_{I}/H_I$.

(3) Let us prove first that the ring $D_{I,J} = H_J+\ze_{X_J}\,D$ is unit-regular for every
\underline{finite} subset $J\sbs I$. It is obvious that
$R_{\emptyset} = 0$ is unit-regular; let $n$ be any positive
integer, assume inductively that $D_{I,J}$ is unit-regular whenever
$|J|<n$ and take $J\sbs I$ such that $|J| = n$. Let us consider the
surjective ring homomorphism
\[
\lmap{\f_{J\setminus J_1,J}}{D_{I,J}}{D_{I,J\setminus J_1}}
\]
(see Proposition \ref{uppersbs}) and note that $D_{I,J\setminus J_1}$ is unit-regular by the inductive
hypo\-thesis. In order to prove unit-regularity of $D_{I,J}$, according to
Vasershtein criterion (see \cite[Proposition 4.12]{Vasershtein:1},
or \cite[Lemma 3.5]{Bac:12} for a ready-to-use version) it will be
sufficient to prove that every unit of $D_{I,J\setminus J_1}$ has the
form $\f_{J\setminus J_1,J}(\za)$ for some unit $\za$ of $D_{I,J}$ and
$\zu D_{I,J}\zu$ is unit-regular for every idempotent
$\zu\in\Ker(\f_{J\setminus J_1,J})$. By denoting with $K$ the set of those elements of $J$ which are isolated in $J$, i. e. are minimal and maximal in $J$, we have from Proposition \ref{pro max} that $K = \{j\in J\mid X_j\cap
X_{J\setminus\{j\}} = \emptyset\}$; since $K\sbs J_1$, we infer that
\[
X_{J\setminus J_1}\cap X_K = \emptyset \quad\text{and}\quad
X_{J\setminus J_1}\cup X_K = X_J.
\]
Let us write
\[
\ze = \ze_{X_J},\quad \ze' = \ze_{X_{J\setminus J_1}}
\quad\text{and}\quad \ze'' = \ze_{X_K},
\]
so that $\ze'$, $\ze''$ are orthogonal idempotents and $\ze'+\ze'' =
\ze$, and suppose that $\zb$ is a unit of $R_{J\setminus J_1}$. If
$D_{I,J\setminus J_1} = H_{J\setminus J_1}$, then it follows from
Proposition \ref{multunit} that $\ze'$ is the multiplicative identity of $H_{J\setminus J_1}$ and
$J\setminus J_1$ is finitely sheltered in $I$; consequently $\ze'' =
\ze-\ze'\in \Ker(\f_{J\setminus J_1,J})$ by \eqref{kermapfKJ}. If
$\zb'$ is the inverse of $\zb$ in $D_{I,J\setminus J_1}$, using the
fact that $\zb$ and $\zb'$ belong to $\ze'D_{I,J}\ze'$ it is immediately seen
that $\zb'+\ze''$ is an inverse for $\zb+\ze''$ in $D_{I,J}$ and
$\f_{J\setminus J_1,J}(\zb+\ze'') = \zb$. Assume that $D_{I,J\setminus J_1} \ne H_{J\setminus J_1}$. Then $D_{I,J\setminus J_1} =
H_{J\setminus J_1}\oplus \ze' D$ by Proposition \ref{multunit} and so
$\zb = \zc+\ze'd$ for unique $\zc\in H_{J\setminus J_1}$ and $d\in
D$. Necessarily $d\ne 0$ and if $\zc'+\ze'd'$ is the inverse of
$\zb$ in $D_{I,J\setminus J_1}$, where $\zc'\in H_{J\setminus J_1}$
and $d'\in D$, then $d' = d^{-1}$. Noting that $\ze''H_{J\setminus J_1} = 0 = H_{J\setminus J_1}\ze''$, we infer that
\begin{gather*}
(\zc+\ze d)(\zc'+\ze d') = (\zc+\ze'd+\ze''d)(\zc'+\ze'd'+\ze''d')
= \\
(\zc+\ze'd)(\zc'+\ze'd') + \ze''dd' = \ze'+\ze'' = \ze.
\end{gather*}
Similarly $(\zc'+\ze d')(\zc+\ze d) = \ze$, hence $\zc+\ze d$ is
a unit in $D_{I,J}$ and $\f_{J\setminus J_1,J}(\zc+\ze d) =
\zc+\ze'd = \zb$.

Next, let $\zu$ be an idempotent of $\Ker(\f_{J\setminus J_1,J})$.
If $J\setminus J_1$ is not finitely sheltered in $I$, then it follows from
\eqref{kermapfKJ} and Lemma \ref{socRJ} that $\Ker(\f_{J\setminus J_1,J}) = H_{J_1} =
\Soc(D_{I,J})$; thus $\zu D_{I,J}\zu$ is a semisimple ring and so it is
unit-regular. If, on the contrary, $J\setminus J_1$ is finitely sheltered in $I$, then
\eqref{kermapfKJ} tells us that $\Ker(\f_{J\setminus J_1,J}) =
H_{J_1}+\ze''D$ and hence, observing that $\ze''H_{J_1\setminus
K}\ze'' = 0$, we get
\[
\Ker(\f_{J\setminus J_1,J}) = H_{J_1\setminus K}
\oplus(H_{K}+\ze''D) = H_{J_1\setminus K} \oplus D_{I,K}.
\]
As a consequence $\zu = \zu'+\zu''$ for unique orthogonal
idempotents $\zu'\in H_{J_1\setminus K}$ and $\zu''\in D_{I,K}$.
Inasmuch as $H_{J_1\setminus K}$ is semisimple and $D_{I,K}$ is
unit-regular by the inductive hypothesis, we conclude that $\zu
D_{I,J}\zu = \zu' H_{J_1\setminus K}\zu'\oplus \zu'' D_{I,K}\zu''$ is
unit-regular and we are done.

Finally, given any $\za\in D_{I}$, there is a finite subset $J\sbs I$ and two elements
$\zb\in H_J$, $d\in D$ such that $\za = \zb+\mathbf{1}d = \zb+
\ze_{X_J}d+(\mathbf{1}-\ze_{X_J})d$. Thus $\za$ belongs to the
\underline{unital} subring $S = D_{I,J}+(\mathbf{1}-\ze_{X_J})D$ of
$D_{I}$, in which $\ze_{X_J}$ is a central idempotent. If $X_J = X$, then $\ze_{X_J} = \mathbf{1}$ and $S = D_{I,J}$. If $X_J \ne X$, then $S\is D_{I,J}\times D$ as rings. In both cases $S$ is unit-regular and hence $\za = \za\zb\za$ for a unit $\zb\in S\sbs D_{I}$, as wanted.

(4) If $U$ is a maximal element of $\Simp_{D_{I}}$, that is
$r_{D_{I}}(U)$ is a maximal right ideal, then $U_{D_{I}}$ is injective
by \cite[Corollary 4.8]{Bac:15}. Otherwise, according to
\eqref{SimpR}, $U = U_{i}$ for some non-maximal element $i\in I$.
Set $J = \{i\le\}$ and note that $D_{I,J}\is D_{I}/r_{D_{I}}(U)$ is a
primitive ring which has $U_{i}$ as the unique (up to an isomorphism)
faithful simple right module. Let us consider the ring $S_i$ introduced immediately before Proposition \ref{subrings} and the $D$-linear map
$\map{\theta}{D_{I,J}}{S_{i}}$ defined by $\theta(\za) =
\ze_{X_i}\za\ze_{X_i}$. Note that if  $j>i$ and $\za\in H_{j}$, then
both $\ze_{X_i}\za$ and $\za\ze_{X_i}$ belong to $H_{i}$ by Theorem
\ref{lemposetring}, (3) and therefore, since $H_{i} = \ze_{X_i}H_{i}\ze_{X_i}$, for every $\za,\zb\in H_{J}$ we have that
\[
\ze_{X_i}(\za\zb)\ze_{X_i} = \ze_{X_i}\za\ze_{X_i}\zb\ze_{X_i}.
\]
From this we infer immediately that $\theta$ is a unital ring
homomorphism. Observe that $\theta$ restricts to the identity on
$H_{i} \sbs \psi_{i}(F_{\l(i)}) = \Soc(S_{i})$; moreover $H_{i} = \Soc(D_{I,J})$ is essential as
a right ideal of $D_{I,J}$ (see Lemma \ref{socRJ}), therefore $\theta$ is a monomorphism. As a
result $D_{I,J}$ can be identified with a subring of $S_{i}$ and $\Soc(D_{I,J}) = H_{i} = D_{I,J}\cap \Soc(S_{i})$. Since in turn $S_{i}\is Q$, it follows from Lemma
\ref{pro simplinject} that $U = U_{i}$ is $D_{I,J}$-injective if and only if $\Soc(D_{I,J}) =  \Soc(S_{i})$, if and only if $H_{i} = \psi_{i}(F_{\l(i)})$, if and only if $i\in I'$.
Inasmuch as $D_{I}$ is regular, then $r_{D_{I}}(U)$ is pure as a left
ideal and so $U_{D_{I}}$ is injective.
Finally, according to \cite[Theorem 2.7]{Bac:12} $D_{I}$ is a right and left
$V$-ring if and only if all primitive factor rings of $D_{I}$ are
artinian, that is, if and only if all primitive ideals of $D_{I}$ are
maximal as right ideals; in view of property (2) this happens if and
only if $I$ is an antichain, that is $\xi = 1$.

(5) Suppose that $\xi$ is a natural number. We will prove that $D_{I,J}$ is hereditary for every subset $J\sbs I$ by applying induction on the has dual classical Krull length $\xi(J)$ of $J$. It will follows that, in particular, $D_{I} = D_{I,I}$ is hereditary.
If $\xi(J) = 0$, that is $J = \emptyset$, then and $D_{I,\emptyset} = 0$ is trivially hereditary. Given a positive integer $n\le\xi$, suppose that $D_{I,J}$ is hereditary whenever $\xi(J)<n$ and
let $J$ be any subset of $I$ such that $\xi(J) = n$. Since $\xi(J_{1}^{\bullet\bullet}) = n-1$, then $D_{I,J_{1}^{\bullet\bullet}}$ is hereditary by the inductive hypothesis. Assume that $J$ is finitely sheltered in $I$ and let $K$ be the set of all isolated elements of $I$ which belong to $J$. Then we have
\[
D_{I,J} = H_J = H_{J_1}\oplus H_{J_{1}^{\bullet\bullet}} = H_K\oplus H_{J_1\setminus K}\oplus D_{I,J_{1}^{\bullet\bullet}} =
H_{K}\oplus D_{I,J\setminus K},
\]
taking Proposition \ref{multunit} into account. Since $K$ is finite, then $H_{K} \is D^K$ is a semisimple ring and hence is hereditary; thus, in order to prove that $D_{I,J}$ is hereditary we may assume that $K = \emptyset$. Consequently $X_J = X_I$ and therefore $\ze_{X_{J_{1}^{\bullet\bullet}}} = \ze_{X_J}$, so that $H_{J_{1}^{\bullet\bullet}} = D_{I,J_{1}^{\bullet\bullet}}$ is a \underline{unitary} and subring of $D_{I}$. Since $D_{I,J_{1}^{\bullet\bullet}}$ is a regular ring and $H_{J_1} = \Soc(D_{I,J})$, it follows from Lemma \ref{lemhereditary} that $D_{I,J}$ is right and left hereditary.

Assume now that $J$ is not finitely sheltered in $I$. Then it follows from Proposition \ref{multunit} that $D_{I,J} = H_J\oplus\ze_{X_J}D = H_{J_1}\oplus H_{J_{1}^{\bullet\bullet}}\oplus\ze_{X_J}D$ and, by using Proposition \ref{uppersbs}, we infer that $H_{J_{1}^{\bullet\bullet}}\oplus\ze_{X_J}D \is D_{I,J}/H_{J_1} \is D_{I,J_{1}^{\bullet\bullet}}$. Thus $H_{J_{1}^{\bullet\bullet}}\oplus\ze_{X_J}D$ is a regular and hereditary ring and Lemma \ref{lemhereditary} applies again, proving that $D_{I,J}$ is right
and left hereditary.

(6) If $I$ is at most countable, then (see Notations \ref{not main})  $|X| = \al = \al_{0}$ and therefore $\FM_{X}(D)$ has countable dimension over $D$. If $I' =\emptyset$ and $i\in I$, then $H_{i}\is D$ if $i$ is a maximal element of $I$, otherwise $H_{i}\is \FM_{X}(D)$. As a result $H_I = \bigoplus_{i\in I}H_{i}$ has countable dimension over $D$ and the same occurs for $D_{I} = H_I+\ze_{X}D$.

(7) If $i\in I$, then it follows from \eqref{loewy} and property (2) that $h(U_i) = \l(U_i)$. Thus we only have to check the behavior of the simple module $V = D_I/H_I$, in case $0<\xi_0$, that is $I$ has not a finite cofinal subset. Set $\a+1 = \l(V)$ and assume that $\a+1<h(V) = \xi_0+1$, namely $\a+1\le\xi_0$. Since $\xi_0\le\xi$, then there is some $i\in I$ such that $\l(i) = \a+1$. We have that $\l(U_i) = \a+1$, therefore $U_i$ and $V$ are not comparable and consequently $\{i\le\}$ is finitely sheltered by \eqref{SimpR1}. Suppose, on the contrary, that $\l(V) = \xi_0+1$. Given any $\a<\xi_0$, we have from Corollary \ref{lengthlayerr} that there is some $U\in\Simp_{D_I}$ such that $\l(U) = \a+1$ and $U\prec V$. Necessarily $U = U_i$ for a unique $i\in I$ with $\l(i) = \a+1$ and $\{i\le\}$ is not finitely sheltered by \eqref{SimpR1}.

(8) According to Proposition \ref{uppersimpp} we only have to show the ``if'' part. Firstly, it is clear from \eqref{SimpR} that $\Simp_{D_I}$ has finitely many maximal elements if and only if $I$ satisfies the same condition. Thus, assume that this condition holds, let $\ZS$ be an upper subset of $\Simp_{D_I}$, set $J = \{j\in I\mid U_j\in\ZS\}$ and note that $J$ is an upper subset of $I$ by property (2). For every $i\in I$, it follows from \eqref{rrUi} that $H_{I\setminus J}$ annihilates $U_i$ if and only if $i\in J$. If either $I$ has a finite cofinal subset or $D_I/H_I\in\ZS$, by property (2) this is enough to conclude that $\Phi(\Psi(\ZS)) = \ZS$. Assume that $I$ has not a finite cofinal subset and $D_I/H_I\nin\ZS$. Then it follows from property (2) that $J$ is order isomorphic to $\ZS$ and, since every element of $\Simp_{D_I}$ is bounded by a maximal element and $\ZS$ is an upper subset, then $J$ is a finitely sheltered upper subset of $I$. Let $m_{1},\ldots,m_{r}$ be the maximal elements of $J$, so that $\{X_{m_{1}},\ldots,X_{m_{r}}\}$ is a partition of $X_J$ by Proposition \ref{pro max}. If we set $H = H_{I\setminus J}+ (\mathbf{1}- \ze_{X_{m_{1}}}+\cdots+\ze_{X_{m_{r}}})D$, then we have from Propositions \ref{multunit} and \ref{uppersbs} that $H$ is an ideal of $D_I$ and $D_I/H \is D_{I,J} = H_J$. From this, with the help of Proposition \ref{simples} we infer that $H$ annihilates $U_i$ if and only if $i\in J$. It is not the case that $H\sbs H_I$ otherwise, since $\ze_{X_{m_{1}}},\ldots,\ze_{X_{m_{r}}}\in H_I$, it would follow that $\mathbf{1}\in H_I\psbs D_I$. Again, we can conclude that $\Phi(\Psi(\ZS)) = \ZS$, proving that $D_I$ is very well behaved.
\end{proof}

Concerning hereditariness, we are presently unable to exhibit an example of non-hereditary, regular and semiartinian ring; nonetheless we have the following easy result.

\begin{pro}\label{pro hered}
If $R$ is a right semiartinian ring with Loewy length at most 2 and projective right socle, then $R$ is right hereditary.
\end{pro}
\begin{proof}
If the Loewy length of $R$ is 1, then $R$ is semisimple and so is hereditary. Assume that $R$ has Loewy length 2. In order to prove that $R$ is right hereditary, it is sufficient to show that if $E$ is an injective right $R$-module with an essential submodule $M$, then $E/M$ is an injective $R$-module. Set $K = \Soc(R_R)$ and note that $EK\sbs M$, so that $E/M$ is canonically a right $R/K$-module. Since $R/K$ is a semisimple ring, then $E/M$ is $R/K$-injective and the left purity of $K$ implies the $R$-injectivity of $E/M$.
\end{proof}

A couple of final remarks are in order. First, on the basis of properties (7) and (8) of Theorem \ref{theo semiartunitreg}, a suitable choice of the artinian poset $I$ produces a semiartinian and regular ring $D_I$ such that $(\Simp_{D_I})_{\a}$ is finite for every $\a$, but $\Simp_{D_I}$ has infinitely many maximal elements, so that $D_I$ is well behaved but not very well behaved (see Propositions \ref{pro orderreg} and \ref{uppersimpp}). The second remark concerns the distribution of non-maximal, injective members of $\Simp_{R}$, where $R$ is a regular and semiartinian ring. On the basis of property (4) of the previous theorem one might wonder wether the subset of these modules is always a lower subset of $\Simp_R$. However this is not the case, as shown by the following example.

\begin{ex}\label{ex injectives}
There exists a semiartinian, hereditary and unit-regular ring $R$
such that $\Simp_R$ is a chain $\{U\prec V\prec W\}$, where $V$ and $W$ are injective but $U$ is not injective.
\end{ex}
\begin{proof}
Let $\al$ and $\beth$ be infinite cardinals with $\beth<\al$ and set $X =\al\bullet\beth$. With the notations of Section \ref{sect wochsbr}, let us consider the partition $\CP_{1} = \{X_{1,\l}\mid\l<\beth\}$ of $X$, where $X_{1,\l} = \{\al\bullet\l+\r\mid\r<\al\}$ for all $\l<\beth$ and note that $|X_{1,\l}| = \al$ and $|\CP_{1}| = \beth$. Given a division ring $D$, let us consider the ring $Q = \CFM_{X}(D)$ and let $T$ be the subset of $Q$ of all matrices whose rows have support of cardinality not exceeding $\beth$:
\[
T = \{\ZA\in Q\mid |\Supp(\za(x,-))|\le\beth\text{ for all }x\in X\}.
\]
Then $T$ is a (unital) subring of $Q$. Indeed, let $\ZA,\ZB\in T$, let $x\in X$ and set
\[
Y = \Supp(\za(x,-)), \quad Z = \Supp(\zb(x,-)), \quad U = \bigcup\{\Supp(\zb(z,-))\mid z\in Z\}.
\]
Then
\[
|\Supp((\za-\zb)(x,-))| = |\Supp(\za(x,-)-\zb(x,-))| \le |Y\cup Z|\le\beth,
\]
showing that $T$ is an additive subgroup of $Q$. Moreover, if $y\in X\setminus U$, then
\[
(\za\zb)(x,y) = \sum_{z\in X}\za(x,z)\zb(z,y) = \sum_{z\in Y}\za(x,z)\zb(z,y) = 0.
\]
It follows that $\za\zb\in T$, because $|U|\le\beth$. Set $H = \FR_{X}(D)\cap T$ and note that $H$ is a semisimple and regular ideal of $T$. With the notations of Theorem \ref{pro-chainflr}, we have that $F_1 = \f_1(\FR_{\beth}(D)\sbs T$. Finally, let us consider the ring
\[
R = H\oplus F_1\oplus \mathbf{1}_{Q}D.
\]
Now it is easy to check that $R$ is a regular and semiartinian ring, where $\Soc{R} = H$ is homogeneous, $\Soc_{2}(R) = H\oplus F_1$ and $\Soc_{2}(R)/\Soc(R) \is F_1$ is homogeneous and $R/\Soc_{2}(R) \is D$. A straightforward application of Lemma \ref{lemhereditary} and Vasershtein criterion shows that $R$ is hereditary and unit-regular. It is clear that $\Prim_{R} = \{\{0\}, \Soc{R},\Soc_1{R}\}$, thus $\Simp_{R}$ is a chain  $\{U\prec V\prec W\}$, where $\Soc{R}$ is the trace of $U$ in $R$, $\Soc_1(R)/\Soc{R}$ is the trace of $V$ in $R/\Soc(R)$ and $W\is D$ is injective because it is a maximal element. According to Lemma \ref{pro simplinject}, $V$ is injective but $U$ is not injective.
\end{proof}

\end{document}